\documentclass[11pt, reqno]{amsart}
\usepackage[utf8]{inputenc}
\usepackage{amsfonts}
\usepackage{hyperref}

\hypersetup{}
\usepackage{amsmath}
\usepackage{xcolor}
\usepackage{amsthm}
\usepackage{pdflscape}
\usepackage{pgfplots}
\usepackage{mathrsfs}
\usepackage{mathtools}
\usepackage[top=2.5cm, bottom=2.5cm, left=2.2cm, right=2.2cm]{geometry}
\usepackage{amssymb,bbm}
\usepackage{shuffle}
\usepackage{enumerate}
\usepackage{enumitem}
\usepackage[initials]{amsrefs}
\usepackage{stmaryrd}
\usepackage{appendix}
\usepackage{tikz}
\usepackage{tikz-cd}

\pgfplotsset{compat = 1.18}

\usepackage[nodayofweek]{datetime}

\usepackage{verbatim}

\usepackage{graphicx}
\usepackage{matlab-prettifier}

\numberwithin{equation}{section}
\numberwithin{figure}{section}

\theoremstyle{plain}
\newtheorem{theorem}{Theorem}[section]
\newtheorem{lemma}[theorem]{Lemma}
\newtheorem{proposition}[theorem]{Proposition}
\newtheorem{corollary}[theorem]{Corollary}

\newtheorem{definition}[theorem]{Definition}

\newtheorem{remark}[theorem]{Remark}

\newtheorem{assumption}{Assumption}

\newtheorem*{definition*}{Definition}
\newtheorem{assumptionPrime}{Assumption}

\newtheorem{openproblem}{Open Problem}

\allowdisplaybreaks
\newcommand{\X}{\mathbf{X}}
\newcommand{\x}{\mathbf{x}}
\newcommand{\eps}{\varepsilon}
\newcommand{\R}{\ensuremath{\mathbb{R}}}

\newcommand{\N}{\ensuremath{\mathbb{N}}}

\newcommand{\bbE}{\mathbb E}
\newcommand{\bbP}{\mathbb P}
\newcommand{\bbI}{\mathbb I}
\newcommand{\bbH}{\mathbb H}
\newcommand{\bbF}{\mathbb{F}}

\newcommand{\mP}{\mathcal{P}}
\newcommand{\mS}{\mathcal{S}}
\newcommand{\mX}{\mathcal{X}}
\newcommand{\mL}{\mathcal{L}}\newcommand{\mC}{\mathcal{C}}
\newcommand{\mH}{\mathcal{H}}
\newcommand{\mI}{\mathcal{I}}
\newcommand{\mF}{\mathcal{F}}

\newcommand{\bv}{\bar{v}}

\newcommand{\rt}{{0}}

 \newcommand{\TV}{{\rm TV}}

\newcommand{\entropy}{\mathcal{H}}
\newcommand{\kap}{\kappa}
\newcommand{\mue}{\mu^\eps}

\newcommand{\energy}{\mathbb{H}}

\newcommand{\potential}{g}

\newcommand{\mPS}{{\mathcal{M}_{\kappa, d}}}

\newcommand{\lyap}{\mathbb{H}}
\newcommand{\info}{\mathbb{I}}

\newcommand{\bbT}{\mathbb{T}}

\newcommand{\mM}{\mathcal{M}}

\newcommand{\CMVE}{$\kap$-MLFE}
\newcommand{\freeEnergy}{sparse free energy}

\newcommand{\mPSb}{\mathcal{Q}_{\kap, d}}
\newcommand{\FPE}{Cayley fixed point}
\newcommand{\FPES}{Cayley fixed points}

\newcommand{\newW}{Q}
\newcommand{\intPot}{W}
\newcommand{\Y}{\mathbf{Y}}
\newcommand{\pathMeasures}{\mathcal{P}_{\kap, d, T}}

\title[An H-theorem for the $\kappa$-MLFE]{An H-theorem for a conditional McKean-Vlasov process related to interacting diffusions on regular trees}

\author{Kevin Hu}
\address{Kevin Hu: Division of Applied Mathematics, Brown University.}
\email{\href{mailto:kevin_hu@brown.edu}{kevin\_hu@brown.edu}}

\author{Kavita Ramanan}
\address{Kavita Ramanan: Division of Applied Mathematics, Brown University.}
\email{\href{mailto:kavita_ramanan@brown.edu}{kavita\_ramanan@brown.edu}}

\date{\today}

\begin{document}

\begin{abstract}
 We study the long-time behavior of
  the $\kappa$-Markov local-field equation ($\kappa$-MLFE), which is a conditional McKean-Vlasov equation
  associated with locally interacting diffusions
on the  $\kappa$-regular tree, for $\kappa \geq 2$. Under suitable assumptions on the coefficients, we prove well-posedness of the $\kappa$-MLFE. We also
     establish an H-theorem by
     identifying an energy functional $\mathbb{H}_\kap$, which we refer to  as the sparse free energy,
     that decreases along the measure flow of the $\kappa$-MLFE, with the rate of decrease governed by a nonnegative functional
     that can be viewed as a 
     modified Fisher information. 
     Moreover, we show that the zeros of the latter functional 
     coincide with the set of
     stationary distributions of the $\kappa$-MLFE and are also 
      marginals of  splitting Gibbs measures
    on the $\kappa$-regular tree.  Furthermore, we show that for a
    natural class of  initial conditions, the corresponding measure flow of the $\kappa$-MLFE  
    converges as $t \rightarrow \infty$ to a stationary distribution,  
thus demonstrating that      $\mathbb{H}_\kap$ acts as a global Lyapunov function.
Under mild additional conditions, in the case $\kappa = 2$  we prove that the energy functional $\mathbb{H}_2$  arises naturally as the renormalized limit of certain relative entropies. We exploit this characterization
to prove a modified logarithmic Sobolev inequality and  establish an exponential rate of convergence of the 
$2$-MLFE measure flow to its unique stationary distribution.
\end{abstract}

\maketitle


\section{Introduction}

\subsection{Background and Motivation.}
\label{subsub:background}
Fix $\kap \in\N$ with $\kappa \geq 2$, and let $\bbT^1_\kap = \{0, 1, \ldots, \kap\}$ denote the set of vertices of  the root neighborhood of the $\kap$-regular tree rooted at $0$. We study the long-time behavior of the following system of stochastic differential equations (SDEs) indexed by the set  $\bbT^1_\kap$: 
\begin{equation}
    \begin{aligned}dX_0(t) &= - \bigg(\nabla U\big(X_0(t)\big) + \sum_{v = 1}^\kappa \nabla W\big(X_0(t) - X_v(t)\big)\bigg) dt + \sqrt{2}dB_0(t), \\
dX_v(t) &= -\gamma\big(t, X_v(t), X_0(t)\big) dt + \sqrt{2}dB_v(t), \quad v = 1, \ldots, \kap,\\
\mu_t &= \text{Law}\big(X_0(t), \ldots, X_\kappa(t)\big),\end{aligned}
\label{eq:intro:CMVE}
\end{equation}where $\{B_v\}_{v \in \{0, \ldots \kap\}}$ is a family of independent standard $d$-dimensional Brownian motions, $U, W: \R^d \rightarrow \R$ are smooth functions, and $\gamma$ is a map from $\R_+ \times \R^d \times \R^d$ to $\R^d$ 
takes the form of a conditional expectation: 
\begin{equation}
    \gamma(t, x, y):= \mathbb{E}\bigg[\nabla U\big(X_0(t)\big) + \sum_{v = 1}^\kappa \nabla \intPot\big(X_0(t) - X_v(t)\big) \,\bigg|\,X_0(t) = x,\, X_1(t) =y \bigg].
    \label{eq:intro:gamma}
\end{equation} 
Note that $\gamma(t, x, y) = J(\mu_t, x, y)$ for a suitable functional $J: \mP((\R^d)^{1 + \kap}) \times \R^d \times \R^d \rightarrow \R^d$, where ${\mathcal P}( (\R^d)^{1 + \kap})$ is the space of probability measures on $(\R^d)^{1 + \kap}$. Equation \eqref{eq:intro:CMVE} is therefore an example of a \textit{McKean-Vlasov equation}, which is a stochastic differential equation with a measure-dependent drift of the form
\begin{equation}
\begin{aligned}
    dZ(t) &= - F\big(t, \mu_t, Z(t) \big) dt + \sqrt{2}dB(t), \quad    \mu_t = \text{Law}\big(Z(t)\big),
    \label{eq:intro:Fdrift}
\end{aligned}
\end{equation}
where $B$ is a standard $m$-dimensional Brownian motion for some $m \in \N$, and $F$ is a suitable drift functional that maps $\R_+ \times {\mathcal P}(\R^m) \times \R^m$ to $\R^m$.
Solutions to such McKean-Vlasov equations are also referred to as nonlinear Markov processes because their associated Kolmogorov  forward equations take the form of nonlinear partial differential equations called \textit{nonlinear Fokker-Planck equations} (NFPE). McKean-Vlasov equations arise in the study of interacting particle systems and have applications in many fields, including physics, neuroscience, biology, and economics. 

A particularly important family of McKean-Vlasov equations is characterized 
by drifts that exhibit a   measure dependence of convolution form:  
\begin{equation}
\label{eq:intro:MV}
     dZ(t) = - \bigg( \nabla U\big( Z(t) \big) +\int \nabla \intPot\big( Z(t) - y\big) \mu_t(dy)\bigg) dt + \sqrt{2}dB(t), \quad \mu_t = \text{Law}\big(Z(t)\big),
\end{equation} 
for suitable confining and interaction potentials $U$ and $W$.
Such equations arise in the study of  pairwise interacting diffusions, and 
correspond to the dynamics in \eqref{eq:intro:Fdrift}, with  a drift functional $F$ that is  nonlocal  in the sense that  $F(t,\nu,x)$  depends on the whole measure $\nu$ and not just the density of $\nu$ at $x$, but has a  simple  affine dependence on the measure.   
In contrast, the drift functional in  \eqref{eq:intro:gamma} has a more complex  nonlinear dependence on the measure through its conditional distributions, and exhibits both local and nonlocal measure dependence.  Thus, the equation \eqref{eq:intro:CMVE} falls within the class of \textit{conditional McKean-Vlasov equations} (henceforth abbreviated to CMVE). CMVE 
 appear in numerous applications, for example, in stochastic Lagrangian models \cites{bossy2011stochasticLagrangian, bossy2019wellposedness}, stochastic volatility \cites{lacker2020inverting, djete2022nonregular},
 McKean-Vlasov equations with common noise \cites{delarue2024ergodicity, jianhai2024long, maillet2023note},
and entropic optimal transport \cite{conforti2023projected}.   
The study of CMVE is much more delicate than that of their non-conditional counterparts because conditional expectations, such as $\gamma$ in \eqref{eq:intro:gamma}, typically lack nice regularity properties and hence, many classical results for McKean-Vlasov equations do not apply \cite{buckdahn2023cmve}. 

In this paper we study the long-time behavior of the CMVE \eqref{eq:intro:CMVE}. A key motivation for analyzing this equation is that it  arises in the study of interacting diffusions on sparse random graphs.  It was shown 
in \cite{lacker2021marginal}  that the limit of the neighborhood empirical measure of interacting diffusions on random $\kappa$-regular graphs over any finite interval of time admits an autonomous characterization via a functional CMVE known as the local-field equation 
(see also \cite{ganguly2022thesis,GanRam24,ramanan2023sparse} for corresponding convergence results for interacting jump processes).   
   The local-field equation has a similar form to \eqref{eq:intro:CMVE}, except that at each time $t \geq 0$, the drift $\gamma$ is replaced by a path-dependent drift functional $\Gamma$  of  the form 
\[   \Gamma(t, f, g):= \mathbb{E}\bigg[\nabla U\big(X_0(t)\big) + \sum_{v = 1}^\kappa \nabla \intPot\big(X_0(t) - X_v(t)\big) \,\bigg|\,X_0[t] = f,\, X_1[t] = g \bigg],  \]
for continuous functions $f, g: [0,t] \to \R^d$, 
where for $u \in \{0,1,\ldots, \kappa\}$, 
 $X_u[t] := \{X_u(s), s \in [0,t]\}$ represents the history of the process.
Thus, the local-field equation describes a non-Markovian process
on $(\R^d)^{\kappa+1}$. The equation  \eqref{eq:intro:CMVE} can  be interpreted as a Markovian version of the local-field equation, and thus we refer to it henceforth as the \emph{$\kappa$-regular Markovian local-field equation}, or $\kappa$-MLFE (see Definition \ref{def:int:local} for a precise definition). 

Although analysis of  the original (non-Markovian) local-field equation may appear daunting,
its long-time behavior fortunately coincides in many cases with that of the $\kappa$-MLFE of  \eqref{eq:intro:CMVE},
even though the finite time marginals of the two in general differ (see Remark \ref{rk:notMarkovProjection}). \
This coincidence of long-time behavior has  been established 
in a companion paper \cite{hu2024gaussian} for the case when $U$ and $K$ are quadratic, $\kappa = 2$ and $d =1$,
though it is expected to  hold more generally, and  will be investigated  in forthcoming work.
Related  results have also been established for 
jump processes (e.g., see \cites{cocomello2023exact,ganguly2022thesis,GanRam25}).
  Thus, just as 
 ergodic properties of the  McKean-Vlasov equation 
\eqref{eq:intro:MV} have informed the understanding of  metastable behavior of particle systems (e.g., see \cites{Bas20,delgadino2023phase}), study of the long-time behavior of the \CMVE{} \eqref{eq:intro:CMVE} 
can shed light on the metastable behavior of the associated particle systems on random regular graphs, which is of considerable interest  \cite{ramanan2023sparse}.

\subsection{Main Contributions}

\subsubsection{Well-posedness of the $\kappa$-MLFE}
\label{subsub:wp}

Our first result,  Theorem \ref{thm:results:bddWp}, establishes well-posedness  (i.e., existence and uniqueness in law  of solutions)  of the $\kappa$-MLFE 
 \eqref{eq:intro:CMVE} under a suitable linear growth condition on the gradients of the potentials (see Assumption \ref{as:results:lyapunov}) and uniform boundedness of the gradient of the interaction potential. 
 The challenge here stems from the difficulty in verifying standard assumptions that  guarantee well-posedness of McKean-Vlasov equations such as  Lipschitz (or H\"older) continuity of the drift with respect to the measure and spatial variables. 
 In general,  well-posedness of CMVE is a more delicate issue than that of McKean-Vlasov equations and, like most 
 works on CMVE  (e.g., \cite{buckdahn2023cmve}), our well-posedness results are restricted to  bounded drifts; however, many of our subsequent results hold in greater generality as long as the  $\kappa$-MLVE is well-posed.  
 We  utilize a  Schauder fixed point argument combined with interior H\"older regularity estimates for Fokker-Planck equations, in the spirit of the recent work \cite{conforti2023projected},
 although we need to take
 extra care to ensure that solutions to the \CMVE{} are invariant under automorphisms of $\bbT_\kap^1$.
 It should be mentioned that the setting in \cite{conforti2023projected}  allows for more general integrable drifts by exploiting the fact that the CMVE measure-flow preserves marginal distributions. However, this special property does not hold in our setting.   Nevertheless, in a companion work \cite{hu2024gaussian}, we take a step towards relaxing the boundededness assumption by showing that the \CMVE{} is well-posed when $U$ and $\intPot$ are quadratic (so that the drifts are linear) when $\kap = 2$ and $d = 1$.

 Curiously,  well-posedness of the seemingly more complicated local-field equation is obtained in the case of unbounded coefficients (that satisfy certain linear growth conditions) in \cite{lacker2021marginal}.  There, the authors utilize a connection between the local-field equation and an infinite particle system which is not available for the \CMVE{}. It would certainly be of interest to provide broader conditions for well-posedness. 

\begin{openproblem}
{\em 
Identify more general  conditions on the drift, such as linear growth or integrability, under which the \CMVE{} is well-posed. } 
\end{openproblem}

\subsubsection{H-theorem,  sparse free energy and  long-time convergence of solutions to the \CMVE{} }
\label{subs:Htheorem}

Given well-posedness of the $\kappa$-MLFE, we show in Theorem \ref{thm:results:lyap} that  
 for each integer $\kappa \geq 2$, there exists an 
energy functional $\energy = \energy_\kap$ that decreases  along the flow $\{\mu_t, t \geq 0\}$ of the $\kappa$-MLFE   \eqref{eq:intro:CMVE}, and satisfies the following  {\em energy dissipation identity},  
\begin{equation}
    \energy (\mu_t) - \energy(\mu_s) = - \int_s^t \info (\mu_r) dr, 
    \label{eq:Htheorem}
\end{equation}
where $\info := \info_\kap$ is the  nonnegative modified Fisher information functional defined in \eqref{eq:results:info}. 
Specifically, we show that 
the energy functional $\mathbb{H}_\kappa$ takes the form 
\begin{equation}
    \energy_\kappa (\nu) :=  \int_{(\R^d)^{1 + \kap}} \bigg( \log \nu(\x) - \frac{\kap}{2} \log \bar{\nu}(x_0, x_1) + U(x_0) + \frac{1}{2}\sum_{v = 1}^\kap W(x_0 - x_v)\bigg) \nu(d \mathbf{x}),    
    \label{eq:intro:energykappa}
\end{equation} 
for any admissible probability measure  $\nu$ on $(\R^d)^{1+\kappa}$  (in the sense of Definition 
\ref{def:not:symProb}). Here, $\nu(\x)$ represents the density of $\nu$ at $x$  and $\bar{\nu}$ represents the $0$-$1$ marginal density of $\nu$ (see Definition \ref{def:not:marginal}).  
From  the non-negativity of the integrand on the right-hand side of \eqref{eq:Htheorem}, it is clear that any stationary distribution of the  $\kappa$-MLFE \eqref{eq:intro:CMVE} must be a  zero of the functional $\mathbb{I}_\kappa$. 
Under additional coercivity conditions  on $U$ and $W$ (see Assumption \ref{as:results:invMeasure}),  in Theorem \ref{thm:results:stationary} we prove that in fact the zeros of 
$\mathbb{I}_\kappa$ are in one-to-one correspondence with the set of stationary distributions of 
the $\kappa$-MLFE. Furthermore, under slightly stronger conditions
(see Assumptions \ref{as:results:strongCoercivity} and \ref{as:results:fixedPoint})
in Theorem \ref{thm:results:convergence} we show that 
from any admissible initial condition, the $\kappa$-MLFE measure flow converges to a stationary distribution.  Together 
with the finiteness of 
$\energy^*_\kappa := \inf_\nu \mathbb{H}_\kappa (\nu)$ (see Proposition \ref{prop:results:energyBound}), this establishes that $\energy_\kappa - \energy^*_\kappa$  acts as a {\em global} Lyapunov function for the measure flow $\{\mu_t\}_{t \geq 0}$  of the $\kappa$-MLFE even when the latter admits multiple stationary distributions. 

\noindent 
{\em Discussion of related prior work. } 
Theorem \ref{thm:results:lyap} falls under the rubric of H-theorems, whose  origin goes back to  Boltzmann's H-theorem. This celebrated result from  statistical mechanics   states that the Boltzmann-Gibbs entropy  is non-decreasing
along the flow of  Boltzmann's kinetic equation, and is constant only  at stationary states known as Maxwellians (see \cite{villani2002review} for a review). In the context of Langevin diffusions, an H-theorem of the form \eqref{eq:Htheorem} famously holds for solutions to linear Fokker-Planck equations with gradient drift.  These equations possess a unique invariant distribution $\pi$ and  satisfy  
\eqref{eq:Htheorem} with $\mathbb{H} = \mathcal{H} (\cdot|\pi)$, where ${\mathcal H}$ is the relative entropy functional and $\mathbb{I}$ is the relative Fisher information (see Section 2 of \cite{jordan1998variational} or Theorem 5.2.2 of \cite{bakry2014diffusion}).

Extending such H-theorems  to other nonlinear PDEs arising in statistical mechanics 
is  also of great interest, although this endeavor is presented with additional challenges. In particular, stationary distributions for such equations may be non-unique and the usual Boltzmann-Gibbs or relative entropy functionals will in general not be non-increasing along the measure flow. Therefore a central difficulty lies in the identification of the correct energy functional for the NFPE. Works that have focused on  H-theorems for NFPEs include \cite{schwammle2007consequences}, which provides a  non-rigorous  way of deriving the correct energy functional for NFPEs with local nonlinearity,  and \cite{barbu2024nfpe} (see also references therein), which 
proves regularity and establishes H-theorems for
singular NFPE of Nemytskii type. 
Moreover, it is well known (e.g., see \cite{tamura1987free, carrillo2003kinetic})  that the identity \eqref{eq:Htheorem} holds for a class of granular media equations, including McKean-Vlasov equations of type \eqref{eq:intro:MV}, with the energy functional taking the form of the  so-called \textit{free energy} given by
\begin{equation}
\begin{aligned}
    \mathbb{H}_{\text{FE}}(\nu) :&= \int_{\R^d} \nu(x) \log \nu(x) dx + \int_{\R^d} U(x) \nu(x) dx + \frac{1}{2}\int_{\R^d} \int_{\R^d} \intPot(x - y) \nu(x) \nu(y) dx dy,    
\end{aligned}
    \label{eq:intro:freeEnergy}
\end{equation}
for absolutely continuous $\nu$ in $\mP(\R^d)$, the space of probability measures on $\R^d$. The corresponding functional $\mathbb{I} = \mathbb{I}_{\text{FE}}$ is sometimes referred to as the \textit{entropy dissipation functional}.

These results, however, do not extend naturally to CMVE such as the \CMVE{}; since $\gamma$ takes the form of a conditional expectation, the nonlinearity in \eqref{eq:intro:gamma} is neither affine (as in \eqref{eq:intro:MV}) nor purely local (as in the setting of \cite{barbu2024nfpe}). In fact, the nonlinear measure dependence in \eqref{eq:intro:gamma} simultaneously exhibits local influence through the conditioning variables and nonlocal effects through the expectation. Two papers that are closer in spirit to our setting have appeared in the sampling literature,  
 in the context of biased adaptive forcing algorithms \cite{lelievre2008forcing} and entropic optimal transport \cite{conforti2023projected}, with the crucial difference that both start with a given target distribution $\pi$ and construct ergodic CMVEs that converge to $\pi$. 
 The convergence analysis in \cite{lelievre2008forcing}   has the flavor
  of an H-theorem with relative entropy as the energy functional, and in \cite{conforti2023projected} an explicit  H-theorem is
established, again with relative entropy as the energy functional. 
  However in our setting the stationary distributions of the \CMVE{} are {\em a priori} unknown and may be non-unique, and so standard relative entropy cannot be used to characterize convergence.  We
  also briefly mention works that have used alternative approaches such as coupling
  techniques used to characterize long-time convergence of McKean-Vlasov equations
\cite{cattiaux2008probabilistic, eberle2019QuantHarrisThm} and 
a class of CMVEs that arise as the limit of mean-field systems with common noise
(see \cite{maillet2023note, jianhai2024long, delarue2024ergodicity} and the references therein).  
However, the \CMVE{} is quite different from common-noise CMVE as the dependency
structure of \eqref{eq:intro:CMVE} is more complex, and seems not easily amenable to coupling
constructions.

Theorem \ref{thm:results:lyap} shows that the functional $\mathbb{H}_\kappa$  identified in \eqref{eq:intro:energykappa} is in fact the correct energy functional for the \CMVE{}, which will henceforth be referred to  as the {\em sparse free energy}.  As with other NFPEs, the H-theorem for the \CMVE{} provides substantial information about long-time behavior. For example, it supplies a mechanism by which one can identify stationary distributions.  A natural next step is to understand the topology of the set of stationary distributions of the \CMVE{}, in a manner similar to that of \cite{Bas20} in the mean-field setting. 

\begin{openproblem}
    {\em When the \CMVE{} admits multiple stationary distributions, characterize the basin of attraction  of each stationary distribution, or equivalently, characterize the map that associates to each admissible initial condition (in the sense of Definition \ref{def:results:entropyVariance}) the corresponding stationary distribution to which the flow starting from that initial condition converges. }
\end{openproblem}

\subsubsection{Correspondence with stationary distributions of particle systems on $\kappa$-regular trees}
\label{subsec:gibbs}

Along the way to  proving a bijection between the set of stationary distributions of the $\kappa$-MLFE and
the zeros of the modified Fisher information functional $\mathbb{I}_\kappa$ mentioned in the last section, we also show that
the zeros of $\info_\kap$ can be characterized as fixed points of a certain recursion on the $\kappa$-regular tree 
in Theorem \ref{thm:results:zeros}. 
We refer to these as  \FPES{}  since the associated recursion 
 is related to classical recursions on regular (Cayley) trees that 
characterize Markov chains or splitting
Gibbs measures on trees with discrete spaces (see Chapter 12 of \cite{georgii1988gibbs}, or
\cite{spitzer1975markov,zachary1983countable})  and those used 
to establish uniqueness of continuous Gibbs measures on trees
(e.g., see \cite{gamarnik2019unique,rozikov2013book}). 
We also identify another bijection between \FPES{} and fixed points of a different recursion, which was first introduced in \cite{lacker2023stationary} to study  $2$-particle 
marginals of a class of continuous Gibbs measures on regular trees (see Corollary \ref{cor:LackerZhang}). Additionally, it is possible to show that 
the \FPES{}  are in one-to-one correspondence with (root neighborhood) marginals of automorphism invariant 2-MRF stationary distributions
of the associated system of interacting diffusions on the infinite $\kappa$-regular tree
(via an argument similar to the one used in  \cite{ganguly2022thesis,GanRam25} for  interacting pure jump processes). 
As a result, we establish 
 a  connection between the dynamical local-field
equations of \cite{lacker2021marginal} and $\kappa$-MLFE on the one hand, and $2$-particle marginals of continuous Gibbs measures that solve the stationary local equations  of  \cite{lacker2023stationary} on the other, thereby resolving an open problem stated in  \cite{lacker2023stationary}.
This connection also provides additional support for the link between the 
long-time behaviors of the \CMVE{} and the associated local-field equation, which
would be desirable to establish in full generality, beyond the linear case considered in \cite{hu2024gaussian}.

\begin{openproblem}
  Does $\bbH_\kap - \bbH_\kap^*$, where recall $\bbH_\kap^* := \inf_{\nu} \bbH_\kap(\nu)$, also  serve as a global Lyapunov function
 for the (non-Markovian) local-field equation? 
\end{openproblem}

\subsubsection{Renormalized relative entropies}
\label{subsub:intro:k2}

We now describe additional results that we obtain in the case $\kappa = 2$, when $\mathbb{T}_2$ can be identified with the integers $\mathbb{Z}$ (viewed as a rooted graph with edges between consecutive integers,  and the root at $0$).  First, we adopt 
a probabilistic approach to identify the sparse free energy  functional $\mathbb{H}_2$. This takes inspiration from
the fact that  the long-time behavior of the \CMVE{} and the local-field equation are expected to coincide, and that the local-field equation describes the limit of the marginal dynamics on the root  vertex  neighborhood of the following system of  interacting diffusions on $\mathbb{T}_2^n$, the  $2$-tree truncated at level $n$,
in the  $n \rightarrow \infty$  asymptotic regime (see e.g., \cite{lacker2021marginal}):
\begin{equation}
  \label{IPS-T2n}
  dX_i^n(t) = - \left( \nabla U (X_i^n(t)) + \sum_{j \in  {\mathcal N}_i}\nabla W(X_i^n(t) - X_{j}^n(t)) \right)
dt  + \sqrt{2} dB_i(t), \quad i  \in \{-n, \ldots, n\},  
\end{equation}
where  ${\mathcal N}_i$ is the  set of neighbors of $i$ in $\mathbb{T}_2^n$ and $\{B_i\}_{i \in \{-n, \ldots, n\}}$ are independent standard Brownian motions.  Given that each $n$-system described above has a unique stationary distribution $\theta^n$ (see Definition \ref{def:results:gibbs}),   it is well known (as mentioned in Section \ref{subs:Htheorem}) that  
the functionals given by 
\begin{equation}
    \label{lin-Lyap}
\mu \mapsto \mathbb{H}^{(n)}(\mu) := {\mathcal H} (\mu\| \theta^n) \quad
\text{ and } \quad \mu \mapsto \widehat{\mathbb{H}}^{(n)}(\mu) := {\mathcal H} (\pi^n \| \mu),
\end{equation}
both decrease along the corresponding linear Fokker-Planck measure flow on ${\mathcal P}((\R^d)^{2n+1})$.

To  analyze the flow of the $2$-MLFE, we show the existence of a lift map 
  \begin{equation}
      \label{eq-liftmap}
  \psi^n: {\mathcal P} ((\R^d)^3) \mapsto 
  {\mathcal P} ( (\R^d)^{2n+1})\:\text{ with }\: \nu \mapsto \psi^n_\nu,
  \end{equation}
  such that for every admissible $\nu \in {\mathcal P}((\R^d)^3)$ (in the sense of Definition \ref{def:results:entropyVariance}), the following limit of renormalized energy functionals  exists and coincides with $\mathbb{H}_2(\nu) - \mathbb{H}_2^*$:
  \begin{equation}
      \label{eq-renorm}
    \lim_{n \rightarrow \infty} \frac{1}{n} \mathbb{H}^{(n)} \big(\psi^n_\nu\big) =     \lim_{n \rightarrow \infty}\frac{1}{n} {\mathcal H} (\psi^n_\nu \| \theta^n) = \bbH_2(\nu) - \bbH_2^*. \end{equation}

  The lift map $\psi^n$ we choose is the unique extension of any probability $\nu$ on $(\R^d)^3$ to a second-order Markov random field ($2$-MRF for short) on $(\R^d)^{2n+1}$; see Definition \ref{def:MRF} and Definition \ref{def:results:2measures} for precise definitions of 2-MRF and $\psi^n$ respectively.   
  Figure \ref{fig:Lift} illustrates this general approach, with $P_t^n$ denoting the semi-group related to the system of $n$ 
  interacting linear diffusions \eqref{IPS-T2n}, and $P_t$ the (nonlinear) semi-group of the $2$-MLFE.
  We believe that this representation  in terms of renormalized limits 
  points to the canonical nature of the sparse free energy, and explains why (its shifted version) 
  exhibits a global Lyapunov property  even when there are multiple stationary distributions. 
  
\begin{figure}[h]
    \[
\begin{tikzcd}[column sep=5 cm, row sep=1.5 cm]
\psi^n_{\mu_0}  \arrow[ r, "\text{$n$-Particle Linear Semigroup} "] & P_t^n \psi^n_{\mu_0} \arrow[d, "\,n \rightarrow \infty"] \arrow[ r, "\text{Lyapunov function: }  \bbH^{(n)}(P_t^n\psi^n_{\mu_0})", "t \rightarrow \infty"'] & \theta^n \arrow[d, dashed, "n \rightarrow \infty"] \\
\mu_0\arrow[u, "\text{Lift Map}" ]  \arrow[r, "\text{Nonlinear Semigroup}"] & {\mu}_t   \arrow[r, "\text{Lyapunov function: }  \bbH_2(\mu_t)", "t \rightarrow \infty"'] & \mathcal{S} 
\end{tikzcd}
\]
\caption{Convergence diagram for entropy renormalization. Here, $\mS$ is the set of zeros of $\bbI_2$, which in our setting are also the limit points of $\mu_t$ (see Theorem \ref{thm:results:convergence}). By Theorem \ref{thm:results:zeros} and Corollary \ref{cor:LackerZhang}, the set $\mathcal{S}$ is also the set of root marginal distributions of continuous Gibbs measures on regular trees, and therefore can be identified as  possible limit points of root marginals of $\theta^n$. This fact is reflected by the dashed arrow.}
  \label{fig:Lift}
\end{figure}
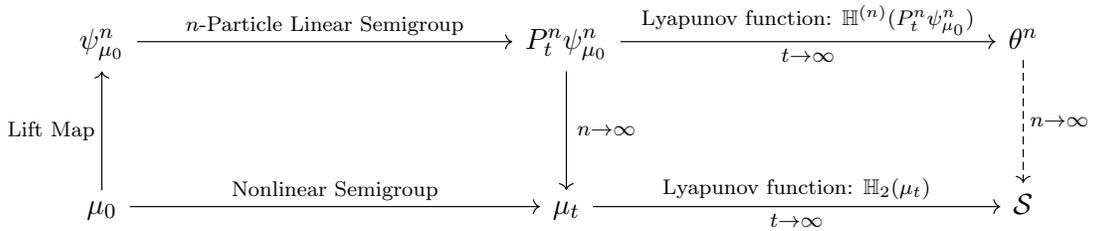

 Although the approach outlined above may seem natural, 
  there are several subtleties.  First, if  one replaces   $\mathbb{H}^{(n)}$ by the functional $\widehat{\mathbb{H}}^{(n)}$ from \eqref{lin-Lyap}, which also  serves as a Lyapunov function for the measure flow associated with the system of $n$ interacting diffusions, in general the limit  in \eqref{eq-renorm} (even when it exists) will  not decrease along the $2$-MLFE measure flow.   Second, 
  the choice of the lift map $\psi^n$ is far from obvious and other {\em a priori} reasonable constructions, such as an extension as a first-order Markov random field ($1$-MRF for short), fail (see  Remark \ref{rk:2MRFNec}).  This may seem all the more  surprising in light of the fact that  the stationary distribution $\theta^n$  of the particle system \eqref{IPS-T2n} is a $1$-MRF.  However, our choice  is linked to the  observation that the {\em trajectories} of the interacting diffusions \eqref{IPS-T2n} on $\mathbb{T}_2^n$  form a 2-MRF, but not in general a $1$-MRF,  even when the initial conditions are i.i.d.
  (see Theorem 2.4 and Section 3.3 of \cite{lacker2021MRF}; and  also \cite{GanRam22} for related results for jump processes).   Finally, it is the local-field equation, and {\em not} its Markovian analog, the $\kappa$-MLFE, which arises
  as the  limit of the marginal dynamics (on the root neighborhood) of interacting diffusions on 
  $\mathbb{T}_2^n$.    Thus, even if the two equations are expected to have the same long-time behavior,
  it is not at all clear that  they should share a common  Lyapunov function.
  However, it is shown that this is indeed  true for the Gaussian case in \cite{hu2024gaussian}, thus pointing
  to the robustness  of the identified sparse free energy. Although our other results hold for all $\kap \in \N$, naive approaches to characterizing $\bbH_\kap$ as a renormalized limit fail for 
  $\kap \geq 3$. This leads naturally
  to the following question. 
  
\begin{openproblem}
    Can $\bbH_\kap$ be represented as the limit of renormalized entropies for $\kap \geq 3$? 
\end{openproblem}

\noindent {\em Relation to Prior Work. } 
A similar probabilistic approach has been used to identify the free energy functional for McKean-Vlasov equations of
the form \eqref{eq:intro:MV} through a connection with  mean-field particle systems.   Consider the system of
$n$ interacting diffusions on the complete graph given by 
\[  dX_i^n (t) = - \left( \nabla U(X_i^n(t)) + \frac{1}{n}\sum_{j =1, j \neq i}^n  \nabla W(X_i^n(t) - X_j^n(t)) \right) dt + \sqrt{2} dB_i (t), \quad i = 1, \ldots, n, 
\]
where $\{B_i\}_{i\in \{1, \ldots, n\}}$ are independent standard Brownian motions.
Then it follows   (see, e.g., Lemma 17 of \cite{guillin2022lsi}) that
  \eqref{eq-renorm} holds with the lift map $\psi^n: \mP(\mathbb{R}^d) \mapsto \mP((\mathbb{R}^d)^n)$
simply given by the  $n$-fold product $\psi^n (\nu) = \nu^{\otimes n}.$     Similar forms of renormalized entropy have also appeared in recent work on quantitative propagation of chaos for singular kernels \cite{jabin2018qpoc}, and 
analogous results have also been established for mean-field jump processes in \cite{budhiraja2015entropy} and \cite{budhiraja2015kolmogorov}.  
In the mean-field setting, the product structure 
of the lift map is natural and is a reflection of propagation of chaos results 
(e.g., see Theorem 1.4 of \cite{sznitman1991poc}),
which imply that for all $m \in \N$, we have 
\begin{equation}
    \text{Law}(X_1^n(t), \ldots, X_m^n(t)) \implies  \mu_t^{\otimes m},
    \label{eq:intro:POC}
\end{equation}
where $\mu_t$ is the law of the limit McKean-Vlasov system \eqref{eq:intro:MV}. 
However, propagation of chaos for sparse particle systems is much more delicate (see \cite{lacker2023localweakconvergence}) and results of the form \eqref{eq:intro:POC} certainly do not hold. Therefore, the correct choice of the lift map in the sparse setting is far from obvious.

\subsubsection{Logarithmic Sobolev inequalities and exponential convergence}
\label{subsub:intro:renorm}

In the case $\kappa = 2$, under additional assumptions we guarantee existence of a unique stationary distribution in Theorem \ref{thm:results:uniqueRegime}.
We then combine the representation of $\energy_2$ as the limit of renormalized entropies discussed in the last section
with a uniform logarithmic Sobolev inequality for $\theta^{n}$  to establish exponential decay of the sparse free energy along the $2$-MLFE flow (see  Theorem \ref{thm:results:LSI}).  When  $\kap \geq 3$, the stationary distributions are continuous Gibbs measures on infinite trees, which may not be unique  (see e.g., \cite{lacker2023stationary, georgii1988gibbs, rozikov2013book}), and thus it is not clear if exponential decay of the sparse free energy  would hold in the regime of non-uniqueness. It would be of interest to determine the rate of convergence even in the uniqueness regime.

\begin{openproblem}
    Obtain rates of convergence of the sparse free energy along the  \CMVE{} measure flow for $\kap \geq 3$. 
\end{openproblem}

\noindent 
{\em Discussion of  related  prior Work. } Several authors have obtained nonlinear log Sobolev inequalities,
by which we mean functional inequalities that upper bound the
free energy $\mathbb{H}_{\text{FE}}$  by a multiple of the entropy dissipation functional $\mathbb{I}_{\text{FE}}$.
 Indeed, this was carried out
 in the  seminal work  of \cite{carrillo2003kinetic}  using both Bakry-\'Emery and HWI techniques, and then used to provide  exponential rates of relaxation   (see also 
 \cites{carrillo2020long, delgadino2023phase} and  references therein).
 Our technique parallels that of  \cite{guillin2022lsi}, which instead uses the previously described characterization of the free energy $\mathbb{H}_{\text{FE}}$
 as the limit of renormalized relative entropies to prove (Poincar\'{e} and) logarithmic Sobolev type inequalities, and associated exponential convergence of the free energy along the  McKean-Vlasov measure flow.

\subsection{Structure of the Paper} The rest of the paper is structured as follows. In Section 2 we introduce some common notation used throughout the paper. In Section 3 we define the \CMVE{} in Definition \ref{def:int:local} and establish its well-posedness (see Theorem \ref{thm:results:bddWp}). In Section \ref{sec:results:hTheorem} we introduce the functionals $\bbH_\kap$ and $\bbI_\kap$ and present the H-theorem (see Theorem \ref{thm:results:lyap}). In Section \ref{sec:results:stationary} we characterize the stationary distributions of the \CMVE{} as the zeros of $\bbI_\kap$ and show convergence of the \CMVE{} flow to a stationary distribution (see Theorem \ref{thm:results:stationary} and Theorem \ref{thm:results:convergence}). In Section \ref{sec:results:CBP} we introduce the \FPES{} and show that they  coincide both with the zeros of $\bbI_\kap$ and the root marginals of certain continuous Gibbs measures that are stationary distributions of interacting diffusions on the infinite $\kap$-regular tree (see Theorem \ref{thm:results:zeros} and Remark \ref{rk:lackerZhang}). In Section \ref{sec:results:kap2}, we present additional results when $\kap = 2$, in particular a characterization of $\bbH_\kap$ as the limit of renormalized relative entropies (Theorem \ref{thm:results:2entropy}), uniqueness of stationary distributions (Theorem \ref{thm:results:uniqueRegime}), and a modified logarithmic Sobolev inequality (Theorem \ref{thm:results:LSI}). Section \ref{sec:wp:proof} presents the proof of well-posedness of the \CMVE{} for bounded interactions. Section \ref{sec:lyap:Htheorems} is devoted to the proof of the H-theorem and properties of the sparse free energy $\bbH_\kap$.  In  Section \ref{sec:lyap:zeros} we establish results about the stationary distributions of the  \CMVE{}. Section \ref{sec:k2Stuff} contains the proofs of results particular to the $\kap = 2$ case. Finally, in Appendix \ref{_sec:ap:PDE} we collect several useful facts about linear Fokker-Planck equations.

\section{Notation}
\subsection{Vectors and Gradients}
\label{sec:not:vectors}
Let $\mX$ be a Polish space. We consider vectors of the form $\mathbf{x} = (x_0, x_{\{1, \ldots, \kap\}}) \in \mX^{1 + \kap}$ where we distinguish the element $x_0$. We use $\nabla$ to denote the weak gradient on $\R^d$ (e.g., see Section 1.1 of \cite{fokkerPlanck} for definitions of weak derivative). We let $\nabla_{\x}$ denote the gradient on $(\R^d)^{1 + \kap}$ and for $v \in \{0, \ldots, \kap\}$, let $\nabla_{x_v}$  denote the partial gradient with respect to $x_v$.

\subsection{Probability Measures, Probability Spaces, and Equipped Functionals}
\label{ss:prbMeasure}

Throughout,  $\mX$ denotes a Polish space, and $\mP(\mX)$ denotes the set of Borel probability measures on $\mX$, equipped with the topology of weak convergence. For any measure $\nu \in \mP(\mX)$ and $m \in \N$, we let $\nu^{\otimes m} \in \mP(\mX^m)$ denote the $m$-fold product measure of $\nu$.

Let $(\Omega, \mF, \bbP)$ be a probability space. For a random variable $X:(\Omega, \mF) \rightarrow \mX$, let $\mL(X) \in \mP(\mX)$ denote the law of $X$ under $\bbP$. For a measure $\nu \in \mP(\mX)$, we let $Y$ denote the canonical random variable on the probability space $(\mX, {\mathcal B}(\X), \nu)$ with law $\nu$. For any measurable function $f: \mX \rightarrow \R^d$, we write
\begin{equation*}
    \bbE^\nu[f(Y)] := \int_\mX f(y) \nu(dy).
\end{equation*}
We write $X \stackrel{(d)}{=} Y$ if $X$ and $Y$ are two random variables that are equal in distribution. When $\nu \in \mP(\R^d)$ is absolutely continuous with respect to $d$-dimensional Lebesgue measure, we abuse notation and denote the density $\tfrac{d\nu}{dx}$ by $\nu$ so that $\nu(dx) = \nu(x) dx$. 

For any two measures $\nu, \tilde{\nu} \in \mP(\mX)$, we write $\nu \ll \tilde{\nu}$ to indicate that  $\nu$ is absolutely continuous with respect to $\tilde{\nu}$. The \textit{relative entropy} of $\nu$ with respect to $\tilde{\nu}$ is given by 
    \begin{equation*}
        \entropy\big( \nu\big| \tilde{\nu} \big) = \left\{
        \begin{aligned}
            &\int_\mX \log\bigg( \frac{ d\nu}{d\tilde{\nu}}(x) \bigg) \nu(dx) 
            &\quad& \mbox{if}\quad
            \nu \ll \tilde{\nu}, 
            \\
            &\infty &\quad& \mbox{otherwise. }
        \end{aligned}
        \right.
    \end{equation*}
In the case $\mathcal{X} = \R^m$, we define the \textit{relative Fisher information} of $\nu$ with respect to $\tilde{\nu}$ to be 
    \begin{equation*}
        \mathcal{I}\big( \nu\big| \tilde{\nu} \big) = \left\{
        \begin{aligned}
            &\int_{\R^m} \bigg|\nabla \log\frac{ d\nu}{d\tilde{\nu}}(x) \bigg|^2 \nu(dx) 
            &\quad& \mbox{if}\quad
            \nu \ll \tilde{\nu}, 
            \\
            &\infty &\quad& \mbox{otherwise. }
        \end{aligned}
        \right.
    \end{equation*}

\subsection{Function Spaces}We let $\R_+$ denote the set of non-negative real numbers. Given $m \in \N$ we let $C^k(\R^m)$ denote the space of $k$-times continuously differentiable functions from $\R^m$ to $\R$ and $C^k_b(\R^m)$ denote the space of bounded functions in $C^k(\R^m)$.

For a measure $\nu \in \mP(\mX)$, we abbreviate $\nu$-almost everywhere to $\nu$-a.e., and write simply a.e. when $\nu$ is Lebesgue measure. Let $L^p(\mu)$ denote the usual space of $p$-integrable functions with respect to $\mu$ and $\|\cdot\|_{L^p(\mu)}$ denote the usual $L^p(\mu)$-norm. If $E \subset \R^m$ is a measurable subset and $\mu$ is Lebesgue measure, we simply write $L^p(E)$ instead for convenience.

Let $W^{s, p}(\R^m)$ denote the Sobolev space of measurable functions $f:\R^m \rightarrow \R$ such that $f \in L^p(\R^m)$ and $|\nabla^r f| \in L^p(\R^m)$ for $r = 1, \ldots, s$, where the derivative is taken in the weak sense. The Sobolev norm is given by
\begin{equation*}
    \|u\|_{W^{s, p}(\R^m)} := \bigg[ \sum_{k = 0}^s \|\nabla^s u \|_{L^p(\R^m)}^p\bigg] ^{\frac{1}{p}}.
\end{equation*}
For $R \in (0, \infty)$, let $B_R:=\{x \in \R^m : |x| < R\}$ denote the open ball of radius $R$ in $\R^m$. We let $L^p_{\text{loc}}(\R^m)$ denote the set of measurable functions $f$ such that $f \in L^p(B_R)$ for every $R \in (0, \infty)$. Similarly, define $W^{s, p}_{\text{loc}}(\R^d)$ to be the set of functions in $L^p_{\text{loc}}(\R^d)$ such that $|\nabla^r f| \in L^p_{\text{loc}}(\R^d)$ for all $r = 1, \ldots, s$.

We let $\mC^m_T$ denote the space of continuous functions from $[0, T]$ to $\R^m$ with the topology of uniform convergence. Similarly, let $\mC^m$ denote the space of continuous functions from $[0, \infty)$ to $\R^m$ equipped with the topology of uniform convergence on compact sets.

For $m, k \in \N$, we say a measurable function $f:\R^m \rightarrow \R^k$ satisfies a \textit{linear growth condition} if there exists $c_f > 0$ such that 
\begin{equation}
    |f(x)| \leq c_f(1 + |x|), \quad \text{for a.e. }x \in \R^d.
    \label{eq:not:linGrowth}
\end{equation}
\subsection{Graphs and Markov Random Fields} For a graph $G = (V, E)$ and $v \in V$, let $d_G:V \times V \rightarrow \R_+$ represent the graph distance. For any $A \subset V$, we let $N_A(G) := \{ v \in V: \inf_{u \in A} d_G(v, u) = 1\}$ denote the \emph{boundary} of $A$ in $V$ and $N^2_A(G) := \{ v: 1 \leq \inf_{u \in A} d_G(v, u) \leq 2\}$ denote the double boundary. When the underlying graph $G$ is clear, we write $\partial A := N_A(G)$ for the boundary, $\partial^2 A := N_A^2(G)$ for the double boundary, and set $\bar{A} := N_A(G) \cup A$. We next recall the definition of a Markov random field (e.g. Definition 1.1 of \cite{lacker2021MRF}):

\begin{definition}[Markov random field]
\label{def:MRF}
    Let $G = (V, E)$ be a finite graph. Let $(Y_v)_{v \in V}$ be a random element of $(\R^d)^V$ with distribution $P \in \mP((\R^d)^V)$. Then  $Y = (Y_v)_{v \in V}$, or equivalently its distribution $P$, is said to be a \emph{first-order Markov random field} (abbreviated as 1-MRF) on $(\R^d)^V$ if $Y_A$ is conditionally independent of $Y_{(\bar{A})^c}$ given $Y_{\partial A}$, for every finite set $A \subset V$. Similarly,  $Y = (Y_v)_{v \in V},$ or equivalently its distribution $P$, is said to be a \emph{second-order Markov random field} (abbreviated as 2-MRF) on $(\R^d)^V$ if $Y_A$ is conditionally independent of $Y_{(A \cup \partial^2 A)^c}$ given $Y_{\partial^2 A}$. 
\end{definition}

\section{The $\kappa$-Regular Markovian Local-Field Equation and its Well-posedness}
\subsection{The $\kap$-regular Markovian local-field equation} In this section, we introduce our main equation of study (Definition \ref{def:int:local}). We first define a space of probability measures that captures the relevant symmetries of the equation (i.e. the automorphisms of $\bbT_\kap^1$).
\begin{definition}[Edge marginal]
\label{def:not:marginal}
Fix $\kap \in \N$ with $\kap \geq 2$. Let $\Pi:(\mX)^{1 + \kap} \rightarrow \mX \times \mX$ be the projection map $\Pi(\x) = (x_0, x_1)$. For any $\nu \in \mP((\mX)^{1 + \kap})$, define the \emph{edge marginal} $\bar{\nu} := \nu \circ \Pi^{-1} \in \mP(\mX \times \mX)$ to be the 0-1 marginal of $\nu$.    
\end{definition}
\begin{definition}[Symmetric probability measures]
\label{def:not:symProb}
    The space $\mathcal{M}_{\kap, d}$  is the set of probability measures $\nu \in \mP((\R^d)^{1 + \kap})$ that satisfy the following two properties:
    \begin{enumerate}
        \item (Leaf Exchangeability) For any permutation $\tau$ on $\{1, \ldots, \kap\}$, we have
        \begin{equation}
            \label{eq:not:symmetry1_}
            \nu(dx_0, dx_{\{1, \ldots, \kap\}}) = \nu(dx_0, dx_{\{\tau(1), \ldots, \tau(\kap)\}}).
        \end{equation}
        \item (Edge Symmetry) $\bar{\nu}$ satisfies
        \begin{equation}
                    \label{eq:not:symmetry2_}
            \bar{\nu}(dx_0, dx_1) = \bar{\nu}(dx_1, dx_0).
        \end{equation}
    \end{enumerate}
    In addition, given $T \in (0, \infty)$, we define the space $\mathcal{M}_{k,d}^T$ to be the set of measures $\mu \in \mP((\mC_T^d)^{1 + \kap})$ such that $\mu_t \in \mathcal{M}_{\kap, d}$ for all $t\in [0, T]$. 
\end{definition}
\begin{remark}[Exchangeability of marginals]
\normalfont
\label{rk:not:exchangeability}
Fix $\nu \in \mPS$ and a random vector $(Y_0, Y_1, \ldots, Y_\kap) \sim \nu$. As a consequence of the above definition, we have $$(Y_0, Y_v) \sim \bar{\nu} \quad\text{and}\quad (Y_0, Y_v) \stackrel{(d)}{=} (Y_v, Y_0), \quad v \in \{1, \ldots, \kap\}.$$ In particular, we have $Y_v \stackrel{(d)}{=}Y_0$ for all $v \in \{1,\ldots, \kap\}$. Moreover, since leaf exchangeability and edge symmetry are both clearly preserved by weak convergence, $\mPS$ is a weakly closed subset of $\mP((\R^d)^{1 + \kap})$.
\end{remark}

We now introduce the $\kappa$-regular Markovian local-field equation.
\begin{definition}[$\kappa$-regular Markovian local-field equation]
\label{def:int:local}
    Let $\kap, d \in \N$ and $T \in (0, \infty)$. Let $U, \intPot:\R^d \rightarrow \R$ be continuously differentiable functions. Define the function
\begin{equation}
        b(\x) := \nabla U(x_0) + \sum_{v = 1}^\kappa \nabla \intPot(x_0 - x_v).
        \label{eq:results:drift}
    \end{equation}
Let $\lambda \in \mPS$. A solution to the \emph{$\kappa$-regular Markov local-field equation}, (henceforth abbreviated to \emph{\CMVE{}}), on $[0, T]$ with potentials $(U, \intPot)$ and initial law $\lambda$ is a tuple 
    \begin{equation*}
        \Big((\Omega, \mF, \bbF, \bbP), (\mu, \gamma), (\mathbf{B}, \mathbf{X})\Big)
    \end{equation*}
    such that
    \begin{enumerate}
        \item $(\Omega, \mF, \bbP)$ is a probability space with a filtration $\bbF = (\mF_t)_{t \in [0, T]}$.
        \item $\mathbf{B} := (B_v)_{v \in \{0, \ldots, \kappa\}}$ is a family of $d$-dimensional independent $\bbF$-Brownian motions on $[0, T]$.
        \item $\X := (X_v)_{v \in \{0, \ldots, \kappa\}}$ is a family of $\bbF$-adapted $d$-dimensional continuous processes on $[0, T]$. Moreover, $\X$ is a Markov process with respect to its natural filtration.
        \item $\mu$ satisfies $\mu = \mL(\X)$, $\mu_0 = \lambda$, and $\mu \in \mathcal{M}_{k,d}^T$. 
        \item The measurable function $\gamma: [0, T] \times \R^d\times \R^d \rightarrow \R^d$ satisfies
        \begin{equation}
            \gamma(t, X_0(t), X_1(t)) = \bbE\big[ b(\X(t) ) \big| X_0(t), X_1(t) \big], \quad {\mu}_t\text{-almost surely for a.e. }t \in [0, T]
            \label{eq:not:gamma}
        \end{equation}
        \item $\X$ satisfies the following system of SDEs for $t \in (0, T]$:
        \begin{equation}
        \label{eq:not:mLocal}
        \begin{split}
            dX_0(t) &= -b(\X(t)) dt + \sqrt{2} dB_0(t), \\
            dX_v(t) &= -\gamma(t, X_v(t), X_0(t)) dt + \sqrt{2} dB_v(t), \quad v = 1, \ldots, \kap.
        \end{split}
        \end{equation}
        \item For each $v \in \{1, \ldots, \kap\}$
        \begin{align}
            \int_0^T \Big( \big|b\big(\X(t)\big)\big|^2 + \big| \gamma\big(t, X_0(t), X_v(t)\big) \big| ^2 + \gamma\big(t, X_v(t), X_0(t)\big)\big|^2 \Big) dt < \infty \quad \text{a.s.}
            \label{eq:cmve:integrability}
        \end{align}
    \end{enumerate}
    Further, the tuple is said to be a solution to the \CMVE{} on $[0, \infty)$ if all properties hold with $[0, T]$ replaced with $[0, \infty)$
    \end{definition}
    
\begin{remark}\normalfont
    When the underlying probability space and Brownian motion are clear, we denote a solution to the \CMVE{} simply by $(\mu, \gamma, \X)$, or  just $(\mu, \gamma)$ when we do not need explicit reference to the stochastic process $\X$.
    
    Given a solution $(\mu, \gamma)$ to the \CMVE{} on $[0, T]$ with potentials $(U, \intPot)$ and initial condition $\lambda \in \mPS$, we define $\eta: [0, T] \times (\R^d)^{1 + \kap} \rightarrow (\R^d)^{1 + \kap} $ to be a measurable function that satisfies
    \begin{equation}
        \big(\eta(t, \x)\big)_v :=  \left\{
        \begin{aligned}
            &\nabla U(x_0) + \sum_{v = 1}^\kap \nabla \intPot(x_0 - x_v), &\quad &v = 0,
            \\
            &\gamma(t, x_v, x_0), &\quad &v \in \{1, \ldots, \kap\},
        \end{aligned}
        \right.
    \label{eq:not:eta}
    \end{equation}
for $\mu_t$-a.e. $\x \in (\R^d)^{1 + \kap}$ and $t \in [0, T]$. We can then write the SDE \eqref{eq:not:mLocal} more succinctly as 
\begin{equation}
\label{eq:not:SimpleCMVE}
\begin{aligned}
    d\X(t) &= - \eta(t, \X(t)) dt + \sqrt{2} d\mathbf{B}(t), \\
    \mu_t &= \mathcal{L}(\X_t).
\end{aligned}
\end{equation}
\end{remark}
\begin{remark}[Local-field equation and \CMVE{}]\normalfont
\label{rk:notMarkovProjection}
    In light of classical mimicking results (see for example \cite{gyongy1986mimicking, brunick2013mimic}), it is tempting to believe that the \CMVE{} is a `projected' version of the local-field equation described in \cite{lacker2021marginal} and that the local-field equation and \CMVE{} share the same time-marginals.  This however is incorrect due to the fact that the conditional expectation in \eqref{eq:not:gamma} and the corresponding one in the full $\kappa$-local field equations is taken with respect to only a subset of the variables, with the subset varying for different vertices.  Thus, the \CMVE{} is not the Markovian projection of the local-field equation, and the time-marginals of the two equations do not in general coincide, although they are expected to be related  (see the discussion in Section \ref{subsub:background}). 
\end{remark}
\subsection{Well-posedness} In this section, we introduce some assumptions and discuss well-posedness of the \CMVE{}. We will frequently refer to the subset of measures in $\mPS$ with finite entropy and variance.

\begin{definition}[Set of admissible measures]
    \label{def:results:entropyVariance}
    Let $\kap, d \in \N$ with $\kap \geq 2$. We write $\mPSb$ for the set of measures $\lambda \in \mPS$ with finite entropy and variance, that is $\lambda$ satisfies $$            \bigg|\int_{(\R^d)^{1 + \kap}} \lambda(\x) \log \lambda(\x) d\x\bigg| + \int_{(\R^d)^{1 + \kap}} |\x|^2 \lambda(d\x) < \infty $$
\end{definition}

We impose the following (natural) linear growth assumptions on $\nabla U$ and $\nabla \intPot$ throughout.
\begin{assumption}[Linear growth of potentials]
    \label{as:results:lyapunov}
     The functions $U:\R^d \rightarrow \R$ and $\intPot:\R^d\rightarrow \R$ are continuously differentiable and $W$ is even (e.g., $W(x) = W(-x)$ for all $x \in \R^d$). Moreover, $\nabla U$ and $\nabla \intPot$ satisfy a linear growth condition; that is, there exists $C \in (0, \infty)$ such that
    \begin{equation*}
        \max\big\{ |\nabla U(x)|, |\nabla \intPot(x)|\big\} \leq C(1 + |x|), \quad x \in \R^d.
    \end{equation*}
\end{assumption}
\begin{remark}\normalfont\label{rk:potentials}
Assumption \ref{as:results:lyapunov} implies that there exists (a possibly different) $C \in (0, \infty)$ such that
\begin{equation}
    \max\big\{ |U(x)|, |\intPot(x)|\big\} \leq C(1 + |x|^2), \quad x \in \R^d.
    \label{eq:results:potentialGrowth}
\end{equation}
Thus, $\intPot(0) < \infty$, and since the \CMVE{} only depends on $\nabla \intPot$, we can without loss of generality set $\intPot(0) = 0$. We note here that Assumption \ref{as:results:lyapunov} does not allow for singular interaction potentials.
\end{remark}
Given our assumptions, we need only consider solutions to the \CMVE{} whose drift satisfies a linear growth condition. This assumption is natural in the context of well-posedness for non-singular SDEs (see e.g., Proposition 5.3.6 and Proposition 5.3.10 in \cite{karatzas1991stochastic}).
\begin{definition}[Linear growth solution and well-posedness]
\label{def:not:linGrowth}
Let $(\mu, \gamma)$ be a solution to the \CMVE{} on $[0, T]$ with potentials $(U, \intPot)$ satisfying Assumption \ref{as:results:lyapunov} and initial condition $\lambda \in \mPSb$. We call $(\mu, \gamma)$ a \emph{linear growth solution} to the \CMVE{} if
there exists $C \in (0, \infty)$ depending only on $( \kappa, d, T, b, \lambda)$ such that\begin{equation}
    \label{eq:res:gammaLinGrowth}
        \sup_{t \in [0, T]} \operatorname*{\bar{\mu}_t-ess\,sup}_{x, y \in \R^d} \frac{|\gamma(t, x, y)|} {1 + |x| + |y|} \leq C.
    \end{equation}
    We say that the \CMVE{} with potentials $(U, \intPot)$ is \emph{well-posed} if for all $\lambda \in \mPSb$ there is a linear growth solution to the \CMVE{} on $[0, \infty)$ with initial condition $\lambda$ that is unique among the class of linear growth solutions. 
\end{definition}

Our first result, which is proved in Section \ref{sec:wp:proof}, shows that boundedness of the interaction term $\nabla \intPot$ is sufficient for well-posedness of the \CMVE{}. Our main results on long-time behavior described in the next section assume only well-posedness and thus apply more generally to the situation where $\nabla \intPot$ may be unbounded, but the \CMVE{} is still well-posed (see \cite{hu2024gaussian} for such an example).

\begin{theorem}[Well-posedness]
    \label{thm:results:bddWp}
    Fix $\kap, d \in \N$ with $\kap \geq 2$. Suppose $(U, \intPot)$ satisfy Assumption \ref{as:results:lyapunov}, and $\|\nabla \intPot\|_{L^\infty(\R^d)} < \infty$. Then the \CMVE{} with potentials $(U, \intPot)$ is well-posed.
\end{theorem}

\section{Results on Long-Time Behavior}

\subsection{The sparse free energy and an H-theorem for the \CMVE{}}
\label{sec:results:hTheorem}
Throughout this section, fix  $\kap, d \in \N$ with $\kap \geq 2$ and $U, \intPot \in C^1(\R^d)$. Define the function $\potential: (\R^d)^{1 + \kap} \rightarrow \R$ by
\begin{equation}
    \label{eq:results:potential}
    \potential(\mathbf{x}) := U(x_0) + \frac{1}{2}\sum_{v = 1}^\kap \intPot(x_0 - x_v).
\end{equation}
Recall the definition of the edge marginal $\bar{\nu}$ from Definition \ref{def:not:marginal} and the space $\mPS$ introduced in Definition \ref{def:not:symProb}. We define the \textit{\freeEnergy{}} $\energy_\kap: \mPS \rightarrow [-\infty, \infty]$ as follows:
\begin{equation}
    \energy_\kap(\nu) := \left \{ \begin{aligned}
        &\int_{(\R^d)^{1 + \kap}} \bigg( \log \nu(\x) - \frac{\kap}{2} \log \bar{\nu}(x_0, x_1) + g(\mathbf{x})\bigg) \nu(d \mathbf{x}), & \quad & \nu\ll \text{Lebesgue,} \\
        & \infty, &\quad & \text{otherwise.}
    \end{aligned} \right.
    \label{eq:results:entropy}
\end{equation}
By  \eqref{eq:results:potentialGrowth}  and \eqref{eq:results:potential}, we have that $|\bbH_\kap(\nu)| < \infty$ whenever $\nu \in \mPSb$. We also define the \emph{modified Fisher information} $\info_\kap : \mPS \rightarrow [0, \infty]$, which plays the role of an entropy production functional, by
    \begin{equation}
        \label{eq:results:info}
        \info_\kap(\nu) := \left \{ \begin{aligned}
        &\int_{(\R^d)^{1 + \kappa}} \bigg[ \big| b(\x) + \nabla_{x_0} \log \nu(\x)\big|^2 + \kappa \bigg|\nabla_{x_1} \log \frac{\nu(\x)}{\bar{\nu}( x_0, x_1)}\bigg|^2 \bigg]\nu(d\x), & \quad & \nu\ll \text{Lebesgue,} \\
        & \infty, &\quad & \text{otherwise.}
    \end{aligned} \right.
    \end{equation}

We now present our titular result, which is an H-theorem for the \CMVE{}. 
\begin{theorem}[H-theorem]
\label{thm:results:lyap}
Fix $d, \kap \in \N$ with $\kap \geq 2$ and $\lambda \in \mPSb$, where
$\mPSb$ is 
the set of admissible measures
from Definition  \ref{def:results:entropyVariance}. 
Let $(U, \intPot)$ satisfy Assumption \ref{as:results:lyapunov} and let $\bbH_\kap$ and $\bbI_\kap$ be as defined in \eqref{eq:results:entropy} and \eqref{eq:results:info}, respectively. If $(\mu, \gamma)$ is a linear growth solution to the \CMVE{} on $[0, T]$ with potentials $(U, \intPot)$ and initial condition $\lambda$ in the sense of Definition \ref{def:not:linGrowth}, then the following energy dissipation identity, 
    \begin{equation}
        \energy_\kap(\mu_t) - \energy_\kap(\mu_r) =  - \int_r^t \info_\kap(\mu_s) ds, 
        \label{eq:results:energyIdentity}
    \end{equation}
holds for almost all $0 < r < t < T$. 
     In particular, $\lyap_\kap$ is decreasing along the measure flow $t \mapsto \mu_t$.
\end{theorem}
 The proof of Theorem \ref{thm:results:lyap} can be found in Section \ref{sec:hTheoremProof}. A key part of the proof is the \textit{a priori} regularity of linear growth solutions to \CMVE{} established in Proposition \ref{thm:ap1:wellPosed} under Assumption \ref{as:results:lyapunov}. Among other things, this guarantees that $\mu_t \in \mPSb$ whenever $\mu$ is a linear growth solution to the \CMVE{}, and therefore the left-hand side of \eqref{eq:results:energyIdentity} is finite. We emphasize that Theorem \ref{thm:results:lyap} does not require boundedness of $\nabla W$; rather it holds more generally whenever the \CMVE{} has a linear growth solution.
 
 In light of Theorem \ref{thm:results:lyap},  when $\bbH_\kap$ is bounded below it can be used to define  a Lyapunov function
 for the flow $t \mapsto \mu_t$.  To guarantee this boundedness property, we impose the following coercivity assumption on $(U, \intPot),$ which will be used repeatedly in the sequel. 
\begin{assumption}[Coercivity]
    \label{as:results:invMeasure} 
    Suppose $U, K: \R^d \rightarrow \R$ are such that there exists a measurable function $q:\R^d \rightarrow \R$ that satisfies the following properties.
    \begin{enumerate}
        \item $q$ is uniformly bounded from below, that is, $\inf_{x \in \R^d} q(x) > - \infty$, and the following quantity is finite: 
        \begin{equation}
            R_q := \int_{\R^d} e^{-q(x)} dx < \infty.
            \label{eq:results:fPartitionFunction}
        \end{equation}
        Moreover for all $p \in [1, \infty)$, we have
        \begin{equation}
            \int_{\R^d}|x|^p e^{-q(x)}dx < \infty.
            \label{eq:results:fMoments}
        \end{equation}
        \item For all $x, y \in \R^d$, we have \begin{equation}
            U(x) + U(y) + \kappa \intPot(x - y) \geq q(x) + q(y).
            \label{eq:results:fBound}
        \end{equation}
    \end{enumerate}
\end{assumption}

The condition \eqref{eq:results:fBound} can be interpreted as requiring the magnitude of the interaction $|\intPot(x-y)|$ to be growing slower than $U(x) + U(y)$ for large $x, y \in \R^d$.

We also introduce the following strengthening of Assumption \ref{as:results:invMeasure} which will imply additional properties of the level sets of the \freeEnergy{} $\mathbb{H}_\kap$.
\begin{assumptionPrime}[Strong coercivity]
\label{as:results:strongCoercivity}
   Suppose $(U, \intPot)$ satisfy Assumption \ref{as:results:lyapunov} and Assumption \ref{as:results:invMeasure}, and for $q$ as in Assumption \ref{as:results:invMeasure}, there exists $\tilde{C}_q \in (0, \infty)$ such that
    \begin{equation}
        \label{eq:zeros:strongCoercivity}
        U(x) + U(y) + \kap \intPot(x - y) - q(x) - q(y) \geq \tilde{C}_q(|x|^2 + |y|^2), \quad x, y \in \R^d.
    \end{equation}
\end{assumptionPrime}

We show in the following proposition, which is proved in Section \ref{sec:levelSets}, that Assumption \ref{as:results:invMeasure} ensures that $\energy_\kap$ is uniformly bounded from below on $\mPSb$ and additionally has compact level sets under Assumption \ref{as:results:strongCoercivity}.

\begin{proposition}[Lower bound and level sets of $\bbH_\kap$] Suppose $(U, W)$ satisfy Assumption \ref{as:results:invMeasure}. 
\label{prop:results:energyBound}
    \label{lem:conv:compactLevelSets}
Then, with $R_q$ as in \eqref{eq:results:fPartitionFunction}, we have
\begin{equation}
   \energy_\kap^* :=  \inf_{\nu\in \mPSb} \energy_\kap(\nu) \geq - \log R_q,
    \label{eq:results:lowerBound}
\end{equation} and $\bbH_\kap$ is lower semicontinuous on $\mPSb$. If $(U, W)$ additionally satisfy Assumption \ref{as:results:strongCoercivity}, then for all $M \in \R$ we have
    \begin{equation}
            \mathcal{R}(M):=\{\nu \in \mPS: \mathbb{H}_\kap(\nu) \leq M\} \subset \mPSb,
            \label{eq:conv:levelsets}
    \end{equation}
and $\mathcal{R}(M)$ is compact in the weak topology. Moreover there exists $C_{q,M} \in (0, \infty)$ such that
     \begin{equation}
        \sup_{\nu \in \mathcal{R}(M)}\bigg\{  \int_{(\R^d)^{1 + \kap}} |\x|^2 \nu(d\x) \bigg\} \leq C_{q, M}.
        \label{eq:conv:secondMoment}
    \end{equation}

\end{proposition}

\subsection{Stationary distributions} \label{sec:results:stationary} Our next series of results use the H-theorem to describe the connection between stationary distributions of the \CMVE{} and zeros of $\mathbb{I}_\kap$, which we denote by
\begin{equation}
\label{eq:results:mE}
    \mS_\kap := \{ \nu \in \mPS : \mathbb{I}_\kap(\nu) = 0\}.
\end{equation}
In analogy with the classical McKean-Vlasov setting, (e.g., see \cites{carrillo2003kinetic, guillin2022lsi}), we show that the H-theorem allows one to identify stationary distributions of the \CMVE{}. In particular, Theorem \ref{thm:results:lyap} suggests that $\mS_\kap$ can be interpreted as the set of critical points of the evolution of the sparse free energy $\lyap_\kap$ along the measure flow of the \CMVE{}. We establish this rigorously in Theorem \ref{thm:results:stationary} below. First, we define stationary distributions of the \CMVE{}.

\begin{definition}[Stationary distributions of the \CMVE{}]
    \label{def:results:stationary}
We say $\nu \in \mPS$ is a \emph{stationary distribution} of the \CMVE{} with potentials $(U, \intPot)$ if $\nu \in \mPSb$ and there exists a linear growth solution $(\mu^\nu, \gamma)$ to the \CMVE{} on $[0, \infty)$ with potentials  $(U, \intPot)$ and initial condition $\nu$ such that
    \begin{equation}
    \label{eq:results:stationary}
        \mu^\nu_t = \nu,\quad t \geq 0.
    \end{equation}
\end{definition}

We will utilize the following assumption, which assumes that elements of $\mS_\kap$ satisfy a linear growth condition in the spirit of Definition \ref{def:not:linGrowth}. We discuss in Remark \ref{rk:results:assumptions} many natural settings where Assumption \ref{as:results:fixedPoint} is satisfied.
\begin{assumption}[Linear growth for elements of $\mS_\kap$]    \label{as:results:fixedPoint}
    Suppose $(U, \intPot)$ satisfy Assumption \ref{as:results:lyapunov} and \ref{as:results:invMeasure} and that for all $\nu \in \mS_\kap$, the associated edge marginal $\bar{\nu}$ satisfies the following linear growth condition for some $C \in (0, \infty)$:\begin{equation*}
        \bigg|\int_{\R^d} \nabla \intPot(x - y) \bar{\nu}(x | y) dx \bigg| \leq C(1 + |y|), \quad \text{a.e.-}y \in \R^d. 
    \end{equation*} 
\end{assumption}

We now state the main theorem of this section, which shows that $\mS_\kap$ coincides with the set of stationary distributions of the \CMVE{}. The proof, which can be found in Section \ref{ss:statProof}, leverages a connection  identified in Theorem \ref{thm:results:zeros} between $\mS_\kap$ and a certain fixed point equation (Definition \ref{def:results:CBP}).  
\begin{theorem}[Zeros of $\info_\kap$]
    \label{thm:results:stationary}
    Let $(U, \intPot)$ satisfy Assumption \ref{as:results:fixedPoint}. Suppose $\nu \in \mPS$ is such that there exists a linear growth solution to the \CMVE{} on $[0, \infty)$ with potentials $(U, \intPot)$ and initial condition $\nu$.
    Then $\nu$ is a stationary distribution of the \CMVE{} if and only if $\nu$ lies in the set $\mS_\kap$ defined in \eqref{eq:results:mE}.  
\end{theorem}

\begin{remark}
\label{rk:results:assumptions}
    \normalfont The following are examples of $(U, \intPot)$ that satisfy Assumptions \ref{as:results:invMeasure}, \ref{as:results:strongCoercivity}, and\ref{as:results:fixedPoint}.
    \begin{enumerate}
    \item Assumption \ref{as:results:invMeasure} is satisfied if $U$ is superlinear and $\intPot$ is bounded below. 
        \item If $U$ grows quadratically at infinite, that is there exists $c, R \in (0, \infty)$ such that for all $|x| > R$ we have
         $   U(x) \geq c|x|^2,$
        and $\|\nabla \intPot \|_{L^\infty(\R^d)} < \infty$, we have that both Assumption \ref{as:results:strongCoercivity} and Assumption \ref{as:results:fixedPoint} are satisfied.
        \item If $U(x) \geq c_1|x|^2$ and $\kap|\intPot(x)| \leq c_2 |x|^2$ for $c_1 > 2c_2$, then Assumption \ref{as:results:strongCoercivity} is satisfied.
        \item When $d = 1$, inspection of the proof of Theorem 1.10 in \cite{lacker2023stationary} shows that if $(U, \intPot)$ are even, twice continuously differentiable, and
        \begin{equation*}
            \inf_x U''(x) > \kap\big( \|\intPot''\|_{L^\infty(\R)} - \inf_x \intPot''(x)\big),
        \end{equation*}
        then Assumption \ref{as:results:fixedPoint} is satisfied. If, in addition, $U$  grows quadratically at infinity, then Assumption \ref{as:results:strongCoercivity} is also satisfied. 
    \end{enumerate}
\end{remark}

Next, we obtain in Theorem \ref{thm:results:convergence} below the convergence of solutions of the \CMVE{} to $\mS_\kap$.  Its proof is relegated to
Section \ref{sec:convergenceProofs}. 

\begin{theorem}[Convergence to stationary distributions]
\label{thm:results:convergence}
      Let $(U, \intPot)$ satisfy Assumption  \ref{as:results:strongCoercivity} and Assumption \ref{as:results:fixedPoint}, and $\lambda \in \mPSb$ satisfy $\lyap_\kap(\lambda) < \infty$.       Let $(\mu, \gamma)$ be a linear growth solution to the \CMVE{} on $[0, \infty)$ with potentials $(U, \intPot)$ and initial condition $\lambda$.
      Then we have
    \begin{equation}
    \label{eq:results:convergenceToEkap}
        \lim_{t \rightarrow \infty}d_{LP}(\mu_t, \mS_\kap) = 0,
    \end{equation}
    where $d_{LP}$ is the Levy-Prokhorov metric.
\end{theorem}

\begin{remark}
  \label{rem-Lyapunov}
 {\em Theorem \ref{thm:results:lyap}, Proposition \ref{prop:results:energyBound}, and Theorem \ref{thm:results:convergence} together imply that $\nu \mapsto \lyap_\kap(\nu) - \bbH_\kap^*$ is a strong global Lyapunov function for the \CMVE{} measure flow  (namely a non-negative functional on $\mPS$ that is continuous and  decreasing along the measure flow for any admissible initial condition and strictly decreasing outside stationary points; see \cite{cazenave1998semilinear}). It is worth emphasizing that this property holds even when the $\kappa$-MLFE admits multiple stationary distributions.}
  \end{remark}

\subsection{Continuous Gibbs measures and \FPES{}}
\label{sec:results:CBP}
An important ingredient in the proof of Theorem \ref{thm:results:stationary} is the following fixed point problem whose solutions we refer to as  \emph{\FPES{}}. 
As mentioned in the introduction,
these fixed points are marginals of certain continuous Gibbs measures that are stationary distributions of interacting diffusions
on the $\kap$-regular tree $\bbT_\kap$.
We show in Theorem \ref{thm:results:zeros} that the set of \FPES{} coincides with $\mS_\kap$. 

\begin{definition}[\FPES{}]
\label{def:results:CBP}
Let $(U, \intPot)$ satisfy Assumption \ref{as:results:lyapunov} and Assumption \ref{as:results:invMeasure}. We say $\nu \in \mPS$ is a \emph{\FPE{}} if $\nu$ is of the form \begin{equation}
        \nu(\x) = \nu_0(x_0)\prod_{i = 1}^\kap \bar{\nu}(x_i | x_0),
        \label{eq:zeros:1MRF}
    \end{equation}
    where $\nu_0 \in \mP(\R^d)$ is absolutely continuous and solves the following fixed point equation
    \begin{equation}
        \label{eq:zeros:fixedPoint}
        \nu_0(x)^{\frac{1}{\kap}} = \frac{1}{\mathcal{Z}_{\nu_0}} e^{-\tfrac{1}{\kap}U(x)}\int_{\R^d} e^{-\intPot(x-y) - \frac{1}{\kap}U(y)}\nu_0(y)^{\frac{\kap-1}{\kap}} dy, \quad \nu_0\text{-a.e. }x \in \R^d,
    \end{equation}
   where the normalizing constant is given by 
    \begin{equation}
        \label{eq:zeros:zMu}
        \mathcal{Z}_{\nu_0} := \int_{\R^d \times \R^d} \exp\Big( - \tfrac{U(x) + U(y)}{\kap} - \intPot(x -y)\Big) \big[\nu_0(x)\nu_0(y)\big]^{\frac{\kap-1}{\kap}} dx dy,
    \end{equation}
    and $\bar{\nu}(x|y) := \bar{\nu}(x, y) / \nu_0(y)$ for $\nu_0$-a.e. $x, y \in \R^d$ where the edge marginal $\bar{\nu}(x, y)$ of $\nu$ is mutually absolutely continuous with respect to $\nu_0^{\otimes 2}$ and is of the form
    \begin{equation}
        \label{eq:zeros:muBar}
        \bar{\nu}(x,y) = \frac{1}{\mathcal{Z}_{\nu_0}}\exp\Big( - \tfrac{U(x) + U(y)}{\kap} - \intPot(x -y)\Big) \big[\nu_0(x)\nu_0(y)\big]^{\frac{\kap-1}{\kap}}, \quad \nu_0^{\otimes 2}\text{-a.e. }(x, y) \in \R^d \times \R^d.
    \end{equation}
\end{definition}
\begin{remark}
\label{rk:results:FPE}
\normalfont
    Note that by \eqref{eq:zeros:zMu}-\eqref{eq:zeros:muBar},  Assumption \ref{as:results:invMeasure} and H\"older's inequality, the normalizing constant $\mathcal{Z}_{\nu_0}$ defined in \eqref{eq:zeros:zMu} is finite for all $\nu_0 \in \mP(\R^d)$. Furthermore, by \eqref{eq:zeros:1MRF}-\eqref{eq:zeros:muBar} we have 
    \begin{equation}
        \nu(\x) = \frac{1}{\mathcal{Z}_{\nu_0}^\kap} \exp \bigg( - U(x_0) - \sum_{v = 1}^\kap \intPot(x_0 - x_v) \bigg) \prod_{v = 1}^\kap e^{-\frac{1}{\kap}U(x_v)}\nu_0(x_v)^{ \frac{\kap-1}{\kap}}, \quad \nu_0^{\otimes (1 + \kap)}\text{-a.e. }\x \in (\R^d)^{1 + \kap},
        \label{eq:zeros:altCBP}
    \end{equation}
    with $\nu_0$ being a solution to the fixed point equation \eqref{eq:zeros:fixedPoint}.  
    Through \eqref{eq:zeros:1MRF} and \eqref{eq:zeros:muBar}, we see that $\nu$ is entirely determined by $\nu_0$ and \eqref{eq:zeros:zMu} follows from \eqref{eq:zeros:fixedPoint} if $\nu_0$ is a probability measure. In view of this observation, we will sometimes abuse terminology (also as usual conflating the measure $\nu_0$ with its density) and say $\nu_0 \in \mP(\R^d)$ is a \FPE{} when $\nu_0 \in \mP(\R^d)$ satisfies \eqref{eq:zeros:fixedPoint}-\eqref{eq:zeros:zMu}.
\end{remark}

Our next theorem identifies \FPES{} with the zeros of $\bbI_\kap$. Its proof is given in Section \ref{sec:stationaryProofs}.

\begin{theorem}[Characterization of \FPES{}]
\label{thm:results:zeros}
Fix $d, \kap \in  \N$ and $\kap \geq 2$. Let $(U, \intPot)$ satisfy Assumption \ref{as:results:lyapunov} and Assumption \ref{as:results:invMeasure}. Let $\mS_\kap$ be as in \eqref{eq:results:mE}. Then $\nu \in \mS_\kap$ if and only if $\nu \in \mPS$ is a  \FPE{} in the sense of Definition \ref{def:results:CBP}. \end{theorem}

\begin{remark}[Relationship to \cite{lacker2023stationary}]
\label{rk:lackerZhang}\normalfont
We show in Proposition \ref{lem:zeros:entropyAndMoments} that under Assumption \ref{as:results:fixedPoint}, a  \FPE{} satisfies many nice properties. We also show in Corollary \ref{cor:LackerZhang} that the \FPE{} equation is equivalent to the fixed point problem described in Definition 1.3 of \cite{lacker2023stationary} (or its natural analog in the case $d > 1$). Together with Theorems 1.4-1.6 of \cite{lacker2023stationary}, this shows that the  \FPES{} are the marginals of continuous Gibbs measures on regular trees that are automorphism invariant stationary distributions of interacting diffusions on the infinite regular tree.  Moreover, existence and uniqueness of \FPES{} follows from the conditions found in Theorem 1.9 and 1.10 of \cite{lacker2023stationary}. For instance, Theorem 1.9 of \cite{lacker2023stationary} is used in Theorem \ref{thm:results:uniqueRegime} below to show that there exists a unique solution to the \FPE{} when $\kap = 2$.
\end{remark}

\subsection{$\kappa = 2$: Renormalized entropy and exponential convergence}\label{sec:results:kap2} In this section we summarize the additional results that we obtain for the special case $\kappa = 2$. In Theorem \ref{thm:results:uniqueRegime} we derive an alternative representation of the sparse free energy as the difference of relative entropies. Furthermore, we show in Theorem \ref{thm:results:LSI} that under additional assumptions, we have a uniform-in-$n$ logarithmic Sobolev inequality for a certain Gibbs measure (see Definition \ref{def:results:gibbs}) and exponential decay of the sparse free energy along the \CMVE{} measure flow. 

\subsubsection{The sparse free energy as the limit of renormalized entropies} We start by defining a sequence of line graphs $\{\bbT_2^n\}_{n \geq 1}$ that approximates (or more precisely, converges locally to) the $2$-regular tree $\bbT_2$.

\begin{definition}[Truncated 2-trees]
    For $n \in \N$, define the \emph{truncated 2-tree} $\bbT_2^n := (V_n, E_n)$  with
    \begin{align*}
        V_n &:= \{-n, \ldots, n\}, \quad E_n := \{(v, v+1) : v = -n, \ldots, n-1\}.
    \end{align*}Here, we have $\x = (x_{-1}, x_0, x_1) \in (\R^d)^3$, and we let $\x^{(n)} := (x_{-n}, \ldots, x_n)$ denote a vector in $(\R^d)^{V_n}$.
    \label{def:results:Tn2}
\end{definition}

Our main insight is the identification of the sparse free energy as the limit of renormalized relative entropies in Theorem \ref{thm:results:2entropy}.  A similar, though somewhat simpler,  procedure has been effectively utilized in the mean-field setting, where a Lyapunov function for the mean-field McKean-Vlasov equation was constructed by taking the limit of renormalized relative entropies with respect to a finite-dimensional Gibbs measure (see e.g., Lemma 17 of \cite{guillin2022lsi} and also \cites{budhiraja2015entropy, budhiraja2015kolmogorov} for the case of finite-state pure jump processes). 
\begin{definition}[Gibbs measure]
\label{def:results:gibbs}
    For $n \in \N$, $n > 2$, define the \emph{Gibbs measure} $\theta^n \in \mP((\R^d)^{V_n})$ by
\begin{equation}
    \theta^n(d\x^{(n)}) := \frac{1}{\mathcal{Z}^n} \exp\bigg( - \frac{1}{2}\sum_{(u, v) \in E_n} \newW(x_u, x_v)\bigg)  d\x^{(n)},
    \label{eq:results:gibbs}
\end{equation}
where $\newW$ is given by  
\begin{equation}
    \newW(x, y) := U(x) + U(y) + 2\intPot(x - y), \quad x, y \in \R^d,
    \label{eq:results:W}
\end{equation}
and $\mathcal{Z}^n$ is the associated normalization constant or \emph{partition function}:
\begin{equation}
    \mathcal{Z}^n := \int_{(\R^d)^{V_n}} \exp\bigg( - \frac{1}{2}\sum_{(u, v) \in E_n} \newW(x_u, x_v)\bigg) d\x^{(n)}.
    \label{eq:results:gibbsPFN}
\end{equation}
Note that if $(U, \intPot)$ satisfy Assumption \ref{as:results:invMeasure}, then $\mathcal{Z}^n$ is finite and $\theta^n$ has finite moments of all orders. 
\end{definition}

The second ingredient in our renormalization procedure is a reconstruction of the finite-dimensional particle system from a given marginal distribution. The results of \cite{lacker2021MRF} suggest that there should be a natural 2nd-order Markov random field (2-MRF) structure for interacting diffusions on sparse graphs. Hence, we define the following lift map.

\begin{definition}[Lift map]
    \label{def:results:2measures}
    Let $\nu \in \mathcal{M}_{2, d}$. Define the \emph{lift map} to be the absolutely continuous measure $\psi_\nu^n \in  \mP((\R^d)^{V_n})$ with density
\begin{equation}
    \label{eq:entropy:2reconstruction}
    \psi_\nu^n\big(\x^{(n)}\big) := \prod_{v = -n+1}^{n-1} \nu(x_{v-1}, x_v, x_{v+1}) \prod_{v = -n}^{n-1} \frac{1}{\bar{\nu}(x_v, x_{v+1})}, \quad \x^{(n)} \in (\R^d)^{V_n}
\end{equation}
We note here that any empty products are interpreted as 1.
\end{definition}
    The lift $\psi_\nu^n$ should be thought of as the natural 2-MRF on $\bbT_{2}^n$ with fixed neighborhood marginals $\nu$. See Lemma \ref{lem:entropy:symmetry} for a precise statement. The following theorem, which is proved in Section \ref{sec:entropyRenormalization}, shows that the suitably normalized log partition function of $\theta^n$ converges and identifies the sparse free energy as the limit of renormalized relative entropies. 
\begin{theorem}[Entropy renormalization]
\label{thm:results:2entropy}
    Suppose $(U, \intPot)$ satisfy Assumption \ref{as:results:invMeasure}. Recall the definition of $\mathcal{Z}^n$ from \eqref{eq:results:gibbsPFN}. There exists a finite constant $\bbH_2^* \in \R$ such that
    \begin{equation}
         \bbH_2^* = \lim_{n \rightarrow \infty} \frac{1}{2n+1} \log \mathcal{Z}^n.
        \label{eq:results:Hstar}
    \end{equation}For all $\nu \in \mathcal{Q}_{2, d}$, we have
    \begin{equation}
        \lim_{n \rightarrow \infty} \frac{1}{2n+1} \mathcal{H}(\psi^n_\nu | \theta^{n}) = \energy_2(\nu) - \bbH_2^*,
        \label{eq:results:2convergence}
    \end{equation}
\end{theorem}

Under an additional condition  (Assumption \ref{as:results:LSI}) we establish uniqueness of the stationary distribution of the $2$-MLFE and also identify the constant $\bbH_2^*$ in Theorem \ref{thm:results:2entropy}. We postpone the technical statement of this condition to Section \ref{sec:LSI:LSI}.

\begin{theorem}[Uniqueness of stationary distributions]\label{thm:results:uniqueRegime} Suppose $(U, \intPot)$ satisfy Assumption \ref{as:results:fixedPoint} and Assumption \ref{as:results:LSI}. Then there exists a unique Cayley fixed point $\pi \in \tilde{\mM}_{2, d}$ and we have $\bbH_2^* = \bbH_2(\pi)$. In other words, for $\mathcal{Z}^n$ as in \eqref{eq:results:gibbsPFN}, we have
\begin{equation*}
    \lim_{n \rightarrow \infty} \frac{1}{2n+1} \log \mathcal{Z}^n = \bbH_2(\pi).
\end{equation*}
Furthermore, we have \begin{equation}
        \bbH_2(\nu) - \bbH_2(\pi) = \mH(\nu|\pi) - \mH(\bar{\nu}|\bar{\pi}), \quad \nu \in \mathcal{Q}_{2, d}.
        \label{eq:results:sfeForm}
    \end{equation}
Moreover, if there exists a solution to the 2-MLFE with potentials $(U, \intPot)$ and initial distribution $\pi$, then $\pi$ is the unique stationary distribution of the 2-MLFE in the sense of Definition \ref{def:results:stationary}.
\end{theorem}

\subsubsection{Logarithmic Sobolev inequalities and exponential convergence} Under the extra condition (Assumption \ref{as:results:LSI}) we also stablish a uniform-in-$n$ logarithmic Sobolev inequality for $\theta^n$. Consequently we also obtain a modified (nonlinear) log-Sobolev inequality for $\lyap_2$ and $\info_2$, and exponential convergence of $\bbH_\kap$ under the \CMVE{} flow. Assumption \ref{as:results:LSI} is reminiscent of (though slightly weaker than) the conditions of Theorem 9 of \cite{guillin2022lsi} for a uniform log-Sobolev inequality for mean-field systems, and sufficient conditions include strong convexity of the potentials (see Remark \ref{rk:LSI:convexity}).

\begin{theorem}[A modified logarithmic Sobolev inequality]
\label{thm:results:LSI}
Suppose $(U, \intPot)$ satisfy Assumption \ref{as:results:LSI}. Then the following properties are satisfied:
\begin{enumerate}
    \item The family of Gibbs measures $\{\theta^n\}_{n \geq 1}$ satisfies a uniform log-Sobolev inequality. That is, there exists $C_\theta \in (0, \infty)$ such that for all $n \in \N$, we have
    \begin{equation}
        \label{eq:LSI:preLimitLSI}
        \entropy(\Lambda | \theta^{n}) \leq 2 C_\theta \mI(\Lambda | \theta^{n}), \quad \Lambda \in \mP\big((\R^d)^{V_n}\big).
    \end{equation}
    \item  There exists $C_0  \in (0, \infty)$ such that for all $\nu \in \tilde{\mM}_{2, d}$ satisfying \begin{equation}
    \int_{(\R^d)^{3}} \Big(|b(\x)|^2 + \big|\nabla_{\x} \log \nu(\x)\big|^2\Big) \nu(d\x) < \infty,
     \label{eq:LSI:finiteFI}
\end{equation}we have the following modified log-Sobolev inequality: \begin{equation}
        \energy_2(\nu) - \bbH_2^* \leq C_0 \info_2(\nu),
        \label{eq:results:LSI}
    \end{equation}
   where $\bbH_\star$ is as defined in \eqref{eq:results:Hstar}.
    \item If $\lambda \in \tilde{\mM}_{2, d}$ and $(\mu, \gamma)$ is a linear growth solution to the $2$-MLFE on $[0, \infty)$, with initial condition $\lambda$, then there exists $c, C \in (0, \infty)$ such that the following inequality holds: 
        \begin{equation}
    \label{eq:results:expConv}
        \energy_2(\mu_t) - \bbH_2^* \leq C \exp(-ct), \quad \text{for a.e. } t \geq 0.
    \end{equation}
\end{enumerate}
\end{theorem}

    \begin{remark}[2-MRF is necessary]
    \normalfont
    \label{rk:2MRFNec}
     Given that the stationary distribution of the $n$-particle distribution is a 1-MRF, one may wonder why $\psi_\nu^n$ is chosen to form a 2-MRF rather than a 1-MRF. Define the 1-MRF lift of $\bar{\nu}$ by \begin{equation*}
        \phi^{n}_{\bar{\nu}}(\x^{(n)}) := \nu_0(x_0)\prod_{v = 0}^{n-1} \bar{\nu}(x_{v+1}|x_{v})\bar{\nu}(x_{-v-1}|x_{-v}), \quad \x^{(n)} \in (\R^d)^{V_n}.
    \end{equation*}
    By the same argument as in the proof of Theorem \ref{thm:results:2entropy}, it can be shown that 
    \begin{equation*}
    \begin{aligned}
    \frac{1}{2n+1} \entropy \big( \phi^{n}_{\bar{\nu}}\big| \theta^{n}\big) \rightarrow \hat{\energy}_2(\bar{\nu}) - {\mathbb{H}}_*,
    \end{aligned}
    \end{equation*}
    where $\hat{\bbH}_2$ is given by
    \begin{equation*}
        \hat{\energy}_2(\bar{\nu}) := \int_{\R^d \times \R^d} \bigg( U(x_0) + \intPot(x_0 - x_1) + \log \bar{\nu}(x_0, x_1) - \log \nu_0(x_0) \bigg) \bar{\nu}(dx_0, dx_1).
    \end{equation*}
    However, this limiting functional  need not decrease along the $\kappa$-MLFE measure flow, that is, if ${\mu}$ is a solution to the \CMVE{},  the inequality  $\tfrac{d}{dt}\hat{\energy}_2(\bar{\mu}_t) \leq 0$ may not hold for a.e. $t \geq 0$. In particular, we show through a numerical example that this is not the case in Figure \ref{fig:entropy:1MRFv2MRF}. This is because $\hat{\energy}_2$ only sees the $2$-particle marginal $\bar{\nu}$ and not the full distribution $\nu$. Therefore $\hat{\energy}_2$ may not be decreasing along the trajectories of the \CMVE{} when the initial distribution is far from a 1-MRF, which is exactly the case in our numerical example.
    \begin{figure}[h]
        \centering
        \includegraphics[scale = 0.5]{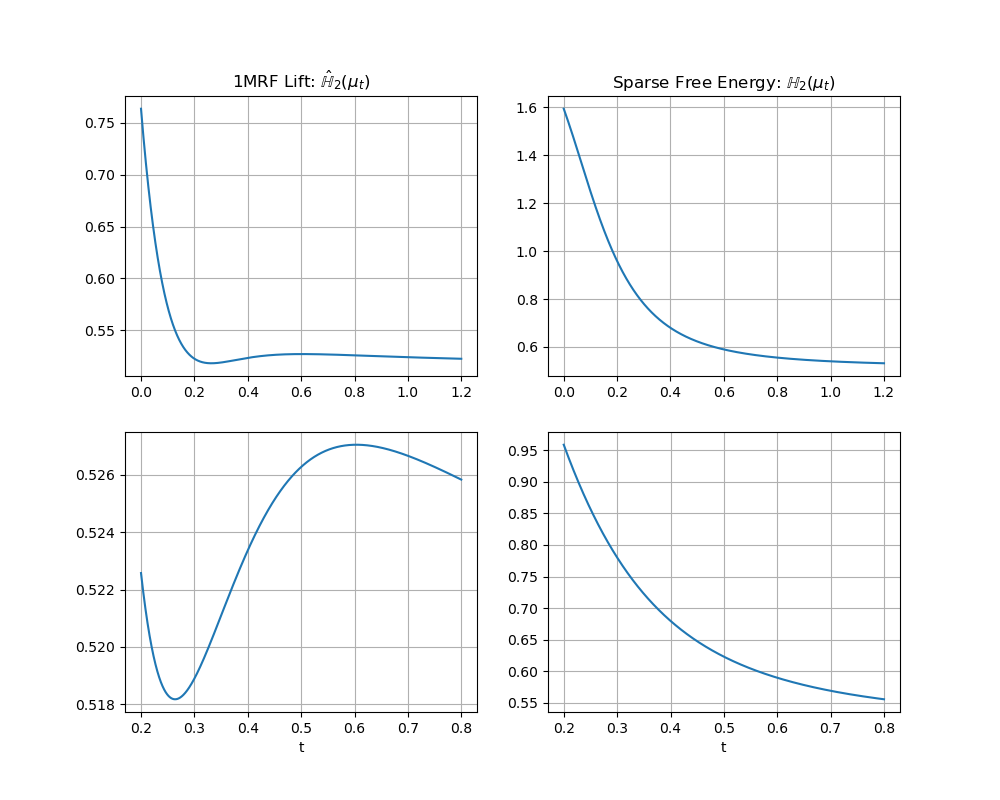}
        \caption{A comparison between $\hat{\energy}_2(\bar{\mu}_t)$ and $\energy_2(\mu_t)$. Here, $\mu_t$ solves the $2$-MLFE with $d = 1$, $\kap = 2$, potentials $U(x) = 7x^2/4$ and $K =-3x^2/8 $, and an initial condition that is not a 1-MRF. The left column shows  the evolution of the 1-MRF renormalized limit   
        $\hat{\bbH}_2(\mu_t)$ and the right column shows the evolution of the sparse free energy $\bbH_2(\mu_t)$. For both columns, the bottom figure shows the top figure zoomed in on the time interval $(0.2, 0.8)$.
        We see that $\energy_2(\mu_t)$ is always decreasing in time while  $\hat{\energy}_2(\bar{\mu}_t)$ is not.}
        \label{fig:entropy:1MRFv2MRF}
    \end{figure}
    \end{remark}

\section{Well-posedness of the \CMVE{}}

\subsection{Preliminary tools}
We first summarize two results from the literature that will be used in the proof of Theorem \ref{thm:results:bddWp}. We start with a weighted Csizar-Kullback-Pinsker inequality. 
\begin{lemma}[Weighted CKP inequality]
    \label{lem:wp:pinsker}
    Let $\mathcal{X}$ be a Polish space, and $\nu$ and $\nu'$ be probability measures in $\mP(\mathcal{X})$. For any non-negative, measurable function $f: \mathcal{X} \rightarrow \R^d$, we have 
    \begin{equation*}
        \big|\bbE^\nu[f(Y)] - \bbE^{\nu'}[f(Y)]\big|^2 \leq 2\bigg(1 + \log \int_{\mathcal{X}}e^{|f|^2} d\nu'\bigg) \mathcal{H}(\nu|\nu').
    \end{equation*}
\end{lemma}
\begin{proof}  
    Note that in the notation of \cite{villani2005pinsker} we have
    \begin{align*}
        \big|\bbE^\nu[f(Y)] - \bbE^{\nu'}[f(Y)]\big|^2 = |\langle \nu - \nu', f \rangle|^2  \leq  \| f(\nu - \nu')\|^2_{TV}.
    \end{align*}
    The claim is then an immediate consequence of Theorem 2.1(ii) of \cite{villani2005pinsker}. 
\end{proof}
We now state a well known entropy estimate for the laws of diffusion processes. Recall that for a path $\omega:[0, T] \rightarrow \R^m$ and $t \in [0, T]$ we write $\omega[t] := \{\omega(s)\}_{s \in [0, t]}$ for the trajectory of $\omega$ on $[0, t]$.
\begin{lemma}[Lemma 3.5 and Remark 3.6 of \cite{conforti2023projected}]
\label{lem:wp:entropyEstimates}
Let $m \in \mathbb{N}$ and $\lambda_0 \in \mathcal{P}(\R^m)$. Suppose $b^1, b^2:[0, T] \times \R^m \rightarrow \R^m$ are measurable and satisfy a linear growth condition, that is there exists $C \in (0, \infty)$ such that
\begin{equation*}
    \sup_{(t, x) \in [0, T] \times \R^d} |b^i(t, x)| \leq C(1 + |x|), \quad i = 1, 2.
\end{equation*}
For $i = 1, 2$, suppose $(\Omega^i, \mathcal{F}^i, \mathbb{F}^i, \mathbb{P}^i)$ is a filtered probability space supporting an $m$-dimensional Brownian motion $B^i$ and a continuous $m$-dimensional $\mathbb{F}^i$-adapted process $Z^i$ satisfying
\begin{equation*}
    dZ^i(t) = b^i(t, Z^i(t)) dt + \sqrt{2}dB^i(t),
\end{equation*}
with $\mL(Z^1(0)) = \mL(Z^2(0)) = \lambda_0$. Then the following identity holds: 
\begin{equation*}
    \mathcal{H} \big( \mathcal{L}(Z^1[T])\,|\, \mathcal{L}(Z^2[T]) \big) = \frac{1}{2} \mathbb{E}^{\mathbb{P}^1} \Bigg[ \int_0^T |b^1(t, Z^1(t)) - b^2(t, Z^1(t))|^2 dt\Bigg]. 
\end{equation*}
\end{lemma}

Next, we present two results that will be used frequently in the sequel. The first one shows that linear growth solutions to the CMVE possess nice regularity and integrability properties.
\begin{proposition}[Properties of linear growth solutions] \label{thm:ap1:wellPosed}
Fix $\kappa, d \in \N$ and $T \in (0, \infty)$. Suppose $\lambda \in \mPSb$ and $(U, \intPot)$ satisfy Assumption \ref{as:results:lyapunov}. Suppose that $(\mu, \gamma)$ is a linear growth solution to the \CMVE{} with potentials $(U, \intPot)$ and initial condition $\lambda$. Then the following properties hold: 
\begin{enumerate}
    \item The time-marginals of $\mu$ have uniformly bounded second moments, that is
\begin{equation*}
    \sup_{t \in [0, T]} \int_{(\R^d)^{1 + \kap}} |\x|^2 \mu_t(d\x) < \infty;
\end{equation*}
        \item There exists a positive locally H\"older continuous function $\mu:[0, T] \times (\R^d)^{1 + \kap} \rightarrow (0, \infty)$ such that $\mu_t(d\x) = \mu_t(\x) d\x$. Moreover, the H\"older coefficient and exponent of $\mu$ depends only on $(\kap, d, T)$, the linear growth constant of $(\gamma, \nabla U, \nabla \intPot)$, and the initial condition $\lambda$.
    \item  We have $\mu_t \in W^{1, 1}((\R^d)^{1 + \kap})$ for almost every $t \in (0, T)$. Moreover the following estimate holds:
\begin{equation*}
    \int_0^T \int_{(\R^d)^{1 + \kap}} \frac{|\nabla_{\x} \mu_t(\x)|^2}{\mu_t(\x)} d\x dt < \infty.
\end{equation*}
\item For all $t \in [0, T]$, we have
\begin{equation*}
    \int \mu_t(\x) |\log \mu_t(\x)| d\x < \infty.
\end{equation*}
\end{enumerate}
In particular, if $\lambda \in \mPSb$ and $(\mu, \gamma)$ is a linear growth solution of the \CMVE{} on $[0, T]$ with initial condition $\lambda$, then $\mu_t \in \mPSb$ for all $t \in [0, T]$.
\end{proposition}
\begin{proof}
     By \eqref{eq:not:SimpleCMVE} and Assumption \ref{as:results:lyapunov}, if $\gamma$ satisfies \eqref{eq:res:gammaLinGrowth} then $\mu$ solves a linear Fokker-Planck equation with a linear growth drift. The proposition is then an immediate corollary of Theorem \ref{thm:ap1:wellPosed2}. 
\end{proof}
The final lemma of this section characterizes the edge marginals $\{\bar{\mu}_t\}_{t \geq 0}$ of solutions $(\mu, \gamma, \X)$ to the \CMVE{} as solutions to a coupled pair of SDEs.
\begin{lemma}[Marginal \CMVE{}]
\label{lem:wp:marginalSDE}
     Suppose $(U, \intPot)$ satisfy Assumption \ref{as:results:lyapunov} and $(\mu, \gamma, \X)$ is a linear growth solution to \CMVE{} on $[0, T]$ with potentials $(U, \intPot)$  and initial condition $\lambda \in \mPSb$. Then for all $t \in [0, T]$, the edge time marginal $\bar{\mu}_t$ satisfies $\bar{\mu}_t = \mL(\bar{X}_0(t), \bar{X}_1(t))$, where $(\bar{X}_0, \bar{X}_1)$ solve the following SDE with initial condition $\bar{\lambda}$:
        \begin{equation}
        \begin{aligned}
            d{\bar{X}}_0(t) &= \gamma(t, \bar{X}_0(t), \bar{X}_1(t)) dt + \sqrt{2}d\tilde{B}_0(t), \\
            d{\bar{X}}_1(t) &= \gamma(t, \bar{X}_1(t), \bar{X}_0(t)) dt + \sqrt{2}d\tilde{B}_1(t), 
                    \end{aligned}
            \label{eq:wp:marginalSDE}
    \end{equation}
    where $\tilde{B}_0, \tilde{B}_1$ are independent $d$-dimensional Brownian motions. Moreover, we have
    \begin{equation*}
    \int_0^T \int_{\R^d \times \R^d} \frac{|\nabla_{(x_0, x_1)} \bar{\mu}_t(x_0, x_1)|^2}{\bar{\mu}_t(x_0, x_1)} dx_0 dx_1 dt < \infty.
\end{equation*}
\end{lemma}

\begin{proof}
        Since $(\mu, \gamma, \X)$ is a linear growth solution to the \CMVE{}, by \eqref{eq:res:gammaLinGrowth}, and Propositions 5.3.6 and 5.3.10 of \cite{karatzas1991stochastic}, the SDE \eqref{eq:wp:marginalSDE} has a unique weak solution. Let $(X_0, X_1)$ be distributed according to the $(0, 1)$-marginal of $\mu$. Moreover, by Assumption \ref{as:results:lyapunov}, \eqref{eq:results:drift}, \eqref{eq:res:gammaLinGrowth}, and Proposition \ref{thm:ap1:wellPosed}(1), we have 
            \begin{equation*}
        \int_0^T \mathbb{E} \Big[ |b(\X_t)|^2 + |\gamma(t, X_1(t), X_0(t)|^2 + |\gamma(t, X_0(t), X_1(t))|^2\Big] dt < \infty.
    \end{equation*}
    Therefore we can combine Corollary 3.7 of \cite{brunick2013mimic} with the SDE for $\X$ in \eqref{eq:not:mLocal} and the expression \eqref{eq:not:gamma} for $\gamma$ to conclude that the law $\bar{\mu}_t$ of  $(X_0(t), X_1(t))$ is equal to $\mL(\bar{X}_0(t), \bar{X}_1(t))$, the unique in law solution to the SDE \eqref{eq:wp:marginalSDE} with initial condition $\bar{\lambda}$. This proves the first assertion of the lemma.
    
 To prove the second assertion, first note that  since $\gamma$ satisfies the linear growth condition \eqref{eq:res:gammaLinGrowth}, by Proposition \ref{prop:ap1:superposition}, $\{\bar{\mu}_t\}_{t \in [0, T]}$ solves a linear Fokker-Planck equation with a linear growth drift. Then the claimed inequality of the second assertion follows by Theorem \ref{thm:ap1:wellPosed2}(3).
\end{proof}
\subsection{Proof of well-posedness} 
\label{sec:wp:proof} In this section we prove Theorem \ref{thm:results:bddWp} via a fixed point argument. To this end, we start with the definition of an auxilliary SDE that will be used to define the fixed point problem. To lighten notation, we write $\pathMeasures := \mP(\mC_T^{d(1+\kap})$. Recall that $\Y$ denotes the canonical random variable (see Section \ref{ss:prbMeasure}).
\begin{lemma}[Frozen SDE] 
\label{lem:wp:frozenLemma} Suppose $(U, \intPot)$ satisfy Assumption \ref{as:results:lyapunov} with $\|\nabla \intPot\|_{L^\infty} < \infty$. Let $\hat{\mu} \in \pathMeasures$ be a probability measure that satisfies the following integrability condition:\begin{equation}
    \int_0^T \bigg[\int_{(\R^d)^{1 + \kap}} |\x|\hat{\mu}_t(d\x)\bigg]dt < \infty.
    \label{eq:wp:hatMoment}
\end{equation}
Then, given $b$ as defined in \eqref{eq:results:drift},   
there exists a progressively measurable function $\hat{\gamma}_{\hat{\mu}}:[0, T] \times \R^d \times \R^d \rightarrow \R^d$ such that
\begin{equation}
    \hat{\gamma}_{\hat{\mu}}(t, x, y) = \left\{ \begin{aligned}&\bbE^{\hat{\mu}}\big[ b(\mathbf{Y}(t)) \big| Y_0(t) = x, Y_1(t) = y\big], &\quad &\text{for }\hat{\mu}_t\text{-a.e. }x, y \in \R^d,  \\
    &\nabla U(x) &\quad &\text{otherwise}.
    \end{aligned}\right. \quad t \in [0, T].
    \label{eq:wp:hatGamma}
\end{equation}
Moreover, there exists $C \in (0, \infty)$ such that
\begin{equation}
    \sup_{\hat{\mu} \in \pathMeasures}\bigg\{ \sup_{(t, x, y) \in [0, T] \times (\R^d)^2} \frac{|\hat{\gamma}_{\hat{\mu}}(t,  x, y)|}{1 + |x| + |y|} \bigg\}\leq C,
    \label{eq:wp:gammaLinGrowth}
\end{equation}
and a unique in a law solution to the following SDE:\begin{equation}
    \label{eq:wp:frozenEq}
                \begin{aligned}               d\hat{X}_0(t) &= - \bigg( \nabla U\big(\hat{X}_0(t)\big) + \sum_{v = 1}^\kap \nabla \intPot\big(\hat{X}_0(t) - \hat{X}_v(t)\big)\bigg) dt + \sqrt{2} dB_0(t), \\d\hat{X}_v(t) &= - \hat{\gamma}_{\hat{\mu}}\big(t, \hat{X}_v(t), \hat{X}_0(t)\big) dt + \sqrt{2} dB_v(t), \quad v \in \{1, \ldots, \kap\},
                \end{aligned}
    \end{equation}
    with initial condition $\lambda \in \mPSb$.
\end{lemma}
\begin{proof}  By Assumption \ref{as:results:lyapunov} and  \eqref{eq:wp:hatMoment}, we have
\begin{equation*}
    \int_0^T \bigg[ \int_{(\R^d)^{ 1 + \kap}} \bigg(\nabla U(x_0) + \sum_{v = 1}^\kap \nabla \intPot(x_0 - x_v) \bigg) \hat{\mu}_t(d\x) \bigg] dt < \infty
\end{equation*}
Then by Proposition 5.1 of \cite{brunick2013mimic}, there exists a function $\hat{\gamma}_{\hat{\mu}}$ that satisfies  \eqref{eq:wp:hatGamma}. By \eqref{eq:results:drift} 
and \eqref{eq:wp:hatGamma} it follows that 
\begin{equation*}
    \hat{\gamma}_{\hat{\mu}}(t, x, y) \leq |\nabla U(x)| + \kap \|\nabla \intPot\|_\infty, \quad x, y \in \R^d.
\end{equation*}
Since $\nabla U$ satisfies a linear growth condition by Assumption \ref{as:results:lyapunov} and $\nabla \intPot$ is bounded by assumption, this implies 
 \eqref{eq:wp:gammaLinGrowth}, which shows $\hat{\gamma}_{\hat{\mu}}$ satisfies a linear growth condition. Therefore by Propositions 5.3.6 and 5.3.10 of \cite{karatzas1991stochastic}, there exists a unique in law weak solution to the SDE \eqref{eq:wp:frozenEq}. \end{proof}
\begin{definition}[Frozen \CMVE{}] Suppose $(U, \intPot)$ satisfy Assumption \ref{as:results:lyapunov} with $\|\nabla \intPot\|_{L^\infty} < \infty$. Let $\hat{\mu} \in \pathMeasures$ satisfy \eqref{eq:wp:hatMoment} and let $\hat{\gamma}_{\hat{\mu}}$ be as in Lemma \ref{lem:wp:frozenLemma}. We call \eqref{eq:wp:frozenEq} the \emph{frozen \CMVE{}} on $[0, T]$ with potentials $(U, W)$, frozen measure $\hat{\mu}$ and initial condition $\lambda \in \mPSb$, and use $\Phi: \pathMeasures \rightarrow \pathMeasures$ to denote the map that takes any frozen measure $\hat{\mu}$ to the unique law of the associated frozen SDE  \eqref{eq:wp:frozenEq}.    
\label{def:wp:frozen}
\end{definition}

\begin{proof}[Proof of existence in Theorem \ref{thm:results:bddWp}]
By Definition \ref{def:int:local} of the \CMVE{} and Definition \ref{def:wp:frozen} of $\Phi$, it is clear that $\mu$ defines a solution to the \CMVE{} if and only if it is a fixed point of $\Phi$ in $\mathcal{M}_{\kap, d}^T$. In view of Schauder's fixed point theorem, to establish existence of a fixed point for $\Phi$, it suffices to verify the following claims:
\begin{enumerate}
    \item There is a weakly closed, convex set $\mathfrak{S} \subset \pathMeasures$ such that $\Phi(\mathfrak{S}) \subset \mathfrak{S}$.
    \item $\Phi$ is weakly continuous on $\mathfrak{S}$.
\end{enumerate}

We first introduce some helpful notation that will be used throughout the proof. For $\hat{\mu} \in \pathMeasures$ satisfying \eqref{eq:wp:hatMoment}, let $\hat{\gamma}_{\hat{\mu}}$ be the measurable function defined in Lemma \ref{lem:wp:frozenLemma}, and define the related measurable function $\zeta_{\hat{\mu}}: [0, T] \times \R^d \times \R^d \rightarrow \R^d$ by \begin{equation}
    \zeta_{\hat{\mu}}(t, x, y) = \hat{\gamma}_{\hat{\mu}}(t,  x, y) - \nabla U(x), \quad (t, x, y) \in [0, T] \times \R^d \times \R^d.
    \label{eq:wp:zeta}
\end{equation} By \eqref{eq:results:drift} and \eqref{eq:wp:hatGamma} it follows that 
\begin{equation*}
    \zeta_{\hat{\mu}}(t, x, y) = \bbE^{\hat{\mu}}\bigg[\sum_{v = 1}^\kap \nabla \intPot(Y_0(t) - Y_v(t)) \bigg| Y_0(t) = x, Y_v(t) = y\bigg], \quad \hat{\mu}_t\text{-a.e. }x, y \in \R^d,\quad t \in [0, T],
    \label{eq:wp:zeta2}
\end{equation*}
and $\zeta_{\hat{\mu}}(t, x, y) = 0$ for all $x,y \in \R^d$, on the remaining $\hat{\mu}_t$-null set. This yields the following uniform bound on $\zeta_{\hat{\mu}}$:
\begin{equation}
\|\zeta_{\hat{\mu}}\|_{L^\infty(\hat{\mu})} \leq \|\nabla W\|_{L^\infty(\R^d)}, \quad \hat{\mu} \in \mP\big(\mC_T^{d(1 + \kap)}\big).
\label{eq:wp:bddZeta}
\end{equation}
The proof proceeds via the following three steps. 

{\em Step 1: Identify a candidate subset.} To define the subset $\mathfrak{S}  \subset\pathMeasures$ to which we can apply Schauder's fixed point theorem, we introduce an intermediary set $\mathfrak{D} \subset \mP_{\kap, d, T}$. Since $\nabla U$ satisfies Assumption \ref{as:results:lyapunov}, Propositions 5.3.6 and 5.3.10 of \cite{karatzas1991stochastic} guarantee that there is a unique in law weak solution to the following SDE: 
    \begin{equation*}
        dZ_v(t) = - \nabla U\big(Z_v(t)\big) dt + \sqrt{2} dB_v(t), \quad v \in \{0, \ldots, \kap\}, \quad \mathbf{Z}(0) \sim \lambda.
    \end{equation*}
Let $\rho := \mL(\mathbf{Z}) \in \pathMeasures$ denote the law of this unique weak solution. Define $\mathfrak{D} \subset \pathMeasures$ to be the set of measures $\hat{\mu}$ that satisfy the following three properties: 
\begin{itemize}
    \item $\hat{\mu}_0 = \lambda$ and $\hat{\mu}$ satisfies \eqref{eq:wp:hatMoment}.
    \item The time-marginals $\hat{\mu}_t$ satisfy the symmetry property \eqref{eq:not:symmetry1_} for all $t \in [0, T]$.
    \item We have 
    \begin{equation}
            \bbE^{\rho}\bigg [ \bigg|\frac{d\hat{\mu}}{d\rho} \bigg|^2\bigg] \leq \sup_{\hat{\nu} \in \mathcal{P}(( \mC_T^d)^{1 + \kap})} \bbE^{\rho}\bigg[ \bigg|\frac{d\Phi(\hat{\nu})}{d\rho}\bigg|^2\bigg] =:C_\Phi.     
            \label{eq:wp:RNBound}
    \end{equation} 
\end{itemize}   
Note that by Girsanov's theorem, Novikov's condition (see e.g., Corollary 3.5.16 of \cite{karatzas1991stochastic}), and \eqref{eq:wp:bddZeta}, the boundedness of $\nabla \intPot$ implies that the constant $C_\Phi$ of \eqref{eq:wp:RNBound} satisfies \begin{equation*}
C_\Phi \leq e^{T\|\nabla \intPot\|_{L^\infty(\R^d)}} < \infty.
  \end{equation*}
It is clear from Remark \ref{rk:not:exchangeability} and Fatou's lemma  that $\mathfrak{D}$ is a weakly closed subset of 
$\pathMeasures$, and further it is easy to see that it is also convex. Moreover, the first item in the definition of $\mathfrak{D}$ and Lemma \ref{lem:wp:frozenLemma} imply that $\Phi$ is well-defined on $\mathfrak{D}$.
  
We now show $\Phi(\mathfrak{D}) \subset \mathfrak{D}$. It is clear from the above that \eqref{eq:wp:RNBound} is preserved by $\Phi$. Moreover, Lemma \ref{lem:wp:frozenLemma} and Theorem \ref{thm:ap1:wellPosed2}(1), guarantee that $\Phi(\hat{\mu})_0 = \lambda$ and that \eqref{eq:wp:hatMoment} is preserved by $\Phi$. To show that the symmetry \eqref{eq:not:symmetry1_} of time-marginals is preserved by $\Phi$, we argue that symmetry properties of the law of the trajectories of the frozen SDE \eqref{eq:wp:frozenEq} follow from symmetry properties of the time-marginals of the frozen measure $\hat{\mu}_t$. For $\hat{\mu} \in \mathfrak{D}$, let $\hat{\X}$ represent the random variable associated with $\Phi(\hat{\mu})$, that is, the solution to \eqref{eq:wp:frozenEq} with frozen measure $\hat{\mu}$.  Since the time marginals $\hat{\mu}_t$ satisfy \eqref{eq:not:symmetry1_},  the drift of \eqref{eq:wp:frozenEq} as a function on $(\R^d)^{1 + \kap}$ is invariant under any permutation $\tau$ of $\{1, \ldots, \kap\}$. Thus for any such $\tau$, the measure $\Phi^\tau(\hat{\mu})$ given by \begin{align*}
        \Phi^\tau(\hat{\mu}) := \mathcal{L}(\hat{X}_0, 
        \hat{X}_{\tau(1)}, \ldots, \hat{X}_{\tau(\kap)}),
    \end{align*}
    is also a solution to \eqref{eq:wp:frozenEq}. By uniqueness of solutions to the frozen SDE \eqref{eq:wp:frozenEq}, we have $\Phi^\tau(\hat{\mu}) = \Phi(\hat{\mu})$ and the time-marginals of $\Phi(\hat{\mu})$ also satisfy \eqref{eq:not:symmetry1_}. 
    
    We define our candidate set $\mathfrak{S} := \overline{\text{conv}(\Phi(\mathfrak{D}))}$ to be the closed convex hull of $\Phi(\mathfrak{D})$ in $\pathMeasures$. Since $\mathfrak{D}$ is weakly closed and convex, it follows that $\Phi(\mathfrak{S}) \subset \Phi(\mathfrak{D}) \subset \mathfrak{S} \subset \mathfrak{D}$.
     
     {\em Step 2: Show that $\Phi$ is weakly continuous on $\mathfrak{S}$ and apply Schauder's fixed point theorem.} We first demonstrate that given any weakly convergent sequence $\hat{\mu}^l \rightarrow \hat{\mu}$ in $\mathfrak{S}$, the time marginals converge in total variation. In other words, letting $d_{\TV}$ denote  total variation distance, we have  
     \begin{equation}
         \lim_{l \rightarrow \infty} d_{\TV}(\hat{\mu}^l_t, \hat{\mu}_t) = 0, \quad \text{a.e. }t\in (0, T).
         \label{eq:wp:TVConvergence}
     \end{equation}

     By Lemma \ref{lem:wp:frozenLemma} and Definition \ref{def:wp:frozen}, each $\hat{\mu} \in \Phi(\mathfrak{D})$ is the solution to an SDE with a linear growth drift. By Theorem \ref{thm:ap1:wellPosed2}(2), for all $t \in [0, T]$, $\hat{\mu}_t$ has a bounded positive density with a H\"older coefficient and exponent that depend only on  $(U, \intPot, T, \kap, \lambda)$. Moreover, this uniform H\"older property is preserved by convex combinations,  that is, for all $t \in (0, T]$ and compact $B \subset (\R^d)^{1 + \kap}$ there exists $\alpha = \alpha(t, B, \lambda, U, \intPot) \in (0, 1]$ and $C = C(t, B, \lambda, U, \intPot) < \infty$ such that for all $\hat{\mu} \in \text{conv}(\Phi(\mathfrak{D}))$, we have  \begin{equation}
     \label{eq:wp:uniformHolder}
         \sup_{x, y \in B} \frac{|\hat{\mu}_t(x) - \hat{\mu}_t(y)|}{|x - y|^{\alpha}} \leq C.
     \end{equation}

    Now let $\{\hat{\mu}^l\}_{l \geq 1}$ be a sequence in the convex hull of $\Phi(\mathfrak{D})$ and suppose $\hat{\mu}^l$ converges weakly to some $\hat{\mu} \in \pathMeasures$. By the uniform H\"older property \eqref{eq:wp:uniformHolder}, for each $t \in (0, T]$ and compact set $B$ the sequence $\{\hat{\mu}^l_t\}_{l \geq 1}$ is equicontinuous on $B$. Therefore, by the Arzela-Ascoli theorem, every subsequence of $\{\hat{\mu}^l_t\}_{l \geq 1}$ has a further subsequence that converges uniformly on compacts. Since we have $\hat{\mu}^l_t \rightarrow \hat{\mu}_t$ weakly, we therefore must have  $\hat{\mu}^l_t \rightarrow \hat{\mu}_t$ uniformly on compact sets. The total variation convergence \eqref{eq:wp:TVConvergence} then follows from Scheffe's lemma. 
     
  Next, note that Assumption \ref{as:results:lyapunov}, the uniform bound \eqref{eq:wp:bddZeta} on $\hat{\zeta}$,  Lemma \ref{lem:wp:entropyEstimates}, and \eqref{eq:not:symmetry1_}, together imply \begin{equation}
    \begin{aligned}
        \mH(\Phi(\hat{\mu}) \,|\, \Phi(\hat{\mu}^l)) &= \frac{1}{4} \int_0^T \bbE^{\Phi(\hat{\mu})}\bigg[ \sum_{v = 1}^\kap\big|\zeta_{\hat{\mu}^l}(t, Y_v(t), Y_0(t)) - \zeta_{\hat{\mu}}(t, Y_v(t), Y_0(t))|^2\bigg] dt, \\   &= \frac{\kap}{4} \int_0^T \bbE^{\Phi(\hat{\mu})}\Big[ \big|\zeta_{\hat{\mu}^l}(t, Y_1(t), Y_0(t)) - \zeta_{\hat{\mu}}(t, Y_1(t), Y_0(t))|^2\Big] dt,
    \end{aligned}
    \label{eq:wp:existenceEntropy}
    \end{equation}
    By the convergence in  \eqref{eq:wp:TVConvergence}, Theorem 3.1 of \cite{crimaldi2005two} (see also Proposition 3.4 of \cite{conforti2023projected}), and \eqref{eq:wp:zeta2}, we have
    \begin{equation*}
            \lim_{l \rightarrow \infty}\mathbb{E}^{\Phi(\hat{\mu})}\Big[\big|\zeta_{\hat{\mu}^l}(t, Y_1(t), Y_0(t)) - \zeta_{\hat{\mu}}(t, Y_1(t), Y_0(t))|^2\Big] = 0, \quad t \in (0, T].
    \end{equation*}
    Since $\zeta_{\hat{\mu}}$ is bounded, the bounded convergence theorem, \eqref{eq:wp:existenceEntropy}, and the above display imply that
    \begin{equation*}
        \lim_{l \rightarrow \infty} \mathcal{H}(\Phi(\hat{\mu})|\Phi(\hat{\mu}^l)) = 0.
    \end{equation*}
    Therefore by Pinsker's inequality, $\hat{\mu}^l \rightarrow \hat{\mu}$ weakly in $\mathfrak{S}$ implies $d_{TV}(\Phi(\hat{\mu}), \Phi(\hat{\mu}^l)) \rightarrow 0$. Therefore, $\Phi$ is weakly continuous on $\mathfrak{S}$ and this concludes the second claim. 
    
Thus, by Schauder's fixed point theorem, there exists a fixed point $\mu \in \mathfrak{S}$ of $\Phi$. By Proposition \ref{thm:ap1:wellPosed}(2), $\mu$ is a continuous positive function. Hence, we can choose $\gamma$ to be the associated conditional expectation of $\mu$. That is, with $b$ as in \eqref{eq:results:drift}, we set
    \begin{equation}
        \gamma(s, x, y):= \frac{1}{\bar{\mu}_s(x, y)}\int_{(\R^d)^{\kap-1}} b(\x) \mu_s(x, y, x_2, \ldots, x_\kap) \prod_{v = 2}^\kap dx_v, \quad (s, x, y) \in [0, T] \times \R^d \times \R^d.
        \label{eq:wp:gamma}
    \end{equation}
    
    {\em Step 3: Conclude that the fixed point is a solution to the \CMVE{}.} Let $\mu \in \mathfrak{S}$ be the fixed  of $\Phi$ described above, let $\gamma$ be as defined in \eqref{eq:wp:gamma}, and let $\X$ be a stochastic process with law $\mu$. By the definition of $\mathfrak{S}$, we have by the first item in the definition of $\mathfrak{D}$ that $\mu_0 = \lambda$. Moreover, as a fixed point of $\Phi$, $\mu$ is a solution to \eqref{eq:not:mLocal} and $\gamma$ as defined in \eqref{eq:wp:gamma} satisfies \eqref{eq:not:gamma}. Therefore to show that $(\mu, \gamma, \X)$ is a solution to the \CMVE{}, it remains to verify that $\X$ satisfies \eqref{eq:cmve:integrability}, $\mu$ is an element of $\mathcal{M}_{\kap, d}^T$, and that $\gamma$ satisfies the linear growth condition \eqref{eq:res:gammaLinGrowth}. By \eqref{eq:results:potential}, \eqref{eq:wp:gamma},    
    Assumption \ref{as:results:lyapunov}, and \eqref{eq:wp:bddZeta}, there exists $C \in (0, \infty)$ such that
    \begin{equation*}
        \sup_{s \in \R_+} \gamma(\mu_s, x, y) \leq  C(1 + |x|) + \|\nabla \intPot\|_{L^\infty(\R^d)},
    \end{equation*}
    Since $\mu$ is the solution to a SDE with a linear growth drift, we see that by Theorem \ref{thm:ap1:wellPosed2}(1), Assumption \ref{as:results:lyapunov}, and the above display, we have 
    \begin{equation}
    \label{eq:wp:finiteMoments}
        \int_0^T \mathbb{E} \Big[ |b(\X_t)|^2 + |\gamma(\mu_t, X_1(t), X_0(t)|^2 + |\gamma(\mu_t, X_0(t), X_1(t))|^2\Big] dt < \infty.
    \end{equation}
    Thus, $\X$ satisfies \eqref{eq:cmve:integrability}. Finally, we show that $\mu \in \mathcal{M}_{\kap, d}^T$. By the definition of $\mathfrak{S}$, $\mu_t$ satisfies \eqref{eq:not:symmetry1_} for all $t \in [0, T]$ and therefore it suffices to show the second symmetry property \eqref{eq:not:symmetry2_} for the time marginals $\{\mu_t\}_{t \in [0, T]}$.    
    Let $(X_0, X_1)$ be distributed according to the $(0, 1)$-marginal of $\mu$. Lemma \ref{lem:wp:marginalSDE} implies that  $({X}_1(t), {X}_0(t))$ is exchangeable for all $t \in [0, T]$. Thus  $\mu_t$ satisfies \eqref{eq:not:symmetry2_}. 
    \end{proof}
    \begin{proof}[Proof of uniqueness in Theorem \ref{thm:results:bddWp}]
    Suppose there are two linear growth solutions $(\mu^1, \gamma^1)$ and $(\mu^2, \gamma^2)$ to the \CMVE{}. Applying Lemma \ref{lem:wp:entropyEstimates} to the \CMVE{} \eqref{eq:not:mLocal}, for all $t \in [0, T]$ we have
    \begin{equation}
        \mathcal{H}(\mu^1[t]|\mu^2[t]) = \frac{1}{4}\int_0^t \sum_{v = 1}^\kap\bbE^{\mu^1}\Big[ |\gamma^1(s,Y_v(s), Y_0(s)) - \gamma^2(s, Y_v(s), Y_0(s))|^2\Big] ds.
        \label{eq:wp:bddUniq1}
    \end{equation}
    For $s \in [0, T]$, let $\mu^i_s( \cdot | x, y)$ denote the conditional law of $\mu^i_s$ given $\{X_0 = x, X_1 = y\}$ for $i = 1, 2$.
    Recall the definition of $\gamma$ in \eqref{eq:not:gamma} and $b$ in \eqref{eq:results:drift}. Since $\mu^1, \mu^2 \in \mPS$ are continuous and positive by Proposition \ref{thm:ap1:wellPosed}(2), we can apply \eqref{eq:not:symmetry1_} and Lemma \ref{lem:wp:pinsker} and observe that for all $s \in [0, t]$ and $x, y \in \R^d$, we have
    \begin{eqnarray*}
    && |\gamma(\mu^1_s, x, y) - \gamma(\mu^2_s, x, y)|^2\\ 
     && \qquad= \big| \bbE^{\mu^1}\big[ b(\Y(s))|Y_0(s) = x, Y_1(s) = y\big] - \bbE^{\mu^2}\big[ b(\Y(s))|Y_0(s) = x, Y_1(s) = y\big] \big|^2 \\
        &&\qquad =  \kap^2 \big| \bbE^{\mu^1}\big[\nabla \intPot(x - Y_2(s)) |Y_0(s) = x, Y_1(s) = y\big] - \bbE^{\mu^2}\big[ \nabla \intPot(x - Y_2(s))|Y_0(s) = x, Y_1(s) = y\big] \big|^2 \\
        && \qquad \leq 2(1 + \kap\|\nabla \intPot\|_{L^\infty(\R^d)}) \mathcal{H}\big(\mu^1_s(\cdot|x, y)\,|\, \mu^2_s(\cdot|x, y)\big).       
    \end{eqnarray*}
    For $v \in \{1, \ldots, \kap\}$, the above display together with the chain rule for relative entropy imply that 
    \begin{equation}
        \bbE^{\mu^1}[ |\gamma^1(s, Y_v(s), Y_0(s)) - \gamma^2(s, Y_v(s), Y_0(s))|^2] \leq 2( 1+ \|\nabla \intPot\|_{L^\infty(\R^d)}) \mathcal{H}(\mu^1_s | \mu^2_s), \quad s \in [0, t],
        \label{eq:wp:bddUniq2}
    \end{equation}
    By the data processing inequality for relative entropy, we have $\mathcal{H}(\mu^1_s|\mu^2_s) \leq \mathcal{H}(\mu^1[s]|\mu^2[s])$. Combining this with \eqref{eq:wp:bddUniq1} and \eqref{eq:wp:bddUniq2} yields
    \begin{equation*}
        \mathcal{H}(\mu^1[t]|\mu^2[t]) \leq \frac{\kap \|\nabla \intPot\|_{L^\infty(\R^d)}}{4}\int_0^t \mathcal{H}(\mu^1[s]|\mu^2[s]) ds, \quad t \in [0, T].
    \end{equation*}
    By Gronwall's inequality, we conclude that $\mathcal{H}(\mu^1[t]|\mu^2[t]) = 0$ for all $t \in [0, T]$ and therefore linear growth solutions to \CMVE{} are unique. 
     
    Finally, we show global existence and uniqueness by a standard iteration argument. Fix $t_0 > 0$. We have shown that there exists a unique solution $\mu$ to the \CMVE{} on $[0, t_0]$ with potentials $(U, \intPot)$ and initial condition $\lambda$. Moreover by Theorem \ref{thm:ap1:wellPosed}, $\mu_{t_0} \in \mPSb$, and there exists a unique solution to the \CMVE{} on $[t_0,2 t_0]$ with potentials $(U, \intPot)$ and initial condition $\mu_{t_0}$. Thus by the Markov property, $\mu$ extends to a unique solution of the \CMVE{} on $[0, 2t_0]$ with potentials $(U, \intPot)$ and initial condition $\lambda$. Repeating this argument shows that $\mu$ extends to a unique solution to the \CMVE{} on $[0, \infty)$.
\end{proof}

\section{Proof of the H-Theorem}
\label{sec:lyap:Htheorems}
\subsection{Preliminaries} We provide two preliminary lemmas that will be useful in the proof of Theorem \ref{thm:results:lyap}. The first is used in Section \ref{sec:hTheoremProof} to compute $\energy_\kap(\mu_t)$. 
\begin{lemma}
    \label{prop:lyap:pdeProp} Let $(U, \intPot)$ satisfy Assumption \ref{as:results:lyapunov}. Let $(\mu, \gamma)$ be a linear growth solution to the \CMVE{} on $[0, T]$ with potentials $(U, \intPot)$ and initial condition $\lambda \in \mPSb$. Recall the definitions of $b$, $g$, and $\eta$ from \eqref{eq:results:drift}, \eqref{eq:results:potential}, and \eqref{eq:not:eta} respectively, and define
    \begin{equation}
    \bar{\eta}(t, x, y) := \big( \gamma(t, x, y), \gamma(t, y, x) \big), \quad (t, x, y) \in [0, T] \times (\R^d)^2.
    \label{eq:lyap:barEta}
\end{equation}
Then for almost every $0 \leq r < t < T$, we have
    \begin{equation}
    \begin{split}
        \int_{(\R^d)^{1 + \kap}}  \Big( \log \mu_t(\x)& + g(\x) \Big) \mu_t(\x)d\x - \int_{(\R^d)^{1 + \kap}}  \Big( \log \mu_r(\x) + g(\x) \Big) \mu_r(\x)d\x \\
        =& - \int_r^t \int_{(\R^d)^{1 + \kap}} \Big(\nabla_{\x} \mu_s(\x) + \eta(s, \x) \mu_s(\x)\Big) \cdot \bigg(\nabla_{\x}\log  \mu_s(\x)  + \nabla_{\x} g(\x) \bigg) d\x ds,
    \end{split}
    \label{eq:lyap:PDEProp1}
\end{equation}
and
\begin{equation}
    \begin{split}
        &\int_{\R^d\times \R^d}  \bar{\mu}_t(x_0, x_1)\log {\bar{\mu}}_t(x_0, x_1) d\x - \int_{\R^d\times \R^d} \bar{\mu}_r(x_0, x_1) \log {\bar{\mu}}_r(x_0, x_1)  d\x \\
        =& - \int_r^t \int_{\R^d\times \R^d} \Big(\nabla_{(x_0, x_1)} \bar{\mu}_s(x_0, x_1) + \bar{\eta}(s, x_0, x_1))\bar{\mu}_s(x_0, x_1) \Big) \cdot \nabla_{(x_0, x_1)} \log {\bar{\mu}}_s(x_0, x_1)   dx_0 dx_1 ds.
    \end{split}
        \label{eq:lyap:PDEProp2}
\end{equation}
\end{lemma}
\begin{proof}
    Since $(U, \intPot)$ satisfy Assumption \ref{as:results:lyapunov},  by \eqref{eq:results:drift} and \eqref{eq:results:potential} it follows that $b$ satisfies a linear growth condition and there exists $C \in (0, \infty)$ such that $g(x) \leq C(1 + |x|^2)$ for all $x \in \R^d$. Since $\mu$ is a linear growth solution to the \CMVE{}, by Proposition \ref{prop:ap1:superposition}, $\{\mu_t\}_{t \in [0, T]}$ solves
\begin{equation*}
    \partial_t \mu_t(\x) = \Delta_{\x} \mu_t(\x) + \nabla_{\x} \cdot \big( \eta(t, \x) \mu_t(\x)\big), \quad (t, \x) \in \R_+ \times (\R^d)^{1 + \kap}.
\end{equation*}
In particular, $\mu_t$ solves a linear Fokker-Planck equation. Hence, we can apply Lemma \ref{lem:ap1:lyapunov1} to obtain \eqref{eq:lyap:PDEProp1}. Since $\mu$ is a linear growth solution to the \CMVE{}, by Lemma \ref{lem:wp:marginalSDE} $\{\bar{\mu}_t\}_{t \geq 0}$ are the time marginals to the solution of the linear SDE \eqref{eq:wp:marginalSDE}. Applying Proposition \ref{prop:ap1:superposition} and then Lemma \ref{lem:ap1:lyapunov1} yields \eqref{eq:lyap:PDEProp2}.
\end{proof}
We introduce some special notation for gradients in this section. For $f:(\R^d)^{1 + \kap} \rightarrow \R$, we let $\nabla_v f(\x)$ denote the gradient with respect to the $v$-th coordinate of $f$. That is, \[\nabla_{\x} f(\x) = (\nabla_0 f(\x), \nabla_1 f(\x), \ldots, \nabla_\kap f(\x)).\] 
This will be especially useful when performing calculations inside expectations.

The next lemma establishes a useful symmetry property of the function $\nabla_\x \log \mu_s(\x)$. \begin{lemma}
 \label{lem:lyap:symLemma}
     Let $(\mu,\gamma, \X)$ be a solution to the \CMVE{} on $[0, T]$ with potentials $(U, \intPot)$ and initial condition $\lambda \in \mPSb$. Let $f : \R^d \times \R^d \rightarrow \R^d$ be a measurable function such that for almost every $s \in \R_+$ we have
\begin{equation}
    \bbE\Big[\big|f\big(X_0(s), X_v(s)\big)\big|^2 + \big|f\big(X_v(s), X_0(s)\big)\big|^2 \Big] < \infty.
    \label{eq:lyap:symLemma1}
\end{equation}
Then the following properties hold:
\begin{enumerate}
    \item For $u \in \{0, \ldots, \kap\}$, let $\bar{u} = 0$ when $u = 0$ and $\bar{u} = 1$ otherwise. Given any $v \in \{1, \ldots, \kap\}$ and $u \in \{0, v\}$, for almost every $s \in \R_+$ we have 
    \begin{equation}
        \bbE\Big[f\big(X_0(s), X_v(s)\big) \nabla_{u} \log \mu_s\big(\X(s)\big)\Big] = \bbE\Big[f\big(X_0(s), X_v(s)\big) \nabla_{\bar{u}} \log \bar{\mu}_s\big(X_0(s), X_v(s)\big)\Big].
        \label{eq:lyap:symLemma2}
    \end{equation}
    \item For all $v \in \{1, \ldots, \kap\}$ and $s \in \R_+$ we have
    \begin{equation}
        \bbE\Big[f\big(X_0(s), X_v(s)\big) \nabla_{0} \log \bar{\mu}_s\big(X_0(s), X_v(s)\big)\Big] = \bbE\Big[f\big(X_v(s), X_0(s)\big) \nabla_{1} \log \bar{\mu}_s\big(X_0(s), X_v(s)\big)\Big].
        \label{eq:lyap:symLemma3}
    \end{equation}
\end{enumerate}
 \end{lemma}
\begin{proof}
Fix $v \in \{1, \ldots, \kap\}$ and $u \in \{0, v\}$. For almost every $s \in \R_+$, by Proposition \ref{thm:ap1:wellPosed}(3) we have $\mu_s \in W^{1, 1}((\R^d)^{1 + \kap})$ and by Definition \ref{def:int:local}(4), $\mu_s \in \mPS$. Recall the definition of $\bar{u}$ from Lemma 6.2.(2). Then we have 
\begin{equation*}
    \int_{(\R^d)^{\kap - 1}} \nabla_{u} \mu_s(\x) \prod_{v' = 1;\, v' \neq v}^\kap dx_{v'} = \nabla_{\bar{u}} \bar{\mu}_s(x_0, x_v).
\end{equation*}
By Proposition \ref{thm:ap1:wellPosed}(3), we have that $\nabla_\x \log \mu_s(\x) \in L^2(\mu_s)$ for almost every $s \in \R_+$. Moreover, since $\mu$ is a linear growth solution to the \CMVE{}, Lemma \ref{lem:wp:marginalSDE} implies that $\nabla_{(0, 1)} \log \bar{\mu}_s(x_0, x_1) \in L^2(\mu_s)$ for almost every $s \in \R_+$. Therefore combining \eqref{eq:lyap:symLemma1}, the last display, and Fubini's theorem, yields
    \begin{equation*}
        \begin{aligned}
            \bbE\Big[f\big(X_0(s), X_v(s)\big) \nabla_{u} \log \mu_s\big(\X(s)\big)\Big] &= \int_{(\R^d)^{1 + \kap}} f(x_0, x_v) \nabla_{u} \mu_s(\x) d\x \\ &= \int_{\R^d \times \R^d} f(x_0, x_v) \nabla_{\bar{u}} \bar{\mu}_s(x_0, x_v) d\x \\
            &= \bbE\Big[f\big(X_0(s), X_v(s)\big) \nabla_{\bar{u}} \log \bar{\mu}_s\big(X_0(s), X_v(s)\big)\Big],
        \end{aligned}
    \end{equation*}
    which proves \eqref{eq:lyap:symLemma2}.
    
Next we turn to the proof of \eqref{eq:lyap:symLemma3}. Recall that $\mu_s \in \mPS$ by item (4) in Definition \ref{def:int:local}. By the symmetry property \eqref{eq:not:symmetry2_} of measures in $\mPS$, we have
    \begin{equation*}
        \bbE\Big[f\big(X_0(s), X_v(s)\big) \nabla_{0} \log \bar{\mu}_s\big(X_0(s), X_v(s)\big)\Big] = \bbE\Big[f\big(X_0(s), X_v(s)\big) \nabla_{1} \log \bar{\mu}_s\big(X_v(s), X_0(s)\big)\Big].
    \end{equation*}
    Moreover $\mu_s \in \mPS$ also implies (by Remark \ref{rk:not:exchangeability}) that the pair $(X_0(s),X_v(s))$ is exchangeable and hence,
    \begin{equation*}
        \bbE\Big[f\big(X_0(s), X_v(s)\big) \nabla_{0} \log \bar{\mu}_s\big(X_v(s), X_0(s)\big)\Big] = \bbE\Big[f\big(X_v(s), X_0(s)\big) \nabla_{1} \log \bar{\mu}_s\big(X_0(s), X_v(s)\big)\Big].
    \end{equation*}
    The previous two displays, when combined, yield \eqref{eq:lyap:symLemma3}.
\end{proof}
\subsection{Proof of the H-theorem}\label{sec:hTheoremProof}
In this section, we prove Theorem \ref{thm:results:lyap}.  

\begin{proof}[Proof of Theorem \ref{thm:results:lyap}] Fix a linear growth solution $(\mu, \gamma, \X)$ to the \CMVE{} on $[0, \infty)$ with potentials $(U, \intPot)$ satisfying Assumption \ref{as:results:lyapunov} and initial law $\lambda \in \mPSb$. Recall the definition of $b$ in \eqref{eq:results:drift}. 

Let $\mathcal{D}_T \subset \{(r, t): 0 \leq r < t \leq T\}$ be the set on which both \eqref{eq:lyap:PDEProp1} and \eqref{eq:lyap:PDEProp2} hold. By Proposition \ref{prop:lyap:pdeProp}, $\mathcal{D}_T$ is a set of full measure. Fix $(r, t) \in \mathcal{D}_T$ such that $r < t$. We split the left hand side of \eqref{eq:results:energyIdentity} into two parts:
\begin{equation}
    \label{eq:pf:LyapunovSplit}
    \energy_\kap(\mu_t) - \energy_\kap(\mu_r) = \Upsilon - \frac{\kap}{2}\bar{\Upsilon},
\end{equation}
where
\begin{align}
    \Upsilon &:=         \bbE\big[\log \mu_t\big( \X(t) \big) + g\big( \X(t) \big)\big] - \bbE\big[\log \mu_r(\X(r)) 
    + g(\X(r))\big],
    \label{eq:lyap:Energy0} \\
    \bar{\Upsilon} &:=         \bbE\big[\log {\bar{\mu}}_t( \X(t))\big] - \bbE\big[\log {\bar{\mu}}_r(\X(r))\big]. \label{eq:lyap:EnergyC}
\end{align}

We begin by computing $\Upsilon$.  By \eqref{eq:lyap:Energy0} and Proposition \ref{prop:lyap:pdeProp}, we have:
    \begin{equation}
        \label{eq:pf:uncondLyapu}
    \begin{aligned}
        \Upsilon 
&= - \int_r^t \Big(\Upsilon^1(s) + \Upsilon^1(s) + \Upsilon^3(s)\Big) ds,
    \end{aligned}
    \end{equation}
    where for almost every $s \in [r, t]$, we define
    \begin{equation}
        \begin{aligned}
    \label{eq:lyap:U0_children}
            \Upsilon^1(s) &:= \int_{(\R^d)^{1 + \kap}}  |\nabla_\x \log \mu_s(\x)|^2 \mu_s(\x)\, d\x = \bbE\Big[ \big|\nabla_\x \log \mu_s\big(\X(s)\big)\big|^2\Big] , \\
            \Upsilon^2(s) &:= \int_{(\R^d)^{1 + \kap}}  \eta(s,
\x)\cdot \nabla_{\x} g(\x) \mu_s(\x)  \,d\x = \bbE\Big[ \eta\big(s, \X(s)\big) \cdot \nabla_\x  g\big(\X(s)\big)\Big] , \\
            \Upsilon^3(s) &:= \int_{(\R^d)^{1 + \kap}}  (\eta(s, \x) + \nabla_{\x} g(\x)) \cdot \nabla_{\x} \mu_s(\x) \, d\x  \\ &= \bbE\Big[ \big( \eta\big(s, \X(s)\big) + \nabla_\x g\big(\X(s)\big) \big) \cdot \nabla_{\x} \log \mu_s\big(\X(s)\big)\Big].
        \end{aligned}
    \end{equation}
We start by expanding $\Upsilon^1(s)$. Since $\mu_s \in \mPS$, by \eqref{eq:not:symmetry1_} we have
\begin{equation}
\begin{aligned}
     \Upsilon^1(s) = \bbE\bigg[\big|\nabla_{0} \log \mu_s\big(\X (s)\big) \big|^2\bigg] + \kappa \bbE\bigg[\big|\nabla_{1} \log \mu_s\big(\X (s)\big) \big|^2\bigg].
\end{aligned}
\label{eq:pf:term3}
\end{equation}
Next, we compute $\Upsilon^2(s)$. By \eqref{eq:results:potential} it follows that
    \begin{equation}
        \nabla_{v} g(\x) = \left\{
        \begin{aligned}        
        &\nabla U(x_0) + \frac{1}{2} \sum_{v = 1}^{\kap} \nabla \intPot(x_0 - x_v), &\quad &\text{if }v = 0, \\
        &- \frac{1}{2}\nabla \intPot(x_0 - x_v),  &\quad &\text{if }v \in\{1, \ldots, \kap\}.
              \end{aligned}
        \right.
        \label{eq:lyap:gGradient}
    \end{equation}
Substituting this into \eqref{eq:lyap:U0_children} and using \eqref{eq:not:eta} and \eqref{eq:results:drift}, we can write
\begin{equation}
        \label{eq:lyap:Kcalculation0}
\begin{aligned}
            \Upsilon^2(s) = \bbE\bigg[b\big(\X(s)\big)\bigg( \nabla U(X_0(s)) &+ \frac{1}{2}\sum_{ v = 1}^\kap \nabla \intPot(X_0(s) - X_v(s))\bigg)\bigg] \\ &- \frac{1}{2}\sum_{v = 1}^\kap\bbE\Big[ \gamma\big(s, X_v(s), X_0(s)\big) \nabla \intPot(X_0(s) - X_v(s)) \Big].
\end{aligned}
\end{equation}
Since $\mu_s \in \mPS$, for each $v \in \{1, \ldots, \kap\}$, \eqref{eq:not:symmetry2_} and the fact that $ \nabla \intPot$ is odd imply that
\begin{equation}
    \begin{aligned}
    \bbE\Big[ \gamma\big(s, X_v(s), X_0(s)\big) \nabla \intPot(X_0(s) - X_v(s)) \Big] &= \bbE\Big[ \gamma\big(s, X_1(s), X_0(s)\big) \nabla \intPot(X_0(s) - X_1(s)) \Big] \\
    &= \bbE\Big[ \gamma\big(s, X_0(s), X_1(s)\big) \nabla \intPot(X_1(s) - X_0(s)) \Big] \\
    &= - \bbE\Big[ \gamma\big(s,X_0(s), X_1(s)\big) \nabla \intPot\big( X_0(s) - X_1(s)\big)\Big]
        \end{aligned}
        \label{eq:lyap:Kcalculuation1}
\end{equation}
Due to the definition of $\gamma$ in \eqref{eq:not:gamma}, the tower property yields
\begin{equation*}
    \begin{aligned}
       \bbE\Big[ \gamma(s, X_0(s), X_1(s)) \nabla \intPot\big( X_0(s) - X_1(s)\big)\Big] &= \bbE\Big[ \bbE \big[ b\big(\X(s)\big)\big|X_0(s), X_1(s)\big] \nabla \intPot\big( X_0(s) - X_1(s)\big)\Big] \\
        &=\bbE\big[ b\big(\X(s)\big) \nabla \intPot\big(X_0(s) - X_1(s)\big)\big],
    \end{aligned}
\end{equation*}
Together with \eqref{eq:lyap:Kcalculation0} and \eqref{eq:results:drift},  this implies that
\begin{equation}
    \label{eq:pf:term1}
    \Upsilon^2(s) = \bbE\big[b\big( \X(s) \big)^2\big].
\end{equation}

Next, we simplify the expression for $\Upsilon^3(s)$ in \eqref{eq:lyap:U0_children}. By \eqref{eq:results:drift}, \eqref{eq:not:eta}, and \eqref{eq:lyap:gGradient} it follows that
\begin{equation}
    \begin{aligned}
    \Upsilon^3(s) = \Upsilon^{3, 1}(s) + \Upsilon^{3, 2}(s),
\end{aligned}
\label{eq:lyap:U3}
\end{equation}
where we define
\begin{equation}
    \begin{aligned}
    \Upsilon^{3, 1}(s) &:=\bbE \bigg[\Big(b\big(\X(s)\big) + \nabla_{0}g\big(\X(s)\big) \Big) \cdot \nabla_{0}\log \mu_s\big(\X(s)\big)- \sum_{v = 1}^\kappa\frac{1}{2}\nabla \intPot\big(X_0(s) - X_v(s)\big) \cdot \nabla_{v} \log \mu_s\big(\X(s)\big)\bigg] \\
    \Upsilon^{3, 2}(s) &:= \sum_{v = 1}^\kappa \bbE\Big[\gamma\big(s, X_v(s), X_0(s)\big) \cdot \nabla_{v} \log \mu_s\big(\X(s)\big)\Big].
    \end{aligned}
    \label{eq:lyap:newDeltas}
\end{equation}
By Assumption \ref{as:results:lyapunov}, $\nabla \intPot$ satisfies a linear growth condition. Therefore by Proposition \ref{thm:ap1:wellPosed}(1), the condition \eqref{eq:lyap:symLemma1} holds with $f(x, y) = \nabla \intPot(x - y)$. Then through two applications of \eqref{eq:lyap:symLemma2} and \eqref{eq:lyap:symLemma3} of Lemma \ref{lem:lyap:symLemma} with $f(x, y) = \nabla \intPot(x - y)$, together with the fact that $\nabla \intPot$ is odd, we obtain 
\begin{equation}
\begin{aligned}
\bbE\Big[ \nabla \intPot\big( X_0(s) - X_v(s)\big) \cdot \nabla_{v} \log \mu_s\big(\X(s)\big)\Big] = - \bbE\Big[ \nabla \intPot\big( X_0(s) - X_v(s)\big) \cdot \nabla_{0} \log \mu_s\big(\X(s)\big)\Big]
\end{aligned}
    \label{eq:lyap:gradRho_gradK}
\end{equation}
Combining \eqref{eq:results:drift}, \eqref{eq:lyap:gGradient}, \eqref{eq:lyap:newDeltas}, and \eqref{eq:lyap:gradRho_gradK}, it follows that
\begin{equation}
    \begin{aligned}
    \Upsilon^{3, 1}(s) &= \bbE \bigg[ \bigg( b\big(\X(s)\big) + \nabla_{0} g\big(\X(s)\big) + \frac{1}{2} \sum_{v = 1}^\kap \nabla \intPot\big(X_0(s) - X_v(s)\big)\bigg) \cdot \nabla_{0} \log \mu_s\big(\X(s)\big) \bigg] \\
    &=  2\bbE\Big[  b\big( \X(s) \big) \cdot \nabla_{0} \log \mu_s\big(\X(s)\big)\Big].
    \end{aligned}
    \label{eq:lyap:U3_1}
\end{equation}
Since $\mu$ is a linear growth solution, by \eqref{eq:res:gammaLinGrowth} and Proposition \ref{thm:ap1:wellPosed}(1) the condition \eqref{eq:lyap:symLemma1} holds with $f(x, y) = \gamma(s, y, x)$. Therefore, \eqref{eq:lyap:symLemma3} of Lemma \ref{lem:lyap:symLemma} with $f(x, y) = \gamma(s, y, x)$ and $u = 0$, together with \eqref{eq:not:symmetry2_}, \eqref{eq:not:gamma}, and the tower property, imply that for every $v \in \{1, \ldots, \kap\}$ we have
\begin{equation*}
\begin{aligned}
    \bbE\Big[\gamma\big(s, X_v(s), X_0(s)\big)\cdot \nabla_{v}\log \mu_s\big(\X(s)\big)\Big] 
    &= \bbE \Big[\gamma\big(s, X_0(s), X_1(s)\big) \cdot \nabla_{0} \log {\bar{\mu}}_s\big(X_0(s), X_1(s)\big) \Big] \\
    &= \bbE \Big[b\big(\X(s)\big) \cdot \nabla_{0} \log {\bar{\mu}}_s\big(X_0(s), X_1(s)\big) \Big],
    \end{aligned}
\end{equation*}
Combining the above display with the definition of $\Upsilon^{3, 2}(s)$ in \eqref{eq:lyap:newDeltas} yields
\begin{equation}
\Upsilon^{3, 2}(s) = \kap \bbE \Big[b\big(\X(s)\big) \cdot \nabla_{0} \log {\bar{\mu}}_s\big(X_0(s), X_1(s)\big) \Big].
    \label{eq:lyap:U3_2}
\end{equation}
Therefore, by \eqref{eq:lyap:U3}, \eqref{eq:lyap:U3_1}, and \eqref{eq:lyap:U3_2}, it follows that
\begin{equation}
\label{eq:pf:term2}
\begin{aligned}
 \Upsilon_3(s) = \bbE\Big[b\big( \X(s) \big)\cdot \Big(2 \nabla_{0} \log \mu_s\big( \X(s) \big) + \kappa \nabla_{0}\log {\bar{\mu}}_s\big(X_0(s), X_1(s)\big)\Big)\Big].
\end{aligned}
\end{equation}
Substituting \eqref{eq:pf:term3}, \eqref{eq:pf:term1}, and \eqref{eq:pf:term2}, back into \eqref{eq:pf:uncondLyapu}, we obtain
\begin{equation}
    \label{eq:pf:uncondLyapunov}
    \begin{aligned}
            \Upsilon = - \int_0^t\bbE\bigg[ \Big|b\big( \X(s) &\big)  + \nabla_{0} \log \mu_s\big( \X(s) \big)\Big|^2\bigg] \\&+ \kappa \bbE\bigg[ \Big|\nabla_{1} \log \mu_s\big( \X(s) \big)\Big|^2 + b\big( \X(s) \big) \cdot \nabla_{0} \log  {\bar{\mu}}_s\big(X_0(s), X_1(s)\big)\bigg] ds.        
    \end{aligned}
\end{equation}

Next, we compute $\bar{\Upsilon}$. By \eqref{eq:lyap:EnergyC} and Lemma \ref{prop:lyap:pdeProp} we have
    \begin{equation}
    \label{eq:pf:condLyapunov}
    \begin{aligned}
        \bar{\Upsilon} = - \int_r^t \Big(\bar{\Upsilon}^1(s) + \bar{\Upsilon}^2(s) \Big) ds,
    \end{aligned}
    \end{equation}
    where, recalling the definition of $\bar{\eta}$ from \eqref{eq:lyap:barEta}, we define
    \begin{align*}
        \bar{\Upsilon}^1(s) :&= \bbE\Big[ \big|\nabla_{(0, 1)} \log {\bar{\mu}}_s\big(X_0(s), X_1(s)\big)\big|^2\Big], \\
        \bar{\Upsilon}^2(s) :&= \bbE\Big[\bar{\eta}\big(s, X_0(s), X_1(s))\big) \cdot \nabla_{(0, 1)} \log \bar{\mu}\big(X_0(s), X_1(s)\big) \Big].
    \end{align*}
Since $\mu$ is a linear growth solution to the \CMVE{}, Lemma \ref{lem:wp:marginalSDE} implies that $\nabla_{(0, 1)} \log \bar{\mu}_s(x_0, x_1) \in L^2(\mu_s)$ for almost every $s \in \R_+$. Therefore \eqref{eq:lyap:symLemma1} holds with $f(x, y) = \nabla_{u} \log \bar{\mu}_s(x, y)$ for $u \in (0, 1)$, and by \eqref{eq:lyap:symLemma3} of Lemma \ref{lem:lyap:symLemma} we have
    \begin{equation}
        \begin{aligned}
       \bar{\Upsilon}^1(s) = \bbE\Big[ \big|\nabla_{(0, 1)} \log {\bar{\mu}}_s\big(X_0(s), X_1(s)\big)\big|^2\Big] = 2  \bbE\Big[ \big|\nabla_{1} \log {\bar{\mu}}_s\big(X_0(s), X_1(s)\big)\big|^2\Big].            
        \end{aligned}
        \label{eq:pf:term4}
    \end{equation}
    Next, we simplify the expression for $\bar{\Upsilon}^2(s)$.  Note that by \eqref{eq:lyap:barEta},
    \begin{equation}
    \begin{aligned}
        \bar{\Upsilon}^2(s) = \bbE\Big[ \gamma\big(s, X_1(s), X_0(s) \big) \cdot\nabla_{1}\log \bar{\mu}_s\big(X_0(s), X_1(s)\big) + b\big(\X(s)\big) \cdot \nabla_{0} \log \bar{\mu}_s\big(X_0(s), X_1(s)\big) \Big].
    \end{aligned}
    \label{eq:lyap:UC_1}
    \end{equation}
   We can apply \eqref{eq:lyap:symLemma3} with $f(x, y) = \gamma(s, y, x)$ (justified as before by Proposition \ref{thm:ap1:wellPosed}), \eqref{eq:not:gamma}, and the tower property to obtain
    \begin{equation*}
        \begin{aligned}
            \bbE\Big[ \gamma\big(s, X_1(s), X_0(s)\big) \cdot\nabla_{1} \log {\bar{\mu}}_s\big(X_0(s), X_1(s)\big)
        \Big] 
        & = \bbE\Big[ \gamma\big(s, X_0(s), X_1(s)\big) \cdot\nabla_{0} \log {\bar{\mu}}_s\big(X_0(s), X_1(s)\big)
        \Big] \\
        &= \bbE\Big[ b\big(\X(s)\big) \cdot \nabla_{0}  \log {\bar{\mu}}_s\big(X_0(s), X_1(s)\big)
        \Big].
        \end{aligned}
    \end{equation*}
    The last two displays when combined yields
    \begin{equation}
        \bar{\Upsilon}^2(s) = 2 \bbE\Big[b\big( \X(s) \big)\cdot \nabla_{0} \log {\bar{\mu}}_s\big(X_0(s), X_1(s)\big) \Big] . 
        \label{eq:pf:term5}
    \end{equation}
    Next, substitute \eqref{eq:pf:term4} and \eqref{eq:pf:term5} into \eqref{eq:pf:condLyapunov} to obtain
    \begin{equation}
    \begin{aligned}
                        \frac{\kap}{2}\bar{\Upsilon} = -\kap \int_0^t  \bbE\Big[\big|\nabla_{1} \log {\bar{\mu}}_s\big(X_0(s), X_1(s)\big)\big|^2  + b\big( \X(s) \big) \cdot \nabla_{0} \log {\bar{\mu}}_s\big(X_0(s), X_1(s)\big) \Big]ds.
    \end{aligned}
                \label{eq:pf:condLyapu}
    \end{equation}
    Then \eqref{eq:pf:uncondLyapunov}, \eqref{eq:pf:condLyapu}, and \eqref{eq:pf:LyapunovSplit} together imply that
    \begin{align*}
        \energy_\kap(\mu_t) - \energy _\kap(\mu_r)
        = -\int_r^t\Big(\bbE&\Big[ \big|b\big( \X(s) \big) + \nabla_{0} \log \mu_s\big( \X(s) \big)\big|^2\Big] \\&+ \kappa \bbE\Big[\big|\nabla_{1}\log\mu_s\big( \X(s) \big)\big|^2\Big]  - \kappa \bbE\Big[\big|\nabla_{1} \log {\bar{\mu}}_s\big(X_0(s), X_1(s)\big)\big|^2\Big]\Big)ds
    \end{align*} 
    
    To finish the proof of \eqref{eq:results:energyIdentity}, it suffices to show that
     \begin{equation*}
         \info_\kap(\mu_s) = \bbE\Big[ \big| b\big(\X(s)\big) + \nabla_{0} \log \mu_s\big(\X(s)\big) \big|^2 + \kap \Big( \big| \nabla_{1} \log \mu_s\big(\X(s)\big) \big|^2 - \big|\nabla_{1} \log {\bar{\mu}}_s\big(X_0(s), X_1(s)\big) \big|^2\Big)\Big].
     \end{equation*}
    Let $\delta$ denote the absolute difference of the two sides of the above display. By \eqref{eq:results:info}, we have
    \begin{equation*}
        \delta = 2\kap \bigg(\bbE\Big[\big|\nabla_{1} \log \bar{\mu}_s\big(X_0(s), X_1(s)\big)\big|^2\Big] - \bbE\Big[\nabla_{1} \log \bar{\mu}_s\big(X_0(s), X_1(s)\big)\cdot  \nabla_{1} \log \mu_s\big(\X(s)\big)\Big]\bigg).
    \end{equation*}
    Once again, we can apply \eqref{eq:lyap:symLemma2} of Lemma \ref{lem:lyap:symLemma} with $f(x, y) = \nabla_{1} \log \bar{\mu}_s(x, y)$ and $u = 1$ to conclude that we have $\delta = 0$. This completes the proof of \eqref{eq:results:energyIdentity}.
\end{proof}

\subsection{Lower bound and level sets of $\bbH_\kap$} 

We now prove Proposition \ref{prop:results:energyBound}. 
\label{sec:levelSets}
\begin{proof}[Proof of Proposition \ref{prop:results:energyBound}]
We first write the sparse free energy in a slightly different form that shares similarities with a good rate function identified in \cite{ChenRamYas}. Let $q$ and $R_q$ be as given in Assumption \ref{as:results:invMeasure}. Define the probability measure $\alpha_q \in \mP(\R^d)$ by \begin{equation}
    \alpha_q(dx) := \frac{1}{R_q} e^{-q(x)}dx.
    \label{eq:lyap:alpha}
\end{equation}
    Fix $\nu \in \mPSb$. By definition of $\mPSb$, $\nu$ has finite entropy and hence, by the results of \cite{ChenRamYas}, the map $\nu \mapsto \entropy( \nu | \alpha_q^{\otimes(1 + \kap)} ) - \frac{\kap}{2} \entropy( \bar{\nu} | \alpha_q^{\otimes 2})$ is a rate function. Then we have \begin{equation}
\entropy\big( \nu \big| \alpha_q^{\otimes(1 + \kap)} \big) - \frac{\kap}{2} \entropy\big( \bar{\nu} \big| \alpha_q^{\otimes 2}\big) \geq 0, \quad \nu \in \mPS.
\label{eq:lyap:rateFunctionPositive}
\end{equation}

Since $\nu$ is an element of $\mPSb$ and therefore has has finite entropy, we have
\begin{equation*}
    \begin{aligned}
        &\int_{(\R^d)^{1 + \kap}} \nu(\x) \log \nu(\x) dx = \entropy\big( \nu \big| \alpha_q^{\otimes(1 + \kap)} \big) - (\kap + 1) \Big( \bbE^\nu[q(Y_0)] + \log R_q\Big), \\
        &\int_{\R^d \times \R^d} \bar{\nu}(x_0, x_1) \log \bar{\nu}(x_0, x_1) dx_0 dx_1 = \entropy\big(\bar{\nu}\big| \alpha_q^{\otimes 2}\big) - 2\Big(\bbE^\nu[q(Y_0)] + \log R_q\Big).
    \end{aligned}
\end{equation*}
The last displays, together with \eqref{eq:results:potential} and \eqref{eq:results:entropy}, yields
\begin{equation}
\label{eq:lyap:alternativeEnergy}
    \energy_\kap(\nu) = 
\entropy\big( \nu \big| \alpha_q^{\otimes(1 + \kap)} \big) - \frac{\kap}{2} \entropy\big( \bar{\nu} \big| \alpha_q^{\otimes 2}\big) - \log R_q + \bbE^{\nu}[g(\mathbf{Y}) - q(Y_0)].
\end{equation}
By the definition of $g$ in \eqref{eq:results:potential}, it follows that
\begin{equation*}
    \bbE^\nu[g(\Y) - q(Y_0)] = \bbE^\nu\bigg[ U(Y_0) - q(Y_0) + \frac{1}{2}\sum_{v = 1}^\kap \intPot(Y_0 - Y_v)\bigg].
\end{equation*}
Since $\nu \in \mPS$, by \eqref{eq:not:symmetry2_} we have
\begin{equation*}
    \bbE^\nu[h(Y_0)] = \frac{1}{2}\bbE^\nu[h(Y_0) + h(Y_1)],
\end{equation*}
for $h = U$ and $h = q$. Similarly, by  \eqref{eq:not:symmetry1_}, we have
\begin{equation*}
\bbE^\nu\bigg[ \frac{1}{2}\sum_{v = 1}^\kap \intPot(Y_0 - Y_v)\bigg] = \frac{\kap}{2} \bbE^\nu[\intPot(Y_0 - Y_1)].
\end{equation*}
Define the function $L:\R^d \times \R^d \rightarrow \R$ by
\begin{equation}
    L(x, y) := U(x) + U(y) + \kap \intPot(x - y) - q(x) - q(y) \geq 0, \quad x, y \in \R^d,
    \label{eq:lyap:L}
\end{equation}
where the nonnegativity of $L$ follows from \eqref{eq:results:fBound}. The last four displays together imply that
\begin{equation}
\label{eq:lyap:gqNonNegative}
    \begin{aligned}
        \bbE^\nu[g(\X) - q(X_0)] = \frac{1}{2} \bbE^{\bar{\nu}}\big[ L(X_0, X_1)\big] \geq 0.
    \end{aligned}
\end{equation}

Combining \eqref{eq:lyap:alternativeEnergy} and \eqref{eq:lyap:L}, we have
\begin{equation}
    \bbH_\kap(\nu) = \entropy\big( \nu \big| \alpha_q^{\otimes(1 + \kap)} \big) - \frac{\kap}{2} \entropy\big( \bar{\nu} \big| \alpha_q^{\otimes 2}\big) - \log R_q + \frac{1}{2}\bbE^{\bar{\nu}}[L(Y_0, Y_1)]
    \label{eq:lyap:altSFE}
\end{equation}
By \eqref{eq:lyap:altSFE}, \eqref{eq:lyap:rateFunctionPositive}, and \eqref{eq:lyap:gqNonNegative}, we obtain $\bbH_\kap(\nu) \geq - \log R_q$ for all $\nu \in \mPSb$. The lower semicontinuity of $\bbH_\kap$ then follows fro \eqref{eq:lyap:altSFE}:
\begin{itemize}
    \item[(i)] The map $\nu \mapsto \mH(\nu|\alpha^{\otimes (1 + \kap)}) - \tfrac{\kap}{2} \mH(\bar{\nu}|\alpha^{\otimes 2})$ is lower semicontinuous by the results of \cite{ChenRamYas}
    \item[(ii)] The map $\nu \mapsto \bbE^{\bar{\nu}}[L(X_0, X_1)]$ is lower semicontinuous by \eqref{eq:lyap:L} and Fatou's lemma.
\end{itemize}

Next we turn to the proof of \eqref{eq:conv:secondMoment}. Suppose $(U, \intPot)$ satisfy Assumption \ref{as:results:strongCoercivity} and fix $\nu \in \mathcal{R}(M)$. The fact that $\nu \in \mPS$, together with \eqref{eq:zeros:strongCoercivity} and \eqref{eq:lyap:L}, implies that
\begin{equation*}
\begin{aligned}
 \int_{(\R^d)^{1 + \kap}} |\x|^2 \nu(d\x) = \frac{1 + \kap}{2} \int_{(\R^d)^{1 + \kap}} \big(|x_0|^2 + |x_1|^2\big)  \bar{\nu}(dx_0dx_1) \leq \ \frac{1 + \kap}{2\tilde{C}_q}\bbE^{\bar{\nu}}[L(Y_0, Y_1)].
\end{aligned}
\end{equation*}
Combining the above display with \eqref{eq:lyap:altSFE} and \eqref{eq:lyap:rateFunctionPositive}, we have 
\begin{equation*}
 \int_{(\R^d)^{1 + \kap}} |\x|^2 \nu(d\x) \leq \frac{1+ \kap}{2 \tilde{C}_q} \big(\bbH_{\kap}(\nu) + \log R_q\big) \leq \frac{1+ \kap}{2 \tilde{C}_q} (M + \log R_q),
\end{equation*}
where the last inequality follows from $\nu \in \mathcal{R}(M)$. Thus, \eqref{eq:conv:secondMoment} holds with $C_{q, M} = \frac{(1 + \kap)(M + \log R_q)}{2\tilde{C}_q}$. 

Finally, we establish \eqref{eq:conv:levelsets} and the compactness of $\mathcal{R}(M)$. In light of \eqref{eq:conv:secondMoment}, \eqref{eq:conv:levelsets} follows from showing elements of $\mathcal{R}(M)$ have finite entropy. We have by \eqref{eq:lyap:altSFE} and the non-negativity of $L$ that
\begin{equation}
    \entropy\big( \nu \big| \alpha_q^{\otimes(1 + \kap)} \big) - \frac{\kap}{2} \entropy\big( \bar{\nu} \big| \alpha_q^{\otimes 2}\big) \leq \log R_q + \bbH_\kap(\nu) \leq \log R_q + M.
\end{equation}
Then by the results of \cite{ChenRamYas}, the last display implies that $\nu$ has finite entropy. Thus $\nu \in \mPSb$ and \eqref{eq:conv:levelsets} holds. To prove that $\mathcal{R}(M)$ is compact, we note that by \eqref{eq:conv:secondMoment}. $\mathcal{R}(M)$ is uniformly integrable and thus tight (e.g., Problem 1.5.8 in \cite{billingsley1968convergence}). The desired compactness then follows from Prokhorov's theorem.
\end{proof}

\section{Stationary Distributions of the \CMVE{}}
\label{sec:lyap:zeros}

\subsection{\FPES{}}\label{sec:stationaryProofs} In this section we prove Theorem \ref{thm:results:zeros}. Recall that $\mS_\kap$ is the set of zeros of $\info_\kap$ defined in \eqref{eq:results:mE}. We first note in the following lemma that \FPES{} satisfy mild regularity properties. The identity \eqref{eq:zeros:changeDandI} justifies an exchange of integral and derivative that will be useful later.

\begin{lemma}
\label{lem:zeros:betheRegularity} Let $(U, \intPot)$ satisfy Assumption \ref{as:results:invMeasure}. Suppose $\nu \in \mPS$ is a \FPE{}. Then $\nu$ along with its marginal distributions $\bar{\nu}$, $\nu_0$, and its conditional distribution $\bar{\nu}(\cdot|\cdot)$ are differentiable Lebesgue almost everywhere. Moreover, for almost every $x \in \R^d$ we have
\begin{equation}
    \nabla \int_{\R^d} e^{-\intPot(x - y) - \frac{1}{\kap}U(y)}\nu_0(y)^{\frac{\kap - 1}{\kap}} dy = - \int_{\R^d} \nabla \intPot(x - y) e^{-\intPot(x - y) - \frac{1}{\kap}U(y)}\nu_0(y)^{\frac{\kap - 1}{\kap}} dy.
    \label{eq:zeros:changeDandI}
\end{equation}
\end{lemma}
\begin{proof}
Assumption \ref{as:results:lyapunov} and Assumption \ref{as:results:invMeasure} imply that the following estimate holds all compact sets $K \subset \R^d$:
\begin{equation}
    \int_K \bigg(\int_{\R^d} |\nabla \intPot(x - y)|^\kap e^{- \kap \intPot(x - y) - U(x)} dx\bigg)^{\frac{1}{\kap}} dy < \infty.
    \label{eq:zeros:apriori}
\end{equation}
Combining Holder's inequality with \eqref{eq:zeros:apriori}, it follows that
\begin{equation*}
    \exp\Big( - \tfrac{1}{\kap}U(y) - \intPot(x - y) \Big) \nu_0(y)^{ \frac{\kap-1}{\kap}}, \, \nabla \intPot(x - y) \exp\Big( - \tfrac{1}{\kap}U(y) - \intPot(x - y) \Big) \nu_0(y)^{ \frac{\kap-1}{\kap}} \in L^1_{\text{loc}}\big((\R^d)^2\big).
\end{equation*}
Then by Fubini's theorem, \eqref{eq:zeros:changeDandI} is satisfied.
Moreover, \eqref{eq:zeros:changeDandI} together with \eqref{eq:zeros:fixedPoint}, \eqref{eq:zeros:apriori}, and Holder's inequality imply that $\nu_0$ is differentiable almost everywhere. By \eqref{eq:zeros:muBar}, we have that $\bar{\nu}(x_v, x_0)$ is also differentiable almost everywhere in $(x_v, x_0)$. By \eqref{eq:zeros:muBar}, we have, that
\begin{equation*}
    \log \bar{\nu}(x_v|x_0) = - \log \mathcal{Z}_{\nu_0} - \frac{U(x_0) + U(x_v)}{\kap} - \intPot(x_0 - x_v) + \frac{\kap - 1}{\kap} \log \nu_0(x_v) - \frac{1}{\kap} \log \nu_0(x_0).
\end{equation*} The above display and the differentiability of $\nu_0$ imply that $\bar{\nu}(x_v|x_0)$ is also differentiable almost everywhere in $(x_v, x_0)$.     
\end{proof}
We now turn to the proof of Theorem \ref{thm:results:zeros}.
\begin{proof}[Proof of Theorem \ref{thm:results:zeros}] The proof will follow from the following three claims.

\textbf{Claim 1:} A measure $\nu \in \mPS$ satisfies $\bbI_\kap(\nu) = 0$ if and only if $\nu$ is an absolutely continuous 1-MRF of the form \eqref{eq:zeros:1MRF} satisfying
\begin{equation}
    \label{eq:zeros:info2}
    \nabla_{x_v}\big(\log \nu(\x) - \log \bar{\nu}(x_0, x_v)\big) = 0,
\end{equation}
for all $v \in \{1, \ldots, \kap\}$ and
\begin{equation}
    \label{eq:zeros:info1}
    \nabla U(x_0) + \sum_{v = 1}^\kap \nabla \intPot(x_0 - x_v) + \nabla_{x_0} \log \nu_0(x_0) + \sum_{v = 1}^\kap \nabla_{x_0} \log \bar{\nu}(x_v|x_0) = 0,
\end{equation}
for $\nu$-a.e. $\x \in (\R^d)^{1 + \kap}$.

\begin{proof}[Proof of Claim 1]
Suppose first that $\nu \in \mPS$ is an absolutely continuous 1-MRF of the form \eqref{eq:zeros:1MRF} satisfying \eqref{eq:zeros:info2} and \eqref{eq:zeros:info1}. By \eqref{eq:zeros:1MRF}, we have \begin{equation}
    \nabla_{x_0} \log \nu(\x) = \nabla_{x_0} \log \nu_0(x_0) + \sum_{v = 1}^\kap \nabla_{x_0} \bar{\nu}(x_v|x_0), \quad \nu\text{-a.e. }\x \in (\R^d)^{1 + \kap}.
    \label{eq:zeros:1MRF2}
\end{equation}
Then the above display, \eqref{eq:zeros:info1}, \eqref{eq:zeros:info2}, and \eqref{eq:results:info} imply that $\bbI_\kap(\nu) = 0$.

Now suppose that $\nu \in \mPS$ satisfies $\bbI_\kap(\nu) = 0$. By \eqref{eq:results:info}, $\nu$ is absolutely continuous. Moreover, by \eqref{eq:results:info} and the fact that $\nu \in \mPS$, for all $v \in \{1, \ldots, \kap\}$ we have 
\begin{equation*}
    \int_{(\R^d)^{1 + \kap}} \bigg|\nabla_{x_1} \log \frac{\nu(\x)}{\bar{\nu}( x_0, x_1)}\bigg|^2 \nu(d\x) =   \int_{(\R^d)^{1 + \kap}} \bigg|\nabla_{x_v} \log \frac{\nu(\x)}{\bar{\nu}( x_0, x_v)}\bigg|^2 \nu(d\x) = 0.
\end{equation*}
Therefore, \eqref{eq:zeros:info2} is satisfied. Fix $v \in \{1, \ldots, \kap\}$. By \eqref{eq:zeros:info2} it follows that
\begin{equation*}
    \nabla_{x_v} \log \nu( x_{\{1, \ldots, \kap\} \backslash \{v\}} | x_0, x_v) = 0,
\end{equation*}
which yields
\begin{equation*}
        \nu( x_{\{1, \ldots, \kap\} \backslash \{v\}} | x_0, x_v) = \nu( x_{\{1, \ldots, \kap\} \backslash \{v\}} | x_0).
\end{equation*}
Fix $A \subset \{1, \ldots, \kap\} \backslash \{v\}$.
On integration, the last display implies that we have
\begin{equation}
    \label{eq:zeros:preMRF}
        \nu( x_{A} | x_0, x_v) = \nu( x_{A} | x_0).
\end{equation}
Fix $B \subset \{0, \ldots, \kap\}$ such that $\rt \in B$ and $A \cap B = \emptyset$, let $m := |B| - 1$, and write $B = \{x_0, x_{b_1}, \ldots, x_{b_m}\}$. Then repeated applications of Bayes' formula and \eqref{eq:zeros:preMRF} yields
\begin{equation*}
    \nu(x_A|x_B) = \nu(x_A | x_0).
\end{equation*}
Therefore, $\nu$ is a 1-MRF on $\mathbb{T}_{\kap}^1$. By the Hammersley-Clifford theorem (e.g. Theorem 2.30 of \cite{georgii1988gibbs}) and the symmetry properties \eqref{eq:not:symmetry1_} and \eqref{eq:not:symmetry2_}, we see $\nu$ is of the form \eqref{eq:zeros:1MRF} for some $\nu_0 \in \mP(\R^d)$ and $\bar{\nu} \in \mP( \R^d \times \R^d)$. By \eqref{eq:results:info} and \eqref{eq:results:drift}, we have
\begin{equation*}
    \int_{(\R^d)^{1 + \kap}} \bigg| \nabla U(x_0) + \sum_{v = 1}^\kap \nabla W(x_0 - x_v) + \nabla_{x_0} \log \nu(\x) \bigg|^2 \nu(d\x) = 0.
\end{equation*}
Then \eqref{eq:zeros:info1} follows by substituting \eqref{eq:zeros:1MRF2} into the above display.
\end{proof}

\textbf{Claim 2:} If $\nu \in \mPS$  is a \FPE{}, then $\bbI_\kap(\nu) = 0$.

\begin{proof}[Proof of Claim 2]
 Let $\nu \in \mPS$ solve \FPE{}. Lemma \ref{lem:zeros:betheRegularity} implies that $\nu$ and its associated conditional and marginal distributions are differentiable Lebesgue almost everywhere. By Claim 1, it suffices to show that $\nu$ satisfies \eqref{eq:zeros:info2} and \eqref{eq:LSI:info1}. The condition \eqref{eq:zeros:info2} is therefore clearly satisfied from \eqref{eq:zeros:1MRF}. To check \eqref{eq:zeros:info1}, observe from \eqref{eq:zeros:muBar} that
\begin{equation*}
    \nabla_{x_0} \log \bar{\nu}(x_v|x_0) = - \frac{1}{\kap} \nabla U(x_0) - \nabla \intPot(x_0 - x_v) - \frac{1}{\kap} \nabla_{x_0} \log \nu_0(x_0), \quad \bar{\nu}\text{-a.e. }(x_0, x_v) \in \R^d \times \R^d.
\end{equation*}
Thus we see that \eqref{eq:zeros:info1} is satisfied, and $\info_\kap(\nu) = 0$. \end{proof}
\textbf{Claim 3:} If $\nu \in \mPS$ satisfies $\bbI_\kap(\nu) = 0$, then $\nu$ is a \FPE{}. 
\begin{proof}[Proof of Claim 3]
We show that any $\nu \in \mPS$ with $\info_{\kap}(\nu) = 0$ satisfies \eqref{eq:zeros:1MRF}, \eqref{eq:zeros:fixedPoint}, and \eqref{eq:zeros:muBar}. Note that \eqref{eq:zeros:info1} implies that for $\nu$-a.e. $\x \in (\R^d)^{ 1 + \kap}$, 
\begin{equation*}
    U(x_0) + \log \nu_0(x_0) + \sum_{v = 1}^\kap \Big(\intPot(x_0 - x_v) + \log \bar{\nu}(x_v | x_0) \Big)
\end{equation*}
does not depend on $x_0$ and is thus a measurable function of $(x_1, \ldots, x_\kap)$, which we denote by $F$. Further, from the above display, we see that $F(x_1, \ldots, x_k) = w(x_1) + \cdots w(x_k)$ for some $w:\R^d \rightarrow \R$. Then we have
\begin{equation}
    \label{eq:zeros:algEqn1}
    \sum_{v = 1}^\kap h(x_0, x_v) = - U(x_0) - \log \nu_0(x_0), \quad \nu\text{-a.e. }\x \in (\R^d)^{1 + \kap},
\end{equation}
where
\begin{equation}
    \label{eq:zeros:h}
    h(x_0, x_v) := \intPot(x_0 - x_v) + \log \bar{\nu}(x_v | x_0) - w(x_v).
\end{equation}
By \eqref{eq:zeros:algEqn1}, for $\bar{\nu}$-a.e $(x_0, x_v)$ the function $h(x_0, x_v)$ only depends on $x_0$ and thus takes the form
\begin{equation*}
    h(x_0, x_v)  = - \frac{1}{\kap}\Big(U(x_0) + \log \nu_0(x_0)\Big) =: h(x_0), \quad \bar{\nu}\text{-a.e. } (x_0, x_v) \in \R^d.
\end{equation*}
By \eqref{eq:zeros:h} this implies that
\begin{equation*}
    \intPot(x_0 - x_v) + \log \bar{\nu}(x_v | x_0) - w(x_v) = - \frac{1}{\kap}\Big(U(x_0) + \log \nu_0(x_0)\Big),
\end{equation*}
which when rearranged yields
\begin{equation}
    \label{eq:zeros:conditional}
    \bar{\nu}(x_v | x_0) = e^{ - \intPot(x_0 - x_v) - \frac{1}{\kap}U(x_0)} \frac{e^{w(x_v)}}{\nu_0(x_0)^{\frac{1}{\kap}}},
\end{equation}
and hence
\begin{equation}
    \label{eq:zeros:muBar2}
    \bar{\nu}(x_v, x_0) = e^{ - \intPot(x_0 - x_v) - \frac{1}{\kap}U(x_0)}e^{w(x_v)}\nu_0(x_0)^{\frac{\kap - 1}{\kap}}.
\end{equation}
Since $\nu \in \mPS$, by \eqref{eq:not:symmetry2_}, we have $\bar{\nu}(x_v, x_0) = \bar{\nu}(x_0, x_v)$ and by \eqref{eq:zeros:muBar2}, it follows that
\begin{equation*}
    e^{ - \intPot(x_0 - x_v) - \frac{1}{\kap}U(x_0)}e^{w(x_v)}\nu_0(x_0)^{\frac{\kap - 1}{\kap}} = e^{ - \intPot(x_v - x_0) - \frac{1}{\kap}U(x_v)}e^{w(x_0)}\nu_0(x_v)^{\frac{\kap - 1}{\kap}}.
\end{equation*}
Since $\intPot$ is even by Assumption \ref{as:results:lyapunov}, we can rearrange terms to obtain
\begin{equation*}
    e^{ - \frac{1}{\kap}U(x_0)}e^{- w(x_0)}\nu_0(x_0)^{\frac{\kap - 1}{\kap}} = e^{ - \frac{1}{\kap}U(x_v)}e^{-w(x_v)}\nu_0(x_v)^{\frac{\kap - 1}{\kap}}.
\end{equation*}
Noticing that the left-hand side depends only on $x_0$ and the right-hand side depends only on $x_v$, we deduce that there exists a constant $C \in \R$ such that
\begin{equation*}
    e^{ - \frac{1}{\kap}U(x_v)}e^{- w(x_v)}\nu_0(x_v)^{\frac{\kap - 1}{\kap}} = C, \quad \text{for all }v \in \{0, 1, \ldots, \kap\}.
\end{equation*}
or equivalently,
\begin{equation*}
    e^{w(x_v)} = \frac{1}{C}e^{ - \frac{1}{\kap}U(x_v)}\nu_0(x_v)^{\frac{\kap - 1}{\kap}}, \quad \text{for all }v \in \{0, 1, \ldots, \kap\}.
\end{equation*}
Substituting the above relation back into \eqref{eq:zeros:muBar2} we obtain 
\begin{equation*}
    \bar{\nu}(x_v, x_0) =  \frac{1}{C}\exp\Big( - \tfrac{U(x_0) + U(x_v)}{\kap} - \intPot(x_0 -x_v)\Big) \big[\nu_0(x_0)\nu_0(x_v)\big]^{\frac{\kap-1}{\kap}}, \quad \text{for all }v \in \{0, 1, \ldots, \kap\}.
\end{equation*}
Since $\bar{\nu}$ is a probability measure, we see from \eqref{eq:zeros:zMu} that $C = \mathcal{Z}_{\nu_0}$. Therefore, $\bar{\nu}$ satisfies \eqref{eq:zeros:muBar}. To see that $\nu_0$ satisfies \eqref{eq:zeros:fixedPoint}, we integrate the above display over $x_0$ to conclude that
\begin{equation*}
    \nu_0(x_v) = \frac{1}{\mathcal{Z}_{\nu_0}} \nu_0(x_v)^{\frac{\kap - 1}{\kap}}e^{- \frac{1}{\kap}U(x_v)} \int_{\R^d} e^{- \intPot(x_0 -x_v) - \frac{1}{\kap}U(x_0)} \nu_0(x_0)^{\frac{\kap - 1}{\kap}} dx_0.
\end{equation*}
Dividing both sides by $\nu_0(x_v)^{\frac{\kap - 1}{\kap}}$ we conclude that $\nu_0$ satisfies \eqref{eq:zeros:fixedPoint}.
\end{proof}
The equivalence of $\mS_\kap$ and the set of \FPES{} then follows from Claim 2 and Claim 3.
\end{proof}
\subsection{Bijection between stationary distributions and zeros of $\bbI_\kap$}
\label{ss:statProof}

To establish the correspondence claimed in Theorem \ref{thm:results:stationary}, we will first establish regularity and integrability properties of the the zeros of $\bbI_\kap$ in the following proposition.

\begin{proposition}[Regularity of \FPES{}]
    \label{lem:zeros:entropyAndMoments}
    Suppose $(U, \intPot)$ satisfies Assumption \ref{as:results:fixedPoint}. If $\nu \in \mS_\kap$, then $\nu\in \mPSb$ and we have
    \begin{equation}
    \label{eq:zeros:marginalFisher}
        \int_{\R^d} \big| \nabla \log \nu_0(x)\big|^2 \nu_0(x) dx < \infty,
    \end{equation}
    Furthermore, $\nu$ is bounded, continuous, and positive, has moments of all orders, and we have for all $x \in \R^d$ that  \begin{equation}
\begin{aligned}
    \nabla \log \nu_0(x) 
    &= - \nabla U(x) - \kappa \bbE^{\bar{\nu}}\big[ \nabla \intPot(Y_0 - Y_1) | Y_0 = x\big].
\end{aligned}    
\label{eq:zeros:nablaMu0}
\end{equation}
\end{proposition}
\begin{proof}
Fix $\nu \in \mS_\kap$ and let $q$ and $R_q$ be the quantities in Assumption \ref{as:results:invMeasure} (which is implied by Assumption \ref{as:results:fixedPoint}). Since $\mathbb{I}_\kap(\nu) = 0$, by Proposition \ref{thm:results:zeros}, $\nu$ is a \FPE{} \eqref{eq:zeros:1MRF}-\eqref{eq:zeros:muBar}.

To establish $\nu \in \mPSb$, we first show $\nu$ is bounded $\nu_0^{\otimes (1 + \kap)}$-almost everywhere. Note that this almost everywhere bound will be hold everywhere once continuity of $\nu_0$ is established. By \eqref{eq:results:fBound}, H\"older's inequality, and the fact that $\nu_0$ is a probability density, we have
\begin{equation}
\begin{aligned}
    \nu_0(x)^{\frac{1}{\kap}} &\leq \frac{1}{\mathcal{Z}_{\nu_0}} \int_{\R^d} e^{- \frac{1}{\kap}(q(x) + q(y))} \nu_0(y)^{\frac{\kap-1}{\kap}} dy 
    \leq \frac{1}{\mathcal{Z}_{\nu_0}} \big(e^{-q(x)}R_q\big)^{\frac{1}{\kap}}, \quad \nu_0\text{-a.e. }x \in \R^d.
\end{aligned}
\label{eq:zeros:nu0bound}
\end{equation}
The uniform lower bound on $q$ then this implies that $\nu_0$ is uniformly bounded. By \eqref{eq:zeros:altCBP} and Assumption \ref{as:results:invMeasure}, we have
\begin{equation*}
    \nu(\x) \leq \frac{1}{\mathcal{Z}_{\nu_0}^\kap} \exp\bigg(- \frac{1}{\kap}\sum_{v=1}^\kap\big(q(x_0) + q(x_v)\big)\bigg) \prod_{v = 1}^\kap\nu_0(x_v)^{ \frac{\kap-1}{\kap}}, \quad \nu_0^{\otimes (1 + \kap)}\text{-a.e. }\x \in (\R^d)^{1 + \kap}.
\end{equation*}
Thus the lower bound on $q$ and upper bound on $\nu_0$ imply that $\nu$ is bounded $\nu_0^{\otimes (1 + \kap)}$-almost everywhere. Next we show $\nu$ has finite moments of all orders. Since $\nu \in \mPS$, by Remark \ref{rk:not:exchangeability}, \eqref{eq:zeros:nu0bound}, and \eqref{eq:results:fMoments}, there exists $C \in(0,\infty)$ such that for all $p \in [1, \infty)$, we have
\begin{equation}
    \int_{(\R^d)^{1 + \kap}} |\x|^p \nu(d\x) = \kap \int_{\R^d} |x_0|^p \nu_0(dx_0) \leq C \int_{\{x\,:\, \nu_0(x) > 0\}} |x|^p e^{-q(x)} dx < \infty. 
    \label{eq:zeros:2ndMoment}
\end{equation}

Next, we show that $\nu$ has finite entropy. Along the way, we will show that the identity \eqref{eq:zeros:nablaMu0} holds for $\nu_0$-a.e. $x \in \R^d$, which will extend to all $x \in \R^d$ once continuity and positivity of $\nu_0$ is established. By \eqref{eq:zeros:1MRF} it follows that
\begin{equation*}
    \begin{aligned}
        \int_{(\R^d)^{1 + \kap}} \log \nu(\x) d\nu(\x) = \kap \int_{\R^d \times \R^d} \log \bar{\nu}(x_0, x_1) \bar{\nu}(dx_0, dx_1) - (\kap - 1) \int_{\R^d} \log \nu_0(x) d\nu_0(x).
    \end{aligned}
\end{equation*}
On the other hand, by \eqref{eq:zeros:muBar} and \eqref{eq:results:potentialGrowth}, for some $C_1 \in (0, \infty)$ we have
\begin{equation*}
\begin{aligned}
        &\bigg|\int_{\R^d \times \R^d} \log \bar{\nu}(x_0, x_1) \bar{\nu}(dx_0, dx_1)\bigg| \\ = &\bigg|\int_{\R^d \times \R^d} \bigg( \log \mathcal{Z}_{\nu_0} - \frac{U(x_0) + U(x_1)}{\kap} - \intPot(x_0 -x_1) + \frac{\kap-1}{\kap}\big(\log\nu_0(x) + \log \nu_0(y)\big)\bigg) \bar{\nu}(dx_0, dx_1)\bigg| \\
        \leq  &C_1 \bigg( 1 + \int_{\R^d} |x|^2 \nu_0(dx) + \bigg|\int \log \nu_0(x) \nu_0(dx)\bigg|\bigg).
\end{aligned}
\end{equation*}
Then the last two displays show that $\nu$ and $\bar{\nu}$ have finite entropy if $\nu_0$ has finite entropy and second moment. By \eqref{eq:zeros:fixedPoint} and \eqref{eq:zeros:muBar}, for $\nu_0\text{-a.e. }x,z\in \R^d$ we have
\begin{equation*}
    \begin{aligned}
        \bar{\nu}(z|x) &= \frac{1}{\mathcal{Z}_{\nu_0}} \exp \bigg( - \frac{U(x) + U(z)}{\kap} - \intPot(x - z) \bigg) \nu_0(z)^{\frac{\kap-1}{\kap}}\nu_0(x)^{ -\frac{1}{\kap}}
        \\
        &= \frac{\exp\Big( - \tfrac{1}{\kap}U(z) - \intPot(x - z) \Big) \nu_0(z)^{ \frac{\kap-1}{\kap}}}{\int_{\R^d} \exp\Big( - \tfrac{1}{\kap}U(y) - \intPot(x - y) \Big) \nu_0(y)^{ \frac{\kap-1}{\kap}} dy }.
    \end{aligned}
\end{equation*}
Combining the previous display with \eqref{eq:zeros:changeDandI}, it follows for $\nu_0$-a.e. $x \in \R^d$ that
\begin{equation}
    \begin{aligned}
        \bbE^{\bar{\nu}} \big[ \nabla \intPot(Y_0 - Y_1) | Y_0 = x\big] &= \int_{\R^d} \nabla \intPot(x - z) \bar{\nu}(z |x) dz \\
        &= - \nabla \bigg[ \log \int_{\R^d} \exp\Big( - \tfrac{1}{\kap}U(z) - \intPot(x - z) \Big) \nu_0(z)^{ \frac{\kap-1}{\kap}} dz  \bigg].
    \end{aligned}
    \label{eq:zeros:gradK}
\end{equation}
By \eqref{eq:zeros:fixedPoint}, we have
\begin{equation*}
    \frac{1}{\kap}\log \nu_0(x) = - \log \mathcal{Z}_{\nu_0} - \frac{1}{\kap}U(x) + \log \int_{\R^d} e^{-\intPot(x - y) - \frac{1}{\kap}U(y)}\nu_0(y)^{\frac{\kap-1}{\kap}}dy, \quad \nu_0\text{-a.e. }x\in \R^d.
\end{equation*}
Taking the gradient of the above display and substituting in \eqref{eq:zeros:gradK} yields \eqref{eq:zeros:nablaMu0} for $\nu_0$-a.e. $x \in \R^d$. Together with Assumption \ref{as:results:lyapunov} and Assumption \ref{as:results:fixedPoint}, this implies the existence of $c \in (0, \infty)$ such that $|\log \nu_0(x)| \leq c( 1 + |x|^2)$ for
 $\nu_0$-a.e. $x \in \R^d$. Therefore $\nu_0$ and hence $\nu$ has finite entropy. This together with \eqref{eq:zeros:2ndMoment} shows that $\nu$ lies in $\mPSb$.

Next, we establish \eqref{eq:zeros:marginalFisher}. Since \eqref{eq:zeros:nablaMu0} holds for $\nu_0$-a.e. $x \in \R^d$, Assumption \ref{as:results:lyapunov} and Assumption \ref{as:results:fixedPoint} together imply that there exists $C_2 \in (0, \infty)$ such that
\begin{equation*}
    |\nabla \log \nu_0(x)|^2 = \big| \nabla U(x) + \kappa \bbE^{\bar{\nu}}\big[ \nabla \intPot(Y_0 - Y_1) | Y_0 = x\big] \big|^2 \leq C_2(1 + |x|^2), \quad \nu_0\text{-a.e. }x \in \R^d.
\end{equation*}
Therefore  \eqref{eq:zeros:marginalFisher} follows on integrating the above display with respect to $\nu_0$ and invoking \eqref{eq:zeros:2ndMoment}.

We now show that $\nu_0$ is continuous.  Once again, since \eqref{eq:zeros:nablaMu0} holds for $\nu_0$-a.e. $x \in \R^d$, Assumption \ref{as:results:lyapunov} and Assumption \ref{as:results:fixedPoint} together with the fact that $\nu_0$ is uniformly bounded $\nu_0$-a.e. yields the existence of $C_3 \in (0, \infty)$ such that
\begin{equation*}
    |\nabla \nu_0(x)| = \big| \nabla U(x) + \kappa \bbE^{\bar{\nu}}\big[ \nabla \intPot(Y_0 - Y_1) | Y_0 = x\big] \big| \nu_0(x) \leq C_3(1 + |x|), \quad \nu_0\text{-a.e. }x \in \R^d.
\end{equation*}
By Corollary 9.3 of \cite{bobkov2022fisher}, the Cauchy-Schwarz inequality, the finite Fisher information of $\nu_0$ in \eqref{eq:zeros:marginalFisher}, and the finite moments of $\nu_0$ in \eqref{eq:zeros:2ndMoment}, the above display implies that for $p \in [1, \infty)$ we have
\begin{equation*}
    \begin{aligned}
    \int_{\R^d} |\nabla \nu_0(x)|^p dx &= \int_{\{x\,:\,\nu_0(x) > 0\}} |\nabla \nu_0(x)|^p dx \\ &\leq \bigg[\int_{\{x\,:\,\nu_0(x) > 0\}} |\nabla \log \nu_0(x)|^2 \nu_0(dx)\bigg]^{\frac{1}{2}} \bigg[\int_{\{x\,:\,\nu_0(x) > 0\}} |\nabla \nu_0(x)|^{2p-2} \nu_0(dx)\bigg]^{\frac{1}{2}} \\ &< \infty.
    \end{aligned}
\end{equation*}
Therefore $\nu_0 \in W^{1, p}(\R^d)$ for all $p \in [1, \infty)$, and so $\nu_0$ is continuous by Morrey's inequality (e.g., see Theorem 5.6.4 of \cite{evans1998PDE}). Thus $\nu_0$ and hence $\nu$ are bounded on all of $\R^d$ and $(\R^d)^{1 + \kap}$ respectively.

Finally, we show $\nu_0$ is positive. Since $\nu_0$ is a continuous probability density, there exists $\eps > 0$ and a compact subset $K$ of $\R^d$ such that $K_{\eps}$ has positive Lebesgue measure, and $\nu_0(x) > \eps$ for all $x \in K_{\eps}$. Recall that $B_R$ denotes the open ball of radius $R$ in $\R^d$. By \eqref{eq:zeros:fixedPoint} and the linear growth conditions on $(U, \intPot)$ in Assumption \ref{as:results:lyapunov}, for all $R \in (0, \infty)$ there exists a positive $\delta_R \in (0, \infty)$ such that 
\begin{equation*}
    \nu_0(x) \geq \frac{1}{\mathcal{Z}_{\nu_0}^\kap} e^{-U(x)} \bigg[ \eps^{\frac{\kap - 1}{\kap}} \int_{K_{\eps}} e^{-\frac{1}{\kap}U(y) - \intPot(x - y)} dy\bigg]^{\kap} \geq \delta_R, \quad \nu_0\text{-a.e. }x \in \bar{B}_R.
\end{equation*}
Since this holds $\nu_0$-a.e. and $\nu_0$ is continuous, either $\nu_0(x) = 0$ for all $x \in \bar{B}_R$ or $\nu_0(x) > 0$ for all $x \in \bar{B}_R$. Suppose towards contradiction that there exists $R^* \in (0, \infty)$ such that $\nu(x) = 0$ for all $x \in \bar{B}_{R^*}$. Then, for all $R > R_*$, the above dichotomy implies that $\nu(x) = 0$ on $\bar{B}_R$ for all $R > R^*$. Then Fatou's lemma implies
\begin{equation*}
    \int_{\R^d} \nu_0(x) dx \leq \liminf_{R \rightarrow \infty} \int_{\bar{B}_R} \nu_0(x) dx = 0.
\end{equation*}
However, this is impossible since $\nu_0$ is a probability measure. Thus $\nu_0(x) > 0$ for all $x \in \R^d$, and \eqref{eq:zeros:nablaMu0} holds everywhere on $\R^d$. \end{proof}

By Theorem \ref{thm:results:zeros}, we establish in the following corollary that the \FPES{} are in correspondence with an $\R^d$ version of the fixed point problem in Definition 1.3 of \cite{lacker2023stationary}. We note that the results  we invoke from \cite{lacker2023stationary} hold on $\R^d$ with essentially no change to the proofs therein.

\begin{definition}[Definition 1.3 of \cite{lacker2023stationary}]
    Let $F:\R^d \rightarrow \R$ be a measurable function $C \in \R$. We say $(F, C)$ solves the \emph{Lacker-Zhang fixed point problem} if
    \begin{align}
        &F(x) = C - \log \int_{\R^d} \exp\Big( - U(y) - \intPot(x - y)- (\kap - 1) F(y)\Big) dy, \quad \text{\normalfont a.e. }x \in \R^d, \label{eq:zeros:LZ1} \\
        &\int_{\R^d} \exp \big(- U(x) - \kap F(x) \big) dx < \infty. \label{eq:zeros:LZ2}
    \end{align} 
    \label{def:lackerZhang}
\end{definition}
 
\begin{corollary}[Lacker-Zhang fixed points are \FPES{}]
\label{cor:LackerZhang}
    Suppose $(U, \intPot)$ satisfies Assumption \ref{as:results:fixedPoint}. Let $\nu_0$ be a solution to the \FPE{} in the sense of Definition \ref{def:results:CBP}. Then, $(F, 0)$
    solves the Lacker-Zhang fixed point problem with
    \begin{equation}
        F(x) = \frac{1}{\kap}\Big(U(x) + \log \nu_0(x) \Big), \quad x \in \R^d.
        \label{eq:zeros:LZfromBP}
    \end{equation} Conversely, if $(F, C)$ solves the Lacker-Zhang fixed point problem, then there is a solution $\nu$ to \FPE{} in the sense of Definition \ref{def:results:CBP} characterized by its marginal $\nu_0$, which is given by
    \begin{equation}
        \nu_0(x) = \frac{1}{\mathcal{Z}_F} e^{-U(x) - \kap F(x)}, \quad  \mathcal{Z}_F := \int_{\R^d}  e^{-U(y) - \kap F(y)} dy, \quad x \in \R^d.
        \label{eq:zeros:BPfromLZ}
    \end{equation}    
\end{corollary}
\begin{proof}
Suppose $\nu_0$ is a \FPE{} in the sense of Definition \ref{def:results:CBP}.  By Proposition \ref{lem:zeros:entropyAndMoments}, $\nu_0$ is continuous and \eqref{eq:zeros:nablaMu0} holds for all $x \in \R^d$. Let $F$ be as defined in \eqref{eq:zeros:LZfromBP}. Taking logarithms of both sides of \eqref{eq:zeros:fixedPoint} and substituting \eqref{eq:zeros:LZfromBP} shows that $F$ solves \eqref{eq:zeros:LZ1} with $C = 0$. The second condition \eqref{eq:zeros:LZ2} holds by \eqref{eq:zeros:LZfromBP} and the fact that $\nu_0$ is a probability measure.

Now, suppose $(F, C)$ solves the fixed point equation of \cite{lacker2023stationary} and define $\mathcal{Z}_F := \int_{\R^d} e^{-U(x) - \kap F(x)} dx.$ Then, by substituting \eqref{eq:zeros:BPfromLZ} into \eqref{eq:zeros:LZ1}, we see that \eqref{eq:zeros:fixedPoint} is satisfied. Moreover  \eqref{eq:zeros:LZ2} and the form of $\mathcal{Z}_F$ in \eqref{eq:zeros:BPfromLZ} shows that $\nu_0$ is a probability measure. In light of the discussion in Remark \ref{rk:results:FPE}, $\nu_0$ extends to a solution of \FPE{}.
\end{proof}

We now prove Theorem \ref{thm:results:stationary}. One direction will be proved directly through Theorem \ref{thm:results:lyap}. The other exploits Theorem \ref{thm:results:zeros} after establishing that \FPES{} are in fact stationary distributions of \CMVE{}.
\begin{proof}[Proof of Theorem \ref{thm:results:stationary}]
First suppose that $\nu$ is a stationary distribution of the \CMVE{}. 
By assumption, there exists a linear growth solution $\mu^\nu$ to the \CMVE{} on $[0, \infty)$ with potentials $(U, \intPot)$ and initial condition $\nu$. Since Assumption \ref{as:results:fixedPoint} implies that $(U, \intPot)$ satisfy Assumption \ref{as:results:lyapunov}, by Theorem \ref{thm:results:lyap} we can fix $0 < s < t < \infty$ such that 
\begin{equation*}
    0 = \lyap_\kap(\mu^\nu_t) - \lyap_\kap(\mu^\nu_s) = \int_s^t \info_\kap(\mu^\nu_r) dr = (t - s) \info_\kap(\nu),
\end{equation*}
where we used $\mu^\nu_t = \mu^\nu_s = \nu$ by stationarity. Since $t - s > 0$, we have $\info_\kap(\nu) = 0$.

Next, suppose $\info_\kap(\nu) = 0$. By Theorem \ref{thm:results:zeros}, $\nu$ is absolutely continuous and its density satisfies the \FPE{} \eqref{eq:zeros:1MRF}-\eqref{eq:zeros:muBar}. Since $(U, \intPot)$ satisfies Assumption \ref{as:results:fixedPoint}, by Proposition \ref{lem:zeros:entropyAndMoments} we have that $\nu \in \mPSb$ and is continuous and positive everywhere. Hence the conditional distribution $\bar{\nu}(\cdot|\cdot)$ from Definition \ref{def:results:CBP} is defined everywhere. Then the function $\tilde{\zeta}^\nu:\R^d \rightarrow \R^d$ given by
\begin{equation}
    \label{eq:zeros:hatZeta} 
    \tilde{\zeta}^\nu(x) := \int_{\R^d} \nabla \intPot(x - y) \bar{\nu}(y|x) dy = \bbE^{\bar{\nu}}\big[\nabla \intPot(Y_0 - Y_1)\big|Y_0 = x \big], \quad x \in \R^d ,
\end{equation}
is well-defined. Since $\nu \in \mPS$ is a 1-MRF, by \eqref{eq:not:symmetry1_} for all $x, y \in \R^d$ and $v \in \{1, \ldots, \kap\}$ we have
\begin{equation*}
    \tilde{\zeta}^\nu(x) = \bbE^{\bar{\nu}}\big[\nabla \intPot(Y_0 - Y_1)\big|Y_0 = x \big] = \bbE^{{\nu}}\big[\nabla \intPot(Y_0 - Y_v)\big|Y_0 = x, Y_1 = y \big].
\end{equation*}
Note that Assumption \ref{as:results:fixedPoint} implies that $\tilde{\zeta}^\nu$ satisfies a linear growth condition. Define the measurable function $\tilde{\gamma}^\nu:\R^d \times \R^d \rightarrow \R^d$ by
\begin{equation}
\tilde{\gamma}^\nu(x, y) := \nabla U(x) + \nabla \intPot(x - y) + (\kap - 1) \tilde{\zeta}^\nu(x).    
\label{eq:zeros:hatGamma1}
\end{equation}
Combining the last two displays yields
\begin{equation}
    \tilde{\gamma}^\nu(x, y) = \bbE^{\nu} \bigg[ \nabla U(Y_0) + \sum_{v = 1}^\kap \nabla \intPot(Y_0 - Y_v)\bigg|Y_0 = x, Y_1 = y\bigg].
    \label{eq:zeros:hatGamma2}
\end{equation}
By Assumption \ref{as:results:lyapunov}, $\tilde{\gamma}^\nu$ inherits a linear growth condition from $\tilde{\zeta}^\nu$. Since $\nabla_x \log \bar{\nu}(x | y) = \nabla_{x} \log \bar{\nu}(x, y)$, \eqref{eq:zeros:muBar}, \eqref{eq:zeros:nablaMu0}, \eqref{eq:zeros:hatZeta} and \eqref{eq:zeros:hatGamma1} together imply that
\begin{equation}
\begin{aligned}    \nabla_x \log \bar{\nu}(x | y) &= \frac{\kap - 1}{\kap} \nabla_x \log \nu_0(x) - \frac{1}{\kap} \nabla U(x) - \nabla \intPot(x - y) \\
&= - \nabla U(x) - \nabla \intPot(x-y) - (\kappa - 1) \bbE^\nu\big[ \nabla \intPot(Y_0 - Y_1) | Y_0 = x\big] \\
&= -\tilde{\gamma}^\nu (x, y).
\end{aligned}
\label{eq:zeros:gradLogConditional}
\end{equation}
Similarly, define the measurable function $\tilde{\eta}^\nu: (\R^d)^{1 + \kap} \rightarrow (\R^d)^{1 + \kap} $ by
    \begin{equation*}
        \big(\tilde{\eta}^\nu(\x)\big)_v :=  \left\{
        \begin{aligned}
            &\nabla U(x_0) + \sum_{v = 1}^\kap \nabla \intPot(x_0 - x_v), &\quad &v = 0,
            \\
            &\tilde{\gamma}^{\nu}(x_v, x_0), &\quad &v \in \{1, \ldots, \kap\}.
        \end{aligned}
        \right.
    \end{equation*}

Then by \eqref{eq:zeros:1MRF},\eqref{eq:zeros:info1}, \eqref{eq:zeros:gradLogConditional}, and the above display, we have $\nabla_{\x} \log \nu(\x) + \tilde{\eta}^\nu( \x) = 0$, and therefore $\nu$ solves the following (stationary) Fokker-Planck equation:
\begin{equation*}
    \Delta_{\x} \nu(\x) + \nabla_{\x} \cdot\big( \tilde{\eta}^\nu(\x) \nu(\x)\big) = 0.
\end{equation*}
Since $\tilde{\gamma}^\nu$ and hence $\tilde{\eta}^\nu$ satisfies a linear growth condition, Proposition \ref{prop:ap1:superposition} implies that the trajectory $\{\mu_t\}_{t \geq 0}$ with $\mu_t = \nu$ for all $t \geq 0$ is the law of the weak solution of the SDE:
\begin{equation}
    d\tilde{\X}(t) = - \tilde{\eta}^\nu\big(\tilde{\X}(t)\big) dt + \sqrt{2} d\mathbf{B}_t,
    \label{eq:zeros:hatSDE}
\end{equation}
with initial condition $\tilde{\X}(0) \sim \nu$. Since $(U, W)$ satisfies Assumption \ref{as:results:lyapunov} and $\tilde{\gamma}^\nu$ satisfies a linear growth condition, by Proposition \ref{lem:zeros:entropyAndMoments} and Fubini's theorem we conclude the existence of $C \in (0 , \infty)$ such that \begin{equation*}
\begin{aligned}
    \mathbb{E} \bigg[\int_0^T \Big(|b(\tilde{\X}_t|^2 + |\gamma^\nu(\tilde{X}_0(t), \tilde{X}_1(t))|^2 + |\gamma^\nu(\tilde{X}_1(t), \tilde{X}_0(t))|^2 \Big) dt \bigg] & \leq C \int_0^T \bbE\big[|\tilde{\X}(t)|^2\big] \\
    &= C(1 + \kap) T \int_{\R^d}|x|^2 \nu_0(dx) \\ & < \infty,
\end{aligned}
\end{equation*}
for all $T \in (0, \infty)$. Therefore $\tilde{\X}$ satisfies \eqref{eq:cmve:integrability}, and thus by \eqref{eq:zeros:hatSDE} and \eqref{eq:zeros:hatGamma2}, $(\mu, \tilde{\gamma}^\nu, \tilde{\X})$ is a solution to the \CMVE{} on $[0, \infty)$ with potentials $(U, \intPot)$ and initial condition $\nu$, where $\tilde{\gamma}^\nu$ is extended trivially to a function on $[0, \infty) \times \R^d \times \R^d$ that is constant in the first argument. Moreover, since $\tilde{\gamma}^\nu$ satisfies a linear growth condition $(\mu, \tilde{\gamma}^\nu, \tilde{\X})$ is a linear growth solution and thus $\nu$ is a stationary distribution of the \CMVE{} with potentials $(U, \intPot)$. \end{proof}

\subsection{Long-time limits of solutions to the \CMVE{}} \label{sec:convergenceProofs} In this section we prove Theorem \ref{thm:results:convergence}. We first collect the following useful lemma, which allows one to control the Fisher information of a marginal distribution.
\begin{lemma}
    \label{lem:conv:marginalInformation}
    Suppose $\nu \in W^{1, 1}((\R^d)^{1 + \kap})$. Then, we have
    \begin{equation*}
        \int_{\R^d \times \R^d} |\nabla \log \bar{\nu}(x_0, x_1)|^2 \bar{\nu}(dx_0, dx_1) \leq \int_{(\R^d)^{1 + \kap}} |\nabla \log \nu(\x)|^2 \nu(d\x).
    \end{equation*}
\end{lemma}
\begin{proof}
    Since $\nu \in W^{1, 1}((\R^d)^{1 + \kap})$, we have for $v \in \{0, 1\}$ and Lebesgue-almost every $\x \in (\R^d)^{1 + \kap}$ that
    \begin{equation*}
        \nabla_{x_v} \int_{(\R^d)^{\kap - 1}} \nu(\x) \prod_{u = 2}^\kap dx_u = \nabla_{x_v} \bar{\nu}(x_0, x_1).
    \end{equation*}
    By the Cauchy-Schwarz inequality, we have
    \begin{equation*}
        |\nabla_{x_v} \bar{\nu}(x_0, x_1)|^2 \leq \bar{\nu}(x_0, x_1) \int_{(\R^d)^{\kap - 1}} \frac{|\nabla \nu(\x)|^2}{\nu(\x)} \prod_{u = 2}^\kap dx_u.
    \end{equation*}
    Dividing both sides by $\bar{\nu}(x_0, x_1)$ and integrating over $\R^d \times \R^d$ finishes the proof.
\end{proof}

We now prove Theorem \ref{thm:results:convergence}. Note that Theorem \ref{thm:results:lyap} and Proposition \ref{prop:results:energyBound} establish that $\bbH_\kap$ is a strong global Lyapunov function. Since, in addition, Proposition \ref{lem:conv:compactLevelSets} shows that $\bbH_\kap$ has compact level sets, the proof of Theorem \ref{thm:results:convergence} uses an argument similar to those used in proofs of LaSalle's invariance principle in metric spaces (e.g., see Chapter 9 of \cite{cazenave1998semilinear}).

\begin{proof}[Proof of Theorem \ref{thm:results:convergence}]
    Let $(\mu, \gamma)$ be the solution to the \CMVE{} with initial condition $\lambda$. The proof proceeds by showing that the set of possible limit points of $\mu$ is recurrent under the \CMVE{} flow and then establishing that $\mS_\kap$ contains this set.
    
    We first identify the set of possible limit points of $\{\mu_t\}_{t \geq 0}$. Set $M := \lyap_{\kap}(\lambda)$ and let $\mathscr{T} \subset [0, \infty)$ be the set of $t \geq 0$ such that $\mu_t \in W^{1, 1}((\R^d)^{1 + \kap})$ and \eqref{eq:results:energyIdentity} holds. Since $\mu$ is a linear growth solution to the \CMVE{}, Proposition \ref{thm:ap1:wellPosed}(3) implies that $\mathscr{T}$ is a set of full measure. Define $\omega(M)$ to be the set of all possible limit points of the \CMVE{} with potentials $(U, \intPot)$ and initial sparse free energy no greater than $M$:
 \begin{equation}
        \omega(M) := \Bigg\{ \nu \in \mPS : 
        \begin{aligned}\exists \xi \in& \mPSb, \text{ a solution $\mu^\xi$ to the \CMVE{} with }\mu^\xi_0 = \xi, \text{ and } \{t_n\}_{n = 1}^\infty \subset \mathscr{T}, \\ &\text{ with } t_n \rightarrow \infty \text{ s.t. } \lyap_{\kap}(\xi) \leq M \text{ and } \lim_{n\rightarrow \infty} \mu_{t_n}^{\xi} = \nu \text{ in } \mP((\R^d)^{1 + \kap}). \end{aligned}\Bigg\},
        \label{eq:zeros:omegaLimitSet}
    \end{equation}
    Note that by the assumed existence of a solution $(\mu, \gamma)$ to the \CMVE{} on $[0, \infty)$ with potentials $(U, \intPot)$, Theorem \ref{thm:results:lyap}, and the compactness of the level sets of $\bbH_\kap$ detailed in Proposition \ref{lem:conv:compactLevelSets}, we have that $\omega(M)$ is not empty.

We now show that $\mathbb{H}_\kap(\nu) \leq M$ for all $\nu \in \omega(M)$.    
By Proposition \ref{lem:conv:compactLevelSets}, $\lyap_\kap$ is lower semicontinuous. Then \eqref{eq:zeros:omegaLimitSet} and Theorem \ref{thm:results:lyap} then imply that for any $\nu \in \omega(M)$, there exists $\xi \in \mPSb$ with $\lyap_\kap(\xi)\leq M$ and a linear growth solution $\mu^\xi$ to the \CMVE{} on $[0, \infty)$ with potentials $(U, \intPot)$ and initial condition $\xi$ such that for some sequence $\{t_n\}_{n \geq 1} \subset \mathscr{T}$ we have
    \begin{equation}
        \lyap_\kap(\nu) \leq \liminf_{n \rightarrow \infty} \lyap_{\kap}(\mu_{t_n}^\xi) \leq \lyap_\kap(\xi) \leq M.
    \end{equation} 
    
 Next, letting $d_{LP}$ denote the Levy-Prokhorov metric, we prove 
 \begin{equation}
 \label{eq:zeros:convergence1}
     \lim_{t \rightarrow \infty} d_{LP}(\mu_t, \omega(M)) = 0.
 \end{equation}
 Suppose towards contradiction that $\liminf_{t \rightarrow \infty} d_{LP}(\mu_t, \omega(M))  \geq \eps$ for some $\eps > 0$. Then there exists a sequence $\{t_n\}_{n \geq 1}\subset \mathscr{T}$ such that $t_n \rightarrow \infty$ as $n \rightarrow \infty$ and
    \begin{equation*}
        \liminf_{n \rightarrow \infty} d_{LP}(\mu_{t_n}, \omega(M)) \geq \eps.
    \end{equation*}Recall the definition of $\mathcal{R}(M)$ in \eqref{eq:conv:levelsets}. 
    Since $M = \bbH_\kap(\lambda)$, by Theorem \ref{thm:results:lyap} $\bbH_\kap$ decreases along the trajectory $\{\mu_t\}_{t\geq 0}$ and we have $\{\mu_{t_n}\}_{n \geq 1} \subset \mathcal{R}(M)$. Since $\mathcal{R}(M)$ is compact by Proposition \ref{lem:conv:compactLevelSets}, $\{\mu_{t_n}\}_{n \geq 1}$ has a convergent subsequence in $\mathcal{R}(M)$. Then by passing to a further subsequence (which we also denote as $\{t_n\}_{n \geq 1}$), there exists $\mu_\infty \in \omega(M)$ such that $d(\mu_{t_n}, \mu_\infty) \rightarrow 0$. But this is a contradiction, and hence \eqref{eq:zeros:convergence1} holds.

Now, fix $\nu \in \omega(M)$ such that, there exists $\{t_n\}_{n \geq 1} \subset \mathscr{T}$ such that $\mu_{t_n} \rightarrow \nu$. In light of \eqref{eq:zeros:convergence1}, to prove \eqref{eq:results:convergenceToEkap} it suffices to show $\nu \in \mathcal{S}_\kap$. To this end, note that since $\mathscr{T}$ has full Lebesgue measure, we can choose $\{t_n\}_{t \geq 1}$ such that $t_n - t_{n-1} \geq 1$. Since $\mathbb{H}(\mu_{t_n})$ is a bounded monotone sequence and hence, Cauchy. Furthermore by Theorem \ref{thm:results:lyap} for any $\eps > 0$ there exists $N_\eps < \infty$ such that for all $n \geq N_\eps$, we have
    \begin{equation*}
        \begin{aligned}
            \eps \geq \lyap_{\kap}(\mu_{t_{n-1}}) - \lyap_{\kap}(\mu_{t_n}) &= \int_{t_{n-1}}^{t_n} \info_{\kap}(\mu_r) dr \geq \inf_{t_{n-1} < r < t_{n}} \info_{\kap}(\mu_r) \geq \inf_{t_{n-1} < r} \info_{\kap}(\mu_r).
        \end{aligned}
    \end{equation*}
    Therefore 
       $\liminf_{n \rightarrow \infty} \mathbb{I}( \mu_{t_n}) \leq \eps$ for all $\eps > 0$ 
    . Sending $\eps \downarrow 0$, we conclude
    \begin{equation}
        \liminf_{n \rightarrow \infty} \mathbb{I}(\mu_{t_n}) = 0. 
        \label{eq:zeros:liminfInfo}
    \end{equation}
    Thus to show $\nu \in \mS_\kap$, it suffices to show that $\mathbb{I}_{\kap}(\nu) \leq \liminf_{n \rightarrow \infty} \mathbb{I}_{\kap}(\mu_{t_n})$. By \eqref{eq:zeros:liminfInfo}, without loss of generality we can pass to another subsequence, which we denote again by $\{t_n\}_{n \geq 1}$, to obtain $\sup_{n \in \N}\mathbb{I}(\mu_{t_n}) \leq \eps$ for some $\eps > 0$. Furthermore, Theorem \ref{thm:results:lyap} implies that $\lyap_\kap(\mu_{t_n}) \leq M$ for all $n \geq 1$.  Then, using $\{\mu_{t_n}\}_{n \geq 1} \subset \mathcal{R}(M)$,  by \eqref{eq:results:info}, Assumption \ref{as:results:lyapunov}, and \eqref{eq:conv:secondMoment}, it follows that there exist constants $C_1, C_2 \in (0, \infty)$ such that
    \begin{equation}
        \int_{(\R^d)^{1 + \kap}} |\nabla_{x_0} \log \mu_{t_n}(\x)|^2 \mu_{t_n}(d\x) \leq 2\info_{\kap}(\mu_{t_n}) + C_1\int_{(\R^d)^{1 + \kap}}(1 + |\x|^2) \mu_{t_n}(d\x) \leq 2\eps + C_2M.
        \label{eq:zeros:fisherBound1}
    \end{equation}
    Since $t_n \in \mathscr{T}$, we have $\mu_{t_n} \in W^{1, 1}((\R^d)^{1 + \kap})$. Then we can apply \eqref{eq:not:symmetry2_} and Lemma \ref{lem:conv:marginalInformation} to obtain
    \begin{equation*}
    \begin{aligned}
        \int_{\R^d \times \R^d} |\nabla_{x_1} \log \bar{\mu}_{t_n}(x_0, x_1)|^2 \bar{\mu}_{t_n}(dx_0, dx_1) &= \int_{\R^d \times \R^d} |\nabla_{x_0} \log \bar{\mu}_{t_n}(x_0, x_1)|^2 \bar{\mu}_{t_n}(dx_0, dx_1)\\ &\leq \int_{(\R^d)^{1 + \kap}} |\nabla_{x_0} \log \mu_{t_n}(\x)|^2 \mu_{t_n}(\x).
    \end{aligned}
    \end{equation*}
    When combined with \eqref{eq:zeros:fisherBound1} and the definition of $\info_\kap$ in \eqref{eq:results:info}, this implies the existence of $\tilde{C} \in (0, \infty)$ such that
    \begin{equation}
\int_{(\R^d)^{1 + \kap}} |\nabla_{x_1} \log \mu_{t_n}(\x)|^2 \mu_{t_n}(\x) < \frac{2}{\kap}\info_{\kap}(\mu_{t_n}) + \int_{\R^d \times \R^d} |\nabla_{x_1} \log \bar{\mu}_{t_n}(x_0, x_1)|^2 \mu_{t_n}(dx_0, dx_1) \leq \tilde{C}.
\label{eq:zeros:fisherBound2}
\end{equation}
Together, \eqref{eq:zeros:fisherBound1} and \eqref{eq:zeros:fisherBound2} imply
\begin{equation*}
    \sup_{n \in \N} \bigg[\int_{(\R^d)^{1 + \kap}}\frac{|\nabla_{\x} \mu_{t_n}(\x)|^2}{\mu_{t_n}(\x)} d\x +    \int_{\R^d \times \R^d}\frac{|\nabla_{\x} \bar{\mu}_{t_n}(x_0, x_1)|^2}{\bar{\mu}_{t_n}(x_0, x_1)} dx_0 dx_1\bigg] < \infty.
\end{equation*}
Then following the arguments of Theorem 14.2 of \cite{bobkov2022fisher}, we can extract a subsequence (which again we denote by $\{t_n\}_{n \geq 1}$) such that 
\begin{equation}
    \sqrt{\mu_{t_n}(\x)}\nabla_{\x}\log \mu_{t_n}(\x)  {\rightarrow}  \sqrt{\nu(\x)}\nabla_{\x}\log \nu(\x)  \text{ weakly in }L^2((\R^d)^{1+ \kap}). 
    \label{eq:conv:conv1}
\end{equation}
and
\begin{equation}
        \sqrt{\mu_{t_n}(\x)} \nabla_{x_1} \log \bar{\mu}_{t_n}(x_0, x_1)  \rightarrow \sqrt{\nu(\x)}\nabla_{x_1} \log\bar{\nu}(x_0, x_1) \text{ weakly in }L^2((\R^d)^{1 + \kap}). 
        \label{eq:conv:conv2}
\end{equation}
Once again by Theorem \ref{thm:results:lyap}, we have $\lyap_\kap(\mu_{t_n}) \leq M$ for all $n \geq 1$. Therefore Assumption \ref{as:results:lyapunov} and \eqref{eq:conv:secondMoment} ensure that there exists $C \in (0, \infty)$ such that
\begin{equation*}
    \sup_{n \in \N} \int_{(\R^d)^{1 + \kap}} |b(\x)|^2 \mu_{t_n}(d\x) \leq C \sup_{n \in \N} \int_{(\R^d)^{1 + \kap}} \big(1 + |\x|^2\big) \mu_{t_n}(d\x) <  \infty.
\end{equation*}
By Theorem 3.18 of \cite{brezis2011Functional}, we have that  $b(\x)\sqrt{\mu_{t_n}(\x)} \rightarrow b(\x)\sqrt{\nu(\x)}$ weakly in $L^2((\R^d)^{1 + \kap})$. It is worth emphasizing that the weak convergence here (and in the remainder of the paragraph) is in $L^2$ and not in the probabilistic sense. Together with \eqref{eq:conv:conv1} and \eqref{eq:conv:conv2} this convergence implies
\begin{equation*}
    \big(b(\x) + \nabla_{\x}\log \mu_{t_n}(\x)  \big)\sqrt{\mu_{t_n}(\x)} {\rightarrow} \big(b(\x) + \nabla_{\x}\log \nu(\x) \big)\sqrt{\nu(\x)}\text{ weakly in }L^2((\R^d)^{1 + \kap}),
\end{equation*}
and
\begin{equation*}
 \sqrt{\mu_{t_n}(\x)} \nabla_{x_1} \big( \log \mu_{t_n}(\x) - \log \bar{\mu}_{t_n}(x_0, x_1)\big)  \rightarrow \sqrt{\nu(\x)}\nabla_{x_1} \big( \log \nu(\x) - \log\bar{\nu}(x_0, x_1)\big) \text{ weakly in }L^2((\R^d)^{1 + \kap}). 
\end{equation*}
Recalling the definition of $\info_\kap$ in \eqref{eq:results:info} and the lower semi-continuity of the $L^2$-norm with respect to weak limits in $L^2((\R^d)^{1 + \kap})$, the last two displays together imply.
\begin{equation*}
    \mathbb{I}_\kap(\nu) \leq \liminf_{n \rightarrow \infty}\mathbb{I}_\kap(\mu_{t_n}).
\end{equation*}
We finish by invoking \eqref{eq:zeros:liminfInfo} to obtain $\info_{\kap}(\nu) =0.$
\end{proof}

\section{Renormalized Entropy and Exponential Convergence}
\label{sec:k2Stuff}
\subsection{Sparse free energy as the limit of renormalized entropies} 
\label{sec:entropyRenormalization}
We present in this section the proof of Theorem \ref{thm:results:2entropy}. Recall the definition of the truncated line graph $\bbT^n_2$ from Definition \ref{def:results:Tn2} and the definition of the lift map $\psi_\nu^n$ from Definition \ref{def:results:2measures}. Also, recall that $\x^{(n)} = (x_{-n}, \ldots, x_{n}) \in (\R^d)^{V_n}$ is a state vector on the entire truncated tree $\bbT_2^n$ and the vector $\x =(x_{-1}, x_0, x_1) \in (\R^d)^3$ represents a root neighborhood state. 

The proof of Theorem \ref{thm:results:2entropy} requires the following lemma which establishes that $\psi_\nu^n$ is a 2-MRF (see Definition \ref{def:MRF}) and facilitates the computation of expectations under $\psi_\nu^n$. 

\begin{lemma}[Symmetries of the lift map]
\label{lem:entropy:symmetry}
    Suppose $\nu \in \mathcal{M}_{2, d}$ is absolutely continuous with respect to Lebesgue measure. Let $\psi_\nu^n$ be defined as in \eqref{eq:entropy:2reconstruction}. Then $\psi_\nu^n$ is a 2-MRF on $(\R^d)^{V_n}$, and can be formulated equivalently in the following ways:
\begin{equation}
       \psi_\nu^n\big(\x^{(n)}\big) =
\begin{dcases}
    \nu(x_{-1}, x_\rt, x_{1}) \prod_{v =1} ^{n-1} \nu(x_{v+1} | x_v, x_{v-1})\nu(x_{-v-1}|x_{-v}, x_{-v+1}) , & \\
    \nu(x_{n-2}, x_{n-1}, x_{n}) \prod_{v = -n}^{n-3} \nu(x_v|x_{v+1}, x_{v+2}),  & \\
    \nu(x_{-n}, x_{-n+1}, x_{-n+2}) \prod_{v = -n+3}^{n} \nu(x_{v}|x_{v-1}, x_{v-2}). &
\end{dcases}
    \label{eq:entropy:altReconstruction}
\end{equation}
Moreover, we have
        \begin{equation}
        \int_{(\R^d)^{ \bbT_2^n}}  f(x_{v-1}, x_v, x_{v+1}) \psi_\nu^{n}(\x^{(n)})d\x^{(n)} = \int_{(\R^d)^3} f(\x) \nu(\x) d\x. 
        \label{eq:entropy:expectations}
    \end{equation}
    for all $f \in L^1(\nu)$ and $v \in \{-n+1, \ldots, n-1\}$.  
\end{lemma}

\begin{proof}[Proof of Lemma \ref{lem:entropy:symmetry}]
First we prove that $\psi^n_\nu$ is a 2-MRF. We define the 2-cliques of $\bbT_n^2$ to be the induced subgraphs of $\bbT_n^2$ of diameter at most 2. Notice that the 2-cliques of $\bbT_n^2$ are precisely the edge $\{-n, -n+1\}$ and the induced subgraphs on both $\{v, v+1\}$ and $\{v-1, v, v+1\}$ for $v \in \{-n+1, \ldots, n-1\}$. By \eqref{eq:entropy:2reconstruction}, $\psi^n_\nu$ can be expressed as a product of functions on 2-cliques and thus, by a version of the Clifford-Hammersley theorem (e.g. Theorem 3.9 of \cite{lauritzen1996graphical} or Proposition 3.2 of \cite{lacker2021MRF}), $\psi_\nu^n$ is a 2-MRF.
The assertion \eqref{eq:entropy:altReconstruction} follows \eqref{eq:entropy:2reconstruction} and Bayes' formula.

Next we show \eqref{eq:entropy:expectations} by induction. Notice that \eqref{eq:entropy:expectations} is trivial for $n = 1$. Assume towards induction that \eqref{eq:entropy:expectations} is true for some $n \in \N$. It then suffices to check that it also holds when $n$ replaced with $n+1$. By the first expression in \eqref{eq:entropy:altReconstruction}, we have
\begin{equation*}
    \begin{aligned}
    \psi^{n+1}_\nu(\x^{(n)}) &= \psi_\nu^n(\x^{(n)})\nu(x_{n+1}|x_n, x_{n-1}) \nu(x_{-n-1}|x_{-n}, x_{-n+1}), \quad \x^{(n+1)} \in (\R^d)^{V_{n+1}}.
    \end{aligned}
\end{equation*} Let $\tilde{V}_{n, v} := V_n / N_v(\bbT^{n}_2)$. Then for $v \in \{-n+1, \ldots, n-1\}$ and $f \in L^1(d\nu)$, we can use the above display, integrate out $x_{n+1}$ and $x_{-n-1}$, and apply the inductive hypothesis to obtain for $\nu$-a.e. $(x_{v-1}, x_v, x_{v+1}) \in (\R^d)^3$ the following:
\begin{equation*}
    \begin{aligned}
        \int_{(\R^d)^{2n}} f(x_{v-1}, x_v, x_{v+1})\psi_\nu^{n+1}\big(\x^{(n+1)}\big) \prod_{u \in \tilde{V}_{n+1, v}} dx_u
        &= \int_{(\R^d)^{2n-2}} f(x_{v-1}, x_v, x_{v+1})\psi_\nu^{n, 2}\big(\x^{(n)}\big) \prod_{u \in \tilde{V}_{n, v}} dx_u \\ 
        &= f(x_{v-1}, x_v, x_{v+1}) \nu(x_{v-1}, x_v, x_{v+1}).
    \end{aligned}
\end{equation*}
Then for $v \in \{-n+1, \ldots, n-1\}$, \eqref{eq:entropy:expectations} follows on integrating both sides of the above display over $(\R^d)^3$. It only remains to show that \eqref{eq:entropy:expectations} holds for $v \in \{-n, n\}$. Since the argument for $v = \pm n$ are exactly the same, we present the proof for only $v = n$. By integrating out $x_{-n-1}$ we have
\begin{equation*}
    \begin{aligned}
    &\int_{(\R^d)^{2n}} \prod_{u = -n-1}^{n-2} \nu(x_u|x_{u+1}, x_{u+2}) \prod_{u = -n-1}^{n-2} dx_u = \int_{(\R^d)^{2n-1}} \prod_{u = -n}^{n-2} \nu(x_u|x_{u+1}, x_{u+2}) \prod_{u = -n}^{n-2} dx_u. 
    \end{aligned}
\end{equation*}
Successively integrating out the conditional distributions as in the above display, we arrive at
\begin{equation*}
    \begin{aligned}
    &\int_{(\R^d)^{2n}} \prod_{u = -n-1}^{n-2} \nu(x_u|x_{u+1}, x_{u+2}) \prod_{u = -n-1}^{n-2} dx_u 
    = 1.
    \end{aligned}
\end{equation*}
Combining the above display with the second expression in \eqref{eq:entropy:altReconstruction},  for $f \in L^1(d\nu)$ and $\nu$-a.e. $(x_{v-1}, x_v, x_{v+1}) \in (\R^d)^3$, and applying Fubini's theorem, we obtain
\begin{equation*}
    \int_{(\R^d)^{2n}} f(x_{n-1}, x_n, x_{n+1}) \psi_{\nu}^{n+1}(\x^{(n+1)}) \prod_{u = -n-1}^{n-2} dx_u = f(x_{n-1}, x_n, x_{n+1})\nu(x_{n-1}, x_n, x_{n+1}).
\end{equation*}
Integrating both sides over $(\R^d)^3$ finishes the proof of \eqref{eq:entropy:expectations}.\end{proof}

We now turn to the proof of Theorem \ref{thm:results:2entropy}. 
Recall the definition of the family of Gibbs measures $\{\theta^n\}_{n \in \N}$ and partition functions $\{\mathcal{Z}^n\}_{n \in \N}$ from Definition \ref{def:results:gibbs}.

\begin{proof}[Proof of Theorem \ref{thm:results:2entropy}]
We first show the convergence of $\frac{1}{2n+1}\log \mathcal{Z}^n$ in \eqref{eq:results:Hstar} by a subadditivity argument. Let $q$ be as in Assumption \ref{as:results:invMeasure}. By the lower bound on $q$, there exists $C_q \in \R$ such that $C_q > - \inf_{x \in \R^d}q(x)$. Then by \eqref{eq:results:fBound} and \eqref{eq:results:gibbs}, for every $k, m \in \N$, we have
\begin{equation*}
    \begin{aligned}
        \mathcal{Z}^{k+m} &= \int_{(\R^d)^{V_{k+m}}} \exp\bigg( -\frac{1}{2} \newW(x_k, y_{-m}) - \frac{1}{2}\sum_{(u, v) \in E^k}\newW(x_u, x_{v}) - \frac{1}{2}\sum_{(u, v) \in E^m}\newW(y_u, y_v))\bigg) d\x^{(k)} d\mathbf{y}^{(m)}\\
        & \leq \mathcal{Z}^k  \mathcal{Z}^m\int_{(\R^d)^{V_{k+m}}} \exp\bigg( -\frac{q(x_k) + q(y_{-m})}{2} \bigg) \theta^k(d\x^{(k)}) \theta^m(d\mathbf{y}^{(m)})\\
        &\leq e^{C_q} \mathcal{Z}^k \mathcal{Z}^m.
    \end{aligned}
\end{equation*}
Taking logarithms of both sides of the last display, we have the following near sub-additivity property:
\begin{equation*}
    \log \mathcal{Z}^{k+m} \leq C_q + \log \mathcal{Z}^k + \log \mathcal{Z}^m.
\end{equation*}
Applying Theorem 2 of \cite{hammersley1962subadditive} to $f(n) = \log \mathcal{Z}^n$, we conclude that $\frac{1}{2n+1}\log \mathcal{Z}^n$ converges to some constant $\bbH_2^* \in \R$ as $n \rightarrow \infty$.

We conclude by showing the identity \eqref{eq:results:2convergence}.   Since $\nu \in \mathcal{Q}_{2, d}$, by the definition of $Q$ in \eqref{eq:results:W} and Assumption \ref{as:results:lyapunov} we know $ \log \nu \in L^1(d\nu)$ and $Q \in L^1(d\nu)$, and thus by the definition of $g$ in \eqref{eq:results:potential} we have
\begin{equation}
    \int_{(\R^d)^3} g(\x) \nu(d\x) = \int_{(\R^d)^2} \frac{1}{2} \newW(x, y) \bar{\nu}(dx, dy)
    \label{eq:entropy:gAndW}
\end{equation}
Moreover, the finite entropy condition on $\nu$ guarantees the absolute continuity of $\nu$. To conclude, observe that \eqref{eq:results:gibbs}, \eqref{eq:entropy:2reconstruction}, and Lemma \ref{lem:entropy:symmetry} with $f = Q$ and $f = \log \nu$, yield the following:
\begin{equation*}
\begin{aligned}
        &\mathcal{H}(\psi^n_\nu | \theta^{n}) - \log \mathcal{Z}^n \\=  &\int_{(\R^d)^{V_n}} \bigg[\sum_{v = -n+1}^{n-1} \log \nu(x_{v-1}, x_v, x_{v+1}) - \sum_{v = -n+1}^{n-2} \log \bar{\nu}(x_v, x_{v+1}) + \sum_{(u, v) \in E^n} \frac{1}{2}\newW(x_u, x_v)\bigg] \psi^n_\nu(d\x^{(n)}) \\
        = & (2n-1) \int_{(\R^d)^3} \log \nu(\x) \nu(d\x) - (2n-2) \int_{(\R^d)^2} \log \bar{\nu}(x, y) \bar{\nu}(dx, dy) + 2n \int_{(\R^d)^3} g(\x) \nu(d\x) ,
\end{aligned}
\end{equation*}
where we used \eqref{eq:not:symmetry1_} and \eqref{eq:entropy:gAndW} in the last equality. By the definition of $\energy_2$ in \eqref{eq:results:entropy} and the last display, it follows that
\begin{equation*}
\begin{aligned}
\bbH_2(\nu) = &\lim_{n \rightarrow \infty}\bigg(\frac{2n-1}{2n+1} \int_{(\R^d)^3} \log \nu(\x) \nu(d\x) - \frac{2n-2}{2n+1} \int_{(\R^d)^2} \log \bar{\nu}(x, y) \bar{\nu}(dx, dy) + \frac{2n}{2n+1} \int_{(\R^d)^3} g(\x) \nu(d\x) \bigg).  \end{aligned}
\end{equation*}
Combining the last two displays with \eqref{eq:results:Hstar} yields \eqref{eq:results:2convergence}.
\end{proof}

\subsection{Log-Sobolev inequality}
\label{sec:LSI:LSI}
This section is dedicated to the proof of Theorem \ref{thm:results:LSI}. We recall the definition of a log-Sobolev inequality, (e.g., see Chapter 5 of \cite{bakry2014diffusion}).
\begin{definition}[log-Sobolev inequality]
\label{def:LSI}
    A measure $\theta \in \mP(\R^m)$ satisfies a log-Sobolev inequality (LSI) with constant $C <\infty$ if for every continuously differentiable function $f: \R^m \rightarrow \R$ such that $\mathbb{E}^\theta[f^2] = 1$, we have
\begin{align*}
    \bbE^\theta[f^2 \log f^2] \leq C \bbE^{\theta}[|\nabla f|^2].
\end{align*}
We use ${C}_{LS}(\theta)$ to denote the best possible constant $C$.
\end{definition}

Our technique relies on properties of the Gibbs measures $\theta^{n}$ defined in Definition \ref{def:results:gibbs} and its conditional distributions. Therefore we first prove the following characterization of the conditional distributions of $\theta^{n}$.
\begin{lemma}[conditional Gibbs measures]
    Suppose $(U, \intPot)$ satisfy Assumption \ref{as:results:invMeasure}. Fix $n \in \N$. Define the family of functions
\begin{equation}
    U^n_v(x):= \begin{cases}
    U(x) &\quad v  \in \{-n+1,\ldots, n-1\}, \\
    \frac{1}{2}U(x) &\quad v \in \{-n, n\}.
    \end{cases}, \quad x \in \R^d,
    \label{eq:LSI:uvn}
\end{equation}
and
\begin{equation}
    H_n\big(\x^{(n)}\big) := - \sum_{v = -n}^n \bigg(U^n_v(x_v) + \sum_{u \sim v} \frac{1}{2}\intPot(x_v - x_u) \bigg), \quad \x^{(n)} \in (\R^d)^{V_n}.
    \label{eq:LSI:hamiltonian}
\end{equation}
Let $\theta^n$ be the Gibbs measure of Definition \ref{def:results:gibbs}. We have the following alternative representation of $\theta^n$:
\begin{equation}
    \theta^{n}(d\x^{(n)}) = \frac{1}{\mathcal{Z}^n} e^{-H_n(\x^{(n)})} d\x^{(n)},
    \label{eq:LSI:gibbs}
\end{equation}
where $Z^{n}$ was defined in \eqref{eq:results:gibbsPFN}. Moreover $\theta^n$ is a 1-MRF, and for each $v \in \bbT_2^n$ and $x_{\partial v} \in (\R^d)^{N_v(\bbT_2^n)}$, the conditional distributions $\theta^n_v$ take the following form:
\begin{equation}
\begin{aligned}
    \theta_v^n\big(dx_v|x_{\partial v}\big) &= \frac{1}{\mathcal{Z}^n_v\big(x_{\partial v}\big)} \exp\bigg( - U^n_v(x_v) - \sum_{u \sim v} \intPot(x_v - x_u) \bigg) dx_v, \\
{\mathcal{Z}^n_v\big(x_{\partial v}\big)} &:= \int_{\R^d} \exp\bigg( - U^n_v(x) - \sum_{u \sim v} \intPot(x - x_u) \bigg) dx.
\end{aligned}
\label{eq:LSI:theta^n_v}
\end{equation}
Moreover for each $v$ and $x_{\partial v} \in (\R^d)^{N_v(\bbT_2^n)}$, $\theta^n_v(\cdot | x_{\partial v})$ has finite moments of all orders.
\label{lem:LSI:condGibbs}
\end{lemma}
\begin{proof}
The representation \eqref{eq:LSI:gibbs} follows immediately from the definition of $\theta^n$ and $W$ in Definition \ref{def:results:gibbs}. By the Clifford-Hammersley Theorem (e.g. Proposition 3.1 in \cite{lacker2021MRF}) and the form of $\theta^n$ in \eqref{eq:results:gibbs}, we have that $\theta^n$ is a 1-MRF in the sense of Definition \ref{def:MRF}. By Bayes rule and \eqref{eq:LSI:gibbs}, the conditional measures $\theta^n_v$ satisfy \eqref{eq:LSI:theta^n_v}. By Assumption \ref{as:results:invMeasure} and since $\theta^n$ is continuous, $\theta^n$ has finite moments and hence $\theta^n_v(\cdot|x_{\partial v})$ has finite moments for all $x_{\partial v} \in (\R^d)^{N_v(\bbT_2^n)}$. 
\end{proof}

We now present our main assumption for Theorem \ref{thm:results:LSI}. First, define
    \begin{equation}
    \label{eq:LSI:w}
        w(r, \varphi) := \sup_{\substack{x, y \in \R^d \\|x - y| = r}} - \bigg \langle \frac{x - y}{|x - y|}, \nabla \varphi(x) - \nabla \varphi(y) \bigg \rangle, \quad r \in (0, \infty), \varphi \in C^1(\R^d). 
    \end{equation}
Roughly speaking, $\omega(r, \varphi)$ controls the convexity of $\varphi$ at scale $r$, and similar functions have been used to study properties of long-time behavior of McKean-Vlasov equations \cite{eberle2019QuantHarrisThm, guillin2022lsi}. We also note a similarity with the \emph{integrated convexity profile} in (1.2.5) of \cite{conforti2023projected}.
\begin{assumption}
    \label{as:results:LSI}
    Assume $(U, \intPot)$ satisfies Assumption \ref{as:results:lyapunov} and Assumption \ref{as:results:invMeasure}. In addition, assume the following properties hold:
    \begin{enumerate}
        \item The condition
        \begin{equation}
            \label{eq:results:H3}
            \hat{c}_{\text{Lip}, i} := \frac{1}{4} \int_0^\infty \exp\bigg( \frac{1}{4}\int_0^s w_i(u) du \bigg) s\, ds < \infty, \quad i = 0, 1,
        \end{equation}
        is satisfied, where
        \begin{align*}
            w_0(r) :&= \sup_{ z \in \R^d} w\Big(r, \frac{1}{2}U + \intPot( \cdot - z) \Big), \\
            w_1(r) :&= \sup_{ z, z' \in \R^d} w\Big(r, U + \intPot( \cdot - z) + \intPot( \cdot - z') \Big).
        \end{align*}
        \item There exists $\hat{C}_{LS} < \infty$ such that the conditional measures $\theta^n_v$ defined in \eqref{eq:LSI:theta^n_v} satisfy a log-Sobolev inequality with constant $\hat{C}_{LS}$ uniformly in $x_{\partial v} \in (\R^d)^{N_v(\bbT_2^n)}$. That is, for all $n \in \N$, $v \in \bbT_n^2$, and $x_{\partial v} \in (\R^d)^{N_v(\bbT_2^n)}$, we have 
        \begin{equation*}
        \begin{aligned}
            \entropy\Big( \nu \Big| \theta^{n}_v(\cdot | x_{\partial v}) \Big) &\leq \hat{C}_{LS} \mI\Big( \nu \Big| \theta^{n}_v(\cdot | x_{\partial v}) \Big), \quad \nu \in \mP(\R^d).
        \end{aligned}
        \end{equation*}
        \item We have $\intPot \in C^2(\R^d)$ and
        \begin{equation}
            \delta^*_0 := \hat{c}_{\text{Lip}, 0} \| \nabla^2 \intPot\|_{L^\infty(\R^d)} < 1, \quad \delta^*_1 := 2\hat{c}_{\text{Lip}, 1} \| \nabla^2 \intPot\|_{L^\infty(\R^d)} < 1.
            \label{eq:results:H5}
        \end{equation}
    \end{enumerate}
\end{assumption}
\begin{remark}
\normalfont
\label{rk:LSI:convexity}
Assumption \ref{as:results:LSI} is similar to the conditions of Theorem 9 of \cite{guillin2022lsi}. Also note that while \ref{as:results:LSI}(2) is stated as a uniform-in-$n$ condition, by symmetry note that the conditional measures are the same for all $n \in \N$. In fact, it suffices to check that the conditional LSI holds for some fixed $n$, and only at a boundary point $v \in \{n, -n\}$ and any point in the bulk $v \in \{-n+1, \ldots, n-2\}$.

It is easy to check that if $(U, \intPot)$ are strongly convex, then conditions (1) and (2) of Assumption \ref{as:results:LSI} are satisfied; in particular, (1) is satisfied immediately by the definition of strong convexity, and (2) is satisfied by the classical Bakry-\'{E}mery condition, see (e.g. Corollary 5.7.2 of \cite{bakry2014diffusion}). A notable example is when $d = 1$ and $U(x) = (\alpha + \beta)|x|^2/2$ and $\intPot(x) = - \beta |x|^2/4$ with $\alpha > |\beta| > 0$. A simple calculation then shows
    \begin{equation*}
        \begin{cases} \omega_0(r) &= - \frac{\alpha r}{2} \\ \omega_1(r) &= - \alpha r \end{cases}, \quad \begin{cases}
            \hat{c}_{\text{Lip}, 0} = \frac{2}{\alpha} \\ \hat{c}_{\text{Lip}, 1} =  \frac{1}{\alpha}
        \end{cases}, \quad \|\nabla^2 \intPot \|_{L^\infty(\R^d)} = \frac{|\beta|}{2}.
    \end{equation*}
    Therefore, since $\alpha > |\beta| > 0$, \eqref{eq:results:H3} and \eqref{eq:results:H5} hold. Assumption 7.6(2) holds since the conditional measures $\theta^n_v$ are explicit Gaussian measures (see \cite{hu2024gaussian} for more details). It is shown in \cite{hu2024gaussian} that $\alpha > |\beta|$ is necessary for the convergence of the associated finite-dimensional interacting Ornstein-Uhlenbeck processes, thereby establishing that Assumption \ref{as:results:LSI} is tight in the case of quadratic potentials.

\end{remark}
Here we present three additional results that will be used in the proof of Theorem \ref{thm:results:LSI}. The first is a generator estimate for Fokker-Planck equations that clarifies \eqref{eq:results:H3} in Assumption \ref{as:results:LSI}. 
\begin{proposition}[Theorem 1.1 in \cite{wu2009gradient}]
    \label{thm:LSI:wu}
     Let $\varphi \in C^2(\R^d)$. Define the differential operator $\mathcal{L}_\varphi := \Delta - \nabla \varphi \cdot \nabla$. Let  $\mL_\varphi^{-1}$ denote the Poisson operator, which is the inverse of $\mL_\varphi$ on the Banach space of Lipschitz continuous functions $C_{\text{\normalfont Lip}, 0}(\R^d)$. That is, for $f \in C_{\text{\normalfont Lip}, 0}(\R^d)$, $h = \mL_\varphi^{-1} f$ if $\mL_{\varphi} h = f$. Recall the definition of $w(r, \varphi)$ in \eqref{eq:LSI:w}. If
    \begin{align*}
        \hat{c}(\varphi) := \frac{1}{4} \int_0^\infty s \exp\bigg(\int_0^s \frac{1}{4}w(r, \phi) dr \bigg) ds < \infty,
    \end{align*}
    then $\mL_\varphi^{-1}$ is a bounded operator on $C_{\text{\normalfont Lip}, 0}(\R^d)$ and we have $\|\mathcal{L}_\varphi^{-1} \|_{\text{\normalfont Lip}} < \hat{c}(\varphi).$
\end{proposition}
The following result is the main tool in the proof of the uniform log-Sobolev inequality \eqref{eq:LSI:preLimitLSI} in Theorem \ref{thm:results:LSI}. It provides a criterion for a uniform LSI to hold for a sequence of finite particle systems. 
\begin{theorem}[Theorem 0.1 in \cite{zergalinski1992dobrushin}]
    \label{thm:LSI:zergalinski}
     Fix $n \in \mathbb{N}$ and $\lambda$ be a 1-MRF on $\R^{\bbT_2^n}$. Recall the notation $\x^{(n)}$ defined in Definition \ref{def:results:Tn2}. For $v \in \bbT_2^n$ and $\x^{(n)} \in (\R^d)^{V_n}$, let $\hat{\x}_v := \{x_u\}_{u \neq v} \in (\R^d)^{2n}$. Denote the corresponding family of specifications (i.e. marginal conditional measures) of $\lambda$ by $ \lambda_v(x_v | \hat{\x}_v)$. 
     
     Suppose the following two properties hold.
    \begin{enumerate}
        \item Expressed in terms of the constant $C_{LS}$ from Definition \ref{def:LSI}, the specifications $\{\lambda_v\}_{v \in \bbT^n_2}$ satisfy the uniform log-Sobolev inequality,
        \begin{align*}
            \tilde{C}_{LS, n} = \sup_{v \in \bbT^n_2, \, \hat{\x}_v \in (\R^d)^{2n}} C_{LS}\Big(\lambda_v\big(\cdot \big| \hat{\x}_v\big)\Big) < \infty.
        \end{align*}
        \item There exist constants $c^n_{uv}(\lambda) \geq 0$ for every $u, v \in \bbT_2^n$, such that every smooth strictly positive function $f: \R^{\bbT_2^n} \rightarrow \R$ satisfies the following for all $\hat{\x}_v \in (\R^d)^{2n}$ :
        \begin{equation}
        \label{eq:LSI:zerg2}
            \bigg | \nabla_{x_u} \sqrt{ \bbE^{\lambda_v}[f^2|\hat{\x}_v]}\bigg| \leq \sqrt{\bbE^{\lambda_v}[|\nabla_{x_u}f|^2|\hat{\x}_v]} + c^n_{uv}(\lambda) \sqrt{\bbE^{\lambda_v}[|\nabla_{x_v}f|^2|\hat{\x}_v]},
        \end{equation}
        where there exists $\delta \in (0, 1)$ such that
        \begin{equation}
            \sup_{u \in \bbT_2^n} \max\bigg\{ \sum_{v \in \bbT_2^n} c^n_{uv}(\lambda), \sum_{v \in \bbT_2^n} c^n_{vu}(\lambda)\bigg\} \leq \delta < 1.
            \label{eq:LSI:zerg3}
        \end{equation}
    \end{enumerate} 
    Then $\lambda$ satisfies an LSI with constant
        \begin{align*}
            C_{LS}(\lambda) \leq \frac{\tilde{C}_{LS, n}}{(1 - \delta)^2}.
        \end{align*}
\end{theorem}

We now establish an additional result, which plays a crucial role in the proof of Theorem \ref{thm:results:LSI}(2). It can be thought of as the relative Fisher information analogue of Theorem \ref{thm:results:2entropy}.
\begin{proposition}[Convergence of renormalized Fisher information]
    \label{lem:LSI:lyapConv}
 Assume $(U, \intPot)$ satisfies Assumption \ref{as:results:invMeasure} and $\nu \in \mPS$ satisfies \eqref{eq:LSI:finiteFI}. Then we have
    \begin{equation}
        \label{eq:LSI:infoConv}
        \lim_{n \rightarrow \infty} \frac{1}{2n+1}\mI(\psi_\nu^n | \theta^n) = \info_2(\nu).
    \end{equation}
\end{proposition}

Proposition \ref{lem:LSI:lyapConv} is proved in Section \ref{sec:LSI:convFI}. We are now ready to prove Theorem \ref{thm:results:LSI}.

\begin{proof}[Proof of Theorem \ref{thm:results:LSI}] 
First, we establish (1) and the uniform log-Sobolev inequality \eqref{eq:LSI:preLimitLSI}. Note that $\theta^n$ is a 1-MRF by Lemma \ref{lem:LSI:condGibbs}. We therefore can appeal to Theorem \ref{thm:LSI:zergalinski}, which provides explicit upper bounds on the LSI constant for Gibbs measures. To this end, we first show that the second condition of Theorem \ref{thm:LSI:zergalinski} holds with the same $\delta$ for each $\theta^{n}$. Before proceeding, we introduce some convenient notation. For $n \in \N$, $v \in \bbT_2^n$, and $f: \R^d \rightarrow \R$, let \[\bbE^{\theta^n_v}[f|x_{\partial v}] := \bbE^{\theta^n_v}[f(Y_v)|Y_{\partial v} = x_{\partial v} ]. \]
Assumption \ref{as:results:LSI}(2) implies that the first condition of Theorem \ref{thm:LSI:zergalinski} holds. To verify the second condition of Theorem \ref{thm:LSI:zergalinski}, define
    \begin{equation}
        c^n_{uv}(\theta^n) := \begin{cases} \frac{1}{2} \delta_{1}^*1_{u \sim v} &\quad v \in \{-n+1, \ldots, n-1\} \\
    \delta_{0}^* 1_{u \sim v} &\quad v \in \{-n, n\}
        \end{cases}, \quad u \in \bbT_2^n.
        \label{eq:LSI:candidateC}
    \end{equation}    
    with $\delta_0^*, \delta_1^* \in (0, 1)$ in \eqref{eq:results:H5}. By Lemma 17 of \cite{guillin2022lsi} and the fact that $\theta^n$ is a 1-MRF,  \eqref{eq:LSI:zerg2}-\eqref{eq:LSI:zerg3} hold with $\delta = \max\{\delta_0^*, \delta_1^*\}$ if for all $n \in \N$, and $a \in C_b^1(\R^d)$,    
    \begin{equation}
        \nabla_{x_u} \bbE^{\theta_v^n}[a|x_{\partial v}] < c_{uv}^n(\theta^n) \bbE^{\theta_v^n}\big[ |\nabla a|\big|x_{\partial v}\big], \quad u, v \in \bbT_n^2, \: x_{\partial v} \in (\R^d)^{\partial v},
        \label{eq:LSI:suffCond}
    \end{equation}
    
   It is evident that for all $u, v \in \mathbb{T}_2^n$ such that $u \not \sim v$, $\nabla_{x_u} \bbE^{\theta_v^n}[a|x_{\partial v}] = 0,$ and therefore \eqref{eq:LSI:candidateC} holds with $c^n_{uv}(\theta^n) = 0$ for all $u \not \sim v$. It will suffice to calculate $\nabla_{x_{v+1}}\bbE^{\theta_v^n}[a|x_{\partial v}]$ for $v \in \{-n, \ldots, n-1\}$, as the other term $\nabla_{x_{v-1}}\bbE^{\theta_v^n}[a|x_{\partial v}]$ for $v \in \{-n+1, \ldots, n\}$ is computed in the exact same way. By \eqref{eq:LSI:theta^n_v} and the product rule, we have,
    \begin{equation}
    \begin{aligned}
    \label{eq:LSI:gradExpectation}
        \nabla_{x_{v+1}}\bbE^{\theta_v^n}[a|x_{\partial v}] = \int_{\R^d} a(x_v) \nabla &\intPot(x_v - x_{v+1}) \theta^n_v(dx_v|x_{\partial v})  \\&- \int_{\R^d} a(x_v) \theta^n_v(dx_v|x_{\partial v}) \int \nabla \intPot(x_v - x_{v+1})\theta_v^n(dx_v|x_{\partial v}).        
    \end{aligned}
    \end{equation}
    Note that by \eqref{eq:LSI:theta^n_v}, Assumption \ref{as:results:invMeasure}, and boundedness of $a$, we have $a(\cdot) \nabla_{x_{v+1}} \theta^n(\cdot | x_{\partial v}) \in L^1(\R^d)$ and thus the exchange of derivative and integral is justified. Define the function $\Psi_v:\R^d \times (\R^d)^{\partial v} \rightarrow \R^d$ by
    \begin{equation}
        \Psi_v(x_v, x_{\partial v}) := \nabla \intPot(x_v - x_{v+1}) - \bbE^{\theta_v^n}\big[\nabla \intPot (\cdot - x_{v+1})|x_{\partial v}\big],
        \label{eq:LSI:PSI}
    \end{equation}
    and the function $b_v^n: \R^d \times (\R^{d})^{\partial v} \rightarrow \R$ by
\begin{align*}
    b_v^n\big( x_v \big| x_{\partial v} \big) := U_v^n(x_v) + \sum_{u \sim v} \intPot(x_v - x_u),
\end{align*}
with $U_v^n$ as defined in \eqref{eq:LSI:uvn}. Define the conditional generator $\mL_v^n$ by
\begin{equation}
\label{eq:LSI:condGenerator}
    \mathcal{L}^n_v(x_{\bar{v}}) :=  \Delta_{x_v} - \nabla_{x_v} b_v^n\big( x_v \big| x_{\partial v} \big) \cdot \nabla_{x_v}.
\end{equation}
Notice that for any $v \in \bbT_n^2$, functions $f_1, f_2 \in C^1_b(\R^d)$, and $\hat{\x}_v \in (\R^d)^{2n}$, by \eqref{eq:LSI:theta^n_v} we have 
\begin{equation*}
\begin{aligned}
    \int_{\R^d} \big(\Delta f_1(x_v) \big)f_2(x_v) \theta_v^n(dx_v|x_{\partial v}) =     - \int_{\R^d} \nabla &f_1(x_v) \cdot \nabla f_2(x_v) \theta_v^n(dx_v|x_{\partial v}) \\
    &+ \int_{\R^d} \Big(\nabla_{x_v} b^n\big(x_v|x_{\partial v}\big) \cdot \nabla f_1(x_v) \Big)f_2(x_v) \theta_v^n(dx_v|x_{\partial v}).     
\end{aligned}
\end{equation*}
Thus $\theta^n_v$ is reversible with respect to $\mathcal{L}^n_v$ in the sense of Definition 1.6.1 and (1.6.2) of \cite{bakry2014diffusion}. Therefore by \eqref{eq:LSI:gradExpectation}, \eqref{eq:LSI:PSI}, \eqref{eq:LSI:condGenerator}, and integrating by parts, for all $z \in \R^d$ with $|z| = 1$ we have
    \begin{align*}
        \Big|z \cdot \nabla_{x_{v+1}} \bbE^{\theta_v^n}[a|x_{\partial v}]\Big| &= \Bigg|\int_{\R^d} a(x_v) \big[z \cdot \Psi_v(x_v, x_{v+1}) \big] \theta_v^n(dx_v|x_{\partial v}) \Bigg|\\
        & = \Bigg|\int_{\R^d} \Big( \mL^n_v(x_{\bv}) a(x_v) \Big) ( (\mL_v^n)^{-1}(x_{\bv}) \big) \Big[z \cdot \Psi_v(x_v, x_{v+1}) \Big]\theta_v^n(dx_v|x_{\partial v})\Bigg| \\
        & = \Bigg|\int_{\R^d} \Big( \nabla_{x_v} a(x_v) \Big) \cdot \big(\nabla_{x_v} (\mL_v^n)^{-1}(x_{\bv}) \big) \Big[z \cdot\Psi_v(x_v, x_{v+1})\Big] \theta_v^n(dx_v|x_{\partial v})\Bigg|. 
    \end{align*}
    The above display therefore shows that
    \begin{equation}
        \label{eq:LSI:equation1}\begin{aligned}
        \Big|z \cdot \nabla_{x_{v+1}} \bbE^{\theta_v^n}[a|x_{\partial v}]\Big| \leq \Big( \bbE^{\theta_v^n}\big[|\nabla a|\big|x_{\partial v}\big] \Big) \sup_{x_v \in \R^d} \Big | \big( \nabla_{x_v} (\mL_v^n)^{-1}(x_{\bv})\big) \Big[z \cdot \Psi_v(x_v, x_{v+1}) \Big]\Big|.            
        \end{aligned}
    \end{equation}
    By applying Proposition \ref{thm:LSI:wu} with $\varphi = b_v^n(\cdot|x_{\partial v})$ and item (1) of Assumption \ref{as:results:LSI}, for $v \in \{-n, n\}$ we have
        \begin{align*}
        \sup_{x_v \in \R^d} \Big | \big( \nabla_{x_v} (\mL_v^n)^{-1}) \Big[z \cdot \Psi_v(\cdot, x_{v+1}) \Big]\Big| 
        & \leq \sup_{x_v \in \R^d} \big\|(\mL_v^n)^{-1}(x_{\bv})\|_{\text{Lip}} \Big|\Psi_v(x_v, x_{v+1}) \cdot z \Big| \\
        & \leq \hat{c}_{\text{Lip}, 0} \sup_{x_v \in \R^d} \big | \Psi_v(x_v, x_{v+1}) \cdot z |,
    \end{align*}
    and similarly for $v \in \{-n+1, \ldots, n-1\}$, 
    \begin{align*}
        \sup_{x_v \in \R^d} \Big | \big( \nabla_{x_v} (\mL_v^n)^{-1}) \Big[z \cdot \Psi_v(\cdot, x_{v+1}) \Big]\Big| 
        & \leq \sup_{x_v \in \R^d} \big\|(\mL_v^n)^{-1}(x_{\bv})\|_{\text{Lip}} \Big|\Psi_v(x_v, x_{v+1}) \cdot z \Big| \\
        & \leq \hat{c}_{\text{Lip}, 1} \sup_{x_v \in \R^d} \big | \Psi_v(x_v, x_{v+1}) \cdot z |.
    \end{align*}
    Moreover \eqref{eq:LSI:PSI}, together with the mean-value theorem, implies that
    \begin{align*}
        \sup_{x_v \in \R^d} | \Psi_v(x_v, x_{v+1}) \cdot z | &\leq \sup_{x_v, x_{v+1} \in \R^d} \big | \nabla_{x_v}  \nabla \intPot(x_v - x_{v+1}) \big| \leq \|\nabla^2 \intPot\|_{L^\infty(\R^d)}.
    \end{align*}
    Combining the previous three displays with \eqref{eq:results:H5}, we have
    \begin{equation*}
        \sup_{x_v \in \R^d} \Big | \big( \nabla_{x_v} (\mL_v^n)^{-1}) \Big[z \cdot \Psi_v(\cdot, x_{v+1}) \Big]\Big| \leq \begin{cases}
            \hat{c}_{\text{Lip}, 0} \|\nabla^2\intPot\|_{L^\infty(\R^d)} = \delta_0^*, &\quad v \in \{-n, n\}, \\
            \hat{c}_{\text{Lip}, 1} \|\nabla^2\intPot\|_{L^\infty(\R^d)} = \tfrac{1}{2}\delta_1^*, &\quad v \in \{-n+1, \ldots, n-1\}.
        \end{cases}.
        \label{eq:LSI:equation2}
    \end{equation*}
Therefore the above display together with \eqref{eq:LSI:equation1} imply that \eqref{eq:LSI:suffCond} holds with $c^n_{uv}(\theta^n)$ as in \eqref{eq:LSI:candidateC}. Thus, the second condition of Theorem \ref{thm:LSI:zergalinski} holds with $\delta = \max\{\delta_0^*, \delta_1^*\}$ for all $n \in \N$. Thus, we can apply Theorem \ref{thm:LSI:zergalinski} to yield \eqref{eq:LSI:preLimitLSI} with $C_\theta = \hat{C}_{LS}/(1 - \delta)^2$ for $\hat{C}_{LS}$ as in Assumption \ref{as:results:LSI}.

 Next, we establish (2) and \eqref{eq:results:LSI}. Applying \eqref{eq:LSI:preLimitLSI} with $\Lambda = \psi_\nu^n$ for each $n$, dividing both sides by $2n+1$, sending $n \rightarrow \infty$, and using Theorem \ref{thm:results:2entropy} and Proposition \ref{lem:LSI:lyapConv}, one obtains \eqref{eq:results:LSI} with $C_0 = 2C_{LS}$. 

Finally, we show (3) and the exponential convergence \eqref{eq:results:expConv}. Let $(\mu, \gamma)$ be a linear growth solution to the $\kap$-CMVE. By Theorem \ref{thm:ap1:wellPosed}, \eqref{eq:LSI:finiteFI} is satisfied by $\mu_t$ for all $t > 0$. Therefore we can apply \eqref{eq:results:LSI} to $\mu_t$ for almost every $t \in (0, \infty)$, and \eqref{eq:results:expConv} holds by Theorem \ref{thm:results:lyap} and Gronwall's inequality.
\end{proof}
\subsection{Uniqueness of stationary distributions} In view of Theorem \ref{thm:results:stationary} and \ref{thm:results:zeros}, the set of stationary distributions of the $2$-MLFE are in one-to-one correspondence with the set of \FPE{} fixed points. To prove Theorem \ref{thm:results:uniqueRegime}, we invoke Corollary \ref{cor:LackerZhang} to obtain uniqueness of the Cayley fixed points from the results of \cite{lacker2023stationary}.
\label{sec:uniquenessRegime}
\begin{proof}[Proof of Theorem \ref{thm:results:uniqueRegime}] First, we show that there is a unique Cayley fixed point. By Corollary \ref{cor:LackerZhang}, it suffices to check that solutions to the fixed point problem of \cite{lacker2023stationary} (reproduced in Definition \ref{def:lackerZhang}) are unique. Since $(U, \intPot)$ satisfy Assumption \ref{as:results:lyapunov} and $\nabla U$ and $\nabla \intPot$ are locally bounded, by Assumption \ref{as:results:invMeasure} we have
\begin{equation*}
    \int_{\R^d} \int_{\R^d} \exp \Big( - U(x) - U(y) - 2 K(x - y) \big) dx dy \leq R_q^2,
\end{equation*}
and (1.15) of \cite{lacker2023stationary} holds. Inspection of the proof of Theorem 1.9 of \cite{lacker2023stationary} shows that its conclusions hold for the $\R^d$ version of the fixed point equation given in Definition \ref{def:lackerZhang}. Then by Corollary \ref{cor:LackerZhang} there is a unique \FPE{}. Thus by Theorem \ref{thm:results:zeros} $\pi$ is the unique zero of $\mathbb{I}_2$, and by Theorem \ref{thm:results:stationary}, if there exists a solution to the 2-MLFE with potentials $(U, W)$ and initial condition $\pi$, then $\pi$ is the unique stationary distribution of the \CMVE{}.

Next, we  turn to the proof of \eqref{eq:results:sfeForm}. It will suffice to show that for $\bbH_2^*$ as in Theorem \ref{thm:results:2entropy}, we have
\begin{equation}
    \bbH_2(\nu) - \bbH_2^* = \mH(\nu|\pi) - \mH(\bar{\nu}|\bar{\pi}), \quad \nu \in \mathcal{Q}_{2, d}.
    \label{eq:LSI:sfeForm}
\end{equation}
Substituting in $\nu = \pi$ in the above display yields $\bbH_2^* = \bbH_2(\pi)$ and hence \eqref{eq:results:sfeForm}. To this end, we will show that \begin{equation}
    \lim_{n \rightarrow \infty} \frac{1}{2n+1} \mH(\psi_\nu^n | \psi_\pi^n) = \mH(\nu|\pi) - \mH(\bar{\nu}|\bar{\pi}), \label{eq:entropy:psiNupsiPientropy}
\end{equation}
and that $(2n+1)^{-1}\mH(\psi^n_\nu|\psi^n_\pi)$ and $(2n+1)^{-1}\mH(\psi^n_\nu|\theta^n)$ agree asymptotically. In particular, by Theorem \ref{thm:results:2entropy} and \eqref{eq:entropy:psiNupsiPientropy}, the equality \eqref{eq:LSI:sfeForm} holds if we have
\begin{equation}
\lim_{n \rightarrow \infty} \frac{1}{2n+1} \Big( \mH(\psi_\nu^n | \psi_\pi^n) - \mH(\psi_\nu^n | \theta^n)\Big)  = 0.
\label{eq:entropy:sfeLimits}
\end{equation}

First, we obtain upper bounds on $\log \pi_0, \log \bar{\pi}$, and $\log \pi$. As discussed in Remark \ref{rk:potentials}, we can take $\intPot(0) = 0$ without loss of generality. Let $q$ be as in Assumption \ref{as:results:invMeasure} and $c \in \R$ be a lower bound on $q$. By \eqref{eq:results:fBound}, we have $|q(x)| \leq \max\{c, |U(x)|\}$ and thus by \eqref{eq:zeros:fixedPoint} and Assumption \ref{as:results:lyapunov}, we have for some $C \in (0, \infty)$ that
\begin{equation}
    |\log \pi_0(x)| \leq 2 \log \mathcal{Z}_{\pi_0} + \log R_q + C\big( 1 + |x|^2\big), \quad x \in \R^d.
    \label{eq:LSI:logPi0Bound}
\end{equation}
Recall from  \eqref{eq:zeros:altCBP} that
\begin{equation}
\label{eq:LSI:piForm}
    \pi(x_{-1}, x_0, x_{1}) = \frac{1}{\mathcal{Z}_{\pi_0}^2} e^{ - \frac{U(x_1) + U(x_{-1})}{2} - U(x_0) - \intPot(x_0 - x_1) - \intPot(x_0 - x_{-1})} \sqrt{\pi_0(x_1)\pi_0(x_{-1})}.
\end{equation}
Combining the above display with \eqref{eq:results:fBound} and \eqref{eq:LSI:logPi0Bound} implies the existence of $C_{0} \in (0, \infty)$ such that
\begin{equation}
    |\log\pi(\x)| \leq C_{0}(1 + |\x|^2), \quad \x \in (\R^d)^3
    \label{eq:LSI:logPiBound}
\end{equation}
Similarly, \eqref{eq:zeros:muBar}, \eqref{eq:results:fBound}, and \eqref{eq:LSI:logPi0Bound} implies that there exists $C_1 \in (0, \infty)$ such that
\begin{equation}
    |\log\bar{\pi}(x_0, x_{1})| \leq C_{1}(1 + |x_0|^2 + |x_1|^2), \quad (x_0, x_1) \in \R^d \times \R^d.
    \label{eq:LSI:logPiBarBound}
\end{equation}
Let $\nu \in \mathcal{Q}_{2, d}$. By the definition of $\psi_\nu^n$ in \eqref{eq:entropy:2reconstruction}, we have
\begin{equation*}
\begin{aligned}
    \mH(\psi_\nu^n| \psi^n_\pi) &= \int_{(\R^d)^{V_n}} \bigg[ \sum_{v = -n+1}^{n-1} \log \frac{\nu(x_{v-1}, x_v, x_{v+1})}{\pi(x_{v-1}, x_v, x_{v+1})} - \sum_{v = -n}^{n-1} \log \frac{\bar{\nu}(x_v, x_{v+1})}{\bar{\pi}(x_v, x_{v+1})}\bigg] \psi_\nu^n(d\x^{(n)}) \\
    &= \sum_{v = -n+1}^{n-1} \int_{(\R^d)^{V_n}}\Big( \log \nu(x_{v-1}, x_v, x_{v+1})  -  \log \pi(x_{v-1}, x_v, x_{v+1}) \Big)\psi_\nu^n(d\x^{(n)})  \\
    & + \sum_{ v = -n}^{n-1} \int_{(\R^d)^{V_n}} \Big(\log \bar{\nu}(x_v, x_{v+1}) - \log \bar{\pi}(x_v, x_{v+1}) \Big)\psi_\nu^n(d\x^{(n)})  
\end{aligned}
\end{equation*}
Since $\nu$ has finite entropy and second moment, by \eqref{eq:LSI:logPiBound} and \eqref{eq:LSI:logPiBarBound} we can apply Lemma \ref{lem:entropy:symmetry} with $f =\log \nu - \log \pi$ and $f = \log \bar{\nu} - \log \bar{\pi} $ to the previous display to obtain
\begin{equation}
\begin{aligned}
    \mH(\psi_\nu^n| \psi^n_\pi) = (2n-1) \mH(\nu|\pi) - (2n) \mH(\bar{\nu}|\bar{\pi}).
\end{aligned}\label{eq:LSI:psiNuPiEntropy}
\end{equation}
Therefore, dividing both sides of the above display by $2n+1$ and taking limits yields \eqref{eq:entropy:psiNupsiPientropy}.

Next, we show \eqref{eq:entropy:sfeLimits}. We first obtain alternative forms of $\mH(\psi^n_\nu|\theta^n)$ and $\mH(\psi^n_\nu|\psi^n_\pi)$ to facilitate the desired calculation. By \eqref{eq:LSI:gibbs}, we have
\begin{equation}
    \mH(\psi^n_\nu|\theta^n) =\log \mathcal{Z}^n + \int_{(\R^d)^{V_n}} \Big(\log \psi^n_\nu(\x^{(n)}) + H_n(\x^{(n)}) \Big) \psi^n_\nu(\x^{(n)})
    \label{eq:entropy:altEntropy1}
\end{equation}
Since $\pi$ is a \FPE{}, we have by \eqref{eq:zeros:1MRF}-\eqref{eq:zeros:zMu} that
\begin{equation*}
    \pi(x_1|x_0, x_{-1}) =  \pi(x_1|x_0) = \frac{1}{\mathcal{Z}_{\pi_0}} \exp\bigg( - \frac{U(x_0) + U(x_1)}{2} - \intPot(x_0 - x_1) \bigg) \pi_0(x_1)^{\frac{1}{2}} \pi_0(x_0)^{-\frac{1}{2}}.
\end{equation*}
Then, combining the first form of $\psi^n_\pi$ in \eqref{eq:entropy:altReconstruction} with the last display, \eqref{eq:LSI:piForm}, and \eqref{eq:LSI:hamiltonian} yields
\begin{equation}
    \psi^n_\pi\big(d\x^{(n)}\big) = \frac{1}{\mathcal{Z}_{\pi_0}^{2n}} \exp( - H_n(\x^{(n)}) ) \pi_0(x_{-n})^{\frac{1}{2}}\pi_0(x_n)^{\frac{1}{2}} d\x^{(n)},
    \label{eq:LSI:piReconstruction}
\end{equation}
for all $n \in \N$. Moreover, since $\psi^n_\pi$ is a probability measure, we have
\begin{equation}
    \label{eq:LSI:pinPfn}
    \tilde{\mathcal{Z}}^n := \mathcal{Z}_{\pi_0}^{2n} = \int_{(\R^d)^{V_n}} \exp\Big( - H_n\big(\x^{(n)}\big) \Big) \pi_0(x_{-n})^{\frac{1}{2}}\pi_0(x_n)^{\frac{1}{2}} d\x^{(n)}.
\end{equation}
Since $\nu \in \mPSb$, by \eqref{eq:LSI:logPi0Bound} we can apply Lemma \ref{lem:entropy:symmetry} with $f = \log \pi_0$ with the last two displays to obtain
\begin{equation*}
\begin{aligned}
    \mH(\psi^n_\nu|\psi_\pi^n) &= \log \tilde{\mathcal{Z}}^n + \int_{(\R^d)^{V_n}} \Big(\log \psi^n_\nu(\x^{(n)}) + H_n(\x^{(n)}) \Big) \psi^n_\nu(\x^{(n)})  + \int_{\R^d}\log \pi_0(x) \nu_0(dx).
\end{aligned}
\end{equation*}
Combining the last display with \eqref{eq:entropy:altEntropy1} yields
\begin{equation*}
    \mH(\psi^n_\nu|\theta^n) - \mH(\psi^n_\nu|\psi^n_\pi) = \log \mathcal{Z}^n - \log \tilde{\mathcal{Z}}^n - \int_{\R^d} \log \pi_0(x) \nu_0(dx).
\end{equation*}
The integral on the right hand side of the previous display is finite by \eqref{eq:LSI:logPi0Bound} and the fact that $\nu \in \mPSb$. Then \eqref{eq:entropy:sfeLimits} would follow from
\begin{equation}
    \lim_{n \rightarrow \infty} \frac{1}{2n+1}\big(\log \mathcal{Z}^n - \log \tilde{\mathcal{Z}}^n \big) = 0.
    \label{eq:LSI:pFnConvergence}
\end{equation}
Since $\bbI_2(\pi) = 0$ and $(U, \intPot)$ satisfy Assumption \ref{as:results:fixedPoint}, Proposition \ref{lem:zeros:entropyAndMoments} implies that $\pi_0$ has finite entropy. By \eqref{eq:LSI:piReconstruction} and \eqref{eq:LSI:gibbs}, and Lemma \ref{lem:entropy:symmetry} with $f = \log \pi_0$, we have
\begin{equation}
    \mH(\psi^n_\pi | \theta^n) = \int_{(\R^d)^{V_n}} \log \frac{\psi^n_\pi(\x^{(n)})}{\theta^n(\x^{(n)})} \psi^n_\pi(d\x^{(n)}) =  \log \mathcal{Z}^n - \log \tilde{\mathcal{Z}}^n + \int_{\R^d} \log \pi_0(x) \pi_0(dx).
    \label{eq:altEntropy2}
\end{equation}
Since $(U, \intPot)$ additionally satisfy Assumption \ref{as:results:LSI}, Theorem \ref{thm:results:LSI}(1) implies that for all  $n \in \N$, we have
\begin{equation}
    \mH(\psi^n_\pi | \theta^n) \leq C _\theta\mI(\psi^n_\pi | \theta^n),
    \label{eq:LSI:piLSI}
\end{equation}
where $C_\theta \in (0, \infty)$ is as in Theorem \ref{thm:results:LSI}. By \eqref{eq:LSI:gibbs} and \eqref{eq:LSI:piReconstruction}, we have
\begin{equation*}
\begin{aligned}
    \mI(\psi^n_\pi | \theta^n) &= \int_{(\R^d)^{V_n}} \bigg| \nabla \log \frac{d\psi^n_\pi}{d\theta^n}\big(\x^{(n)}\big)\bigg|^2 \psi^n_\pi\big(d\x^{(n)}\big) \\
    &= \int_{(\R^d)^{V_n}} \Big| \frac{1}{2}\nabla \log \pi_0(x_n) + \frac{1}{2}\nabla \log \pi_0(x_{-n})\Big|^2 \psi^n_\pi(d\x^{(n)})  \\
    & \leq \frac{1}{2} \int_{(\R^d)^{V_n}} \Big[\big| \nabla \log \pi_0(x_n)\big|^2 + \big|\nabla \log \pi_0(x_{-n})\big|^2\Big] \psi^n_\pi(d\x^{(n)})  \\
    &= \int_{\R^d} \big| \nabla \log \pi_0(x)\big|^2 \pi_0(dx),
\end{aligned}
\end{equation*}
where in the last equality, we used \eqref{eq:zeros:marginalFisher} and Lemma 8.1 with $f = |\nabla \log \pi_0|^2$. By \eqref{eq:altEntropy2}, \eqref{eq:LSI:piLSI}, the above display, and the fact that $\pi_0$ has finite entropy, we have
\begin{equation*}
\sup_{n \in \N}\Big\{\big|\log \mathcal{Z}^n - \log \tilde{\mathcal{Z}^n}\big| \Big\}  =  \sup_{n \in \N}\bigg\{ \mH(\psi^n_\pi | \theta^n) + \bigg| \int_{\R^d} \log \pi_0(x) \pi_0(dx)\bigg| \bigg\} < \infty.
\end{equation*}
Thus \eqref{eq:LSI:pFnConvergence}  and hence \eqref{eq:entropy:sfeLimits} hold. 
\end{proof}

\subsection{Convergence of renormalized Fisher information} 
\label{sec:LSI:convFI} In this section we prove Proposition \ref{lem:LSI:lyapConv}.  The following lemma will be useful in the proof.
\begin{lemma}
\label{lem:LSI:orthogonal}
    Let $\nu \in \mPS$ satisfy \begin{equation}
 \label{eq:LSI:finiteFI2}
     \int_{(\R^d)^{1 + \kap}} \Big(\big|\nabla_{\x} \log \nu(\x)\big|^2 \Big) \nu(d\x) < \infty.
 \end{equation}
    Fix $n \in \N$ with $n \geq 3$. The following properties hold for $u \in V_{n-2}$: 
    \begin{enumerate}
        \item For all $A \subset \{-n, \ldots, u+1\}$ and $f:(\R^d)^A \rightarrow \R^d$ such that \begin{equation}
        \label{eq:LSI:fInt}
\int_{(\R^d)^{V_n}} |f(x_A)|^2 \psi^n_\nu(d\x^{(n)}) < \infty,
\end{equation}
we have
\begin{equation}
    \int_{(\R^d)^{V_n}} f(x_A) \cdot \nabla_{x_u} \log \nu(x_{u+2}|x_{u+1}, x_u) \psi^n_\nu(d\x^{(n)}) = 0.
    \label{eq:LSI:f0}
\end{equation}
\item  For all $B \subset \{u-1, \ldots, n\}$ and $f:(\R^d)^B \rightarrow \R^d$ such that \begin{equation*}
\int_{(\R^d)^{V_n}} |f(x_B)|^2 \psi^n_\nu(d\x^{(n)}) < \infty,
\end{equation*}we have
\begin{equation*}
    \int_{(\R^d)^{V_n}} f(x_B) \cdot \nabla_{x_v} \log \nu(x_{u-2}|x_{u-1}, x_u) \psi^n_\nu(d\x^{(n)}) = 0.
\end{equation*}
    \end{enumerate}
\end{lemma}
\begin{proof}
    The two claims are proved in the exact same way. For brevity we prove only the first. By the third form in \eqref{eq:entropy:altReconstruction}, \eqref{eq:LSI:finiteFI2}, and \eqref{eq:LSI:fInt}, we can integrate out $x_i$ for $i \not \in  \{-n, \ldots, u+2\}$ and apply Fubini's theorem to obtain
    \begin{equation*}
    \begin{aligned}
        &    \int_{(\R^d)^{V_n}} f(x_A) \cdot \nabla_{x_u} \log \nu(x_{u+2}|x_{u+1}, x_u) \psi^n_\nu(d\x^{(n)}) 
    \\ =& \int_{(\R^d)^{u + 1 + n}}  \Xi(x_{u+1}, x_u) \cdot f(x_A)\nu(x_{-n}, x_{-n+1}, x_{-n+2}) \prod_{j =-n+2}^u \nu(x_{j+1}|x_j, x_{j-1}) \prod_{k = -n}^{u+1} dx_k,
    \end{aligned}
    \end{equation*}
    where
    \begin{equation*}
        \Xi(x_{u+1}, x_u) := \int_{\R^d}\big(\nabla_{x_u} \log \nu(x_{u+2}|x_{u+1}, x_u) \big) \nu(x_{u+2}|x_{u+1}, x_u)  dx_{u+2}.
    \end{equation*}
    Next we show
    \begin{equation}
        \Xi(x_{u+1}, x_u) = \nabla_{x_u} \int_{\R^d} \nu(x_{u+2}|x_{u+1}, x_u)  dx_{u+2} =0.
        \label{eq:LSI:Xi0}
    \end{equation}
    Then it suffices to justify the change of integral and derivative. We have
    \begin{equation}
    \begin{aligned}
        \Xi(x_{u+1}, x_u) &= \int_{\R^d} \nabla_{x_u} \nu(x_{u+2}|x_{u+1}, x_u) dx_{u+2} \\
        &= \frac{1}{\bar{\nu}(x_{u+1}, x_u)}\int_{\R^d}\nabla_{x_u} \nu(x_{u+2}, x_{u+1}, x_u) dx_{u+2}  - \frac{\nabla_{x_u} \bar{\nu}(x_{u+1}, x_u)}{\bar{\nu}(x_{u+1}, x_u)}.
    \end{aligned}
    \label{eq:LSI:Xi1}
    \end{equation}
    By Cauchy-Schwarz and \eqref{eq:LSI:finiteFI2}, we have
    \begin{equation*}
        \int_{(\R^d)^3} |\nabla \nu(\x)| d\x = \int_{(\R^d)^3} |\nabla  \log\nu(\x)|^2 \nu(\x) d\x < \infty, 
    \end{equation*}
    and therefore $
        |\nabla_{x_u} \nu(x_{u+2}, x_{u+1},x_u)| \in L^1\big((\R^d)^3\big).$ Thus by Fubini's theorem, we have
    \begin{equation*}
        \int_{\R^d}\nabla_{x_u} \nu(y, x_{u+1}, x_u) dy = \nabla_{x_u}\int_{\R^d} \nu(y, x_{u+1}, x_u) dy = \nabla_{x_u} \bar{\nu}(x_{u+1}, x_u),
    \end{equation*}
    The previous display and \eqref{eq:LSI:Xi1} therefore verify \eqref{eq:LSI:Xi0}, and we conclude that \eqref{eq:LSI:f0} holds.
\end{proof}

\begin{proof}[Proof of Proposition \ref{lem:LSI:lyapConv}]

The proof has a similar flavor to that of Theorem \ref{thm:results:2entropy}. Observe that by the definition of the Fisher information, we have
\begin{equation}
\label{eq:LSI:info1}
    I(\psi_\nu^{n} | \theta^n) = \sum_{v = -n}^n \int_{(\R^d)^{V_n}} \bigg|
    \nabla_{x_v} \log \frac{d\psi_\nu^n}{d\theta^n}(\x^{(n)}) \bigg|^2 \psi^n_\nu(d\x^{(n)}).
\end{equation}
Fix $v \in V_n$. By \eqref{eq:entropy:2reconstruction} and \eqref{eq:LSI:gibbs}, for $\psi^n_\nu$-a.e. $\x^{(n)} \in (\R^d)^{V_n}$ we have
\begin{equation}
\begin{aligned}
    \nabla_{x_v}\log \frac{d\psi_\nu^n}{d\theta^n} (\x^{(n)})
    &= \nabla_{x_v} \bigg( \sum_{u = -n + 1}^{n-1} \log \nu(x_{\bar{u}}) - \sum_{u = -n+1}^{n-2} \log \bar{\nu}(x_u, x_{u+1}) + \frac{1}{2}\sum_{(u, w) \in E^n} \newW(x_u, x_w) \bigg) \\
    &= \nabla_{x_v} \bigg( \sum_{u \in \bar{v}} \log \nu(x_{\bar{u}}) - \sum_{u \sim v}\bigg( \log \bar{\nu}(x_u, x_{v}) + \frac{1}{2}\newW(x_u, x_v) \bigg)\bigg).
\end{aligned}
\label{eq:LSI:gradFormula}
\end{equation}
We first obtain an \textit{a priori} bound on the $L^2(d\psi_\nu^n)$-norm of the above term. 

There exists $C \in (0, \infty)$ such that
\begin{equation*}
\begin{aligned}
    &\int_{(\R^d)^{V_n}} \bigg|
    \nabla_{x_v} \log \frac{d\psi_\nu^n}{d\theta^n}(\x^{(n)}) \bigg|^2 \psi^n_\nu(d\x^{(n)}) \\
    = & \int_{(\R^d)^{V_n}} \bigg|
    \nabla_{x_v}  \bigg(\sum_{u \in \bar{v}} \log \nu(x_{\bar{u}}) - \sum_{u \sim v} \Big(\log \bar{\nu}(x_u, x_v) - \frac{1}{2}Q(x_u, x_v)\Big) \bigg)\bigg|^2 \psi^n_\nu(d\x^{(n)}) \\
    \leq & C \int_{(\R^d)^{V_n}} \bigg[ \sum_{u \in \bar{v}}  |\nabla_{x_v} \log \nu(x_{\bar{u}})|^2 + \sum_{u \sim v} \Big(|\nabla_{x_v} \bar{\nu}(x_u, x_v)|^2 + |Q(x_u, x_v)|^2\Big) \bigg] \psi^n_\nu(d\x^{(n)}).
\end{aligned}
\end{equation*}
By \ref{eq:LSI:finiteFI} and Lemma \ref{lem:conv:marginalInformation}, we can apply Lemma \ref{lem:entropy:symmetry} with $f = |b(\x)|^2 + |\nabla_{x_v} \log \nu|^2 $ and $f = |\nabla_{x_v} \log \bar{\nu}|^2$ to observe that for some $\tilde{C} \in (0, \infty)$, we have
\begin{equation}
    \int_{(\R^d)^{V_n}} \bigg|
    \nabla_{x_v} \log \frac{d\psi_\nu^n}{d\theta^n}(\x^{(n)}) \bigg|^2 \psi^n_\nu(d\x^{(n)}) \leq \tilde{C}, \quad v \in V_n.
    \label{eq:LSI:aPrioriInfoBound}
\end{equation}

Next, we obtain more refined bounds on the above quantity for $v \in V_{n-2}$. By \eqref{eq:results:W} and \eqref{eq:results:drift} we have
\begin{equation*}
    \nabla_{x_v}\sum_{u \sim v}  \frac{1}{2}  \newW(x_u, x_v) = \nabla U(x_v) + \sum_{u \sim v} \nabla \intPot(x_v - x_u) = b(x_{\bar{v}}), \quad v \in V_{n-2}
\end{equation*}
Moreover for $v \in V_{n-2}$ and $\psi^n_\nu$-a.e. $\x^{(n)} \in (\R^d)^{V_n}$ we have 
\begin{equation*}
\begin{aligned}
    \nabla_{x_v} \bigg( \sum_{u \sim \bar{v}} \log \nu(x_{\bar{u}}) - \sum_{(u, v) \in E^n} \log \bar{\nu}(x_u, x_{u+1}) \bigg) = \nabla_{x_v} \log \nu(&x_{v-2} | x_{v - 1}, x_v)  + \nu(x_{\bar{v}}) \\  &+ \nabla_{x_v}\log \nu(x_{v+2}|x_{v+1}, x_v).
\end{aligned}
 \end{equation*}
For $v \in V_{n-2}$, the previous two displays together imply that for $\psi^n_\nu$-a.e. $\x^{(n)} \in (\R^d)^{V_n}$ we have
\begin{equation}
    \bigg| \nabla_{x_v}\log \frac{d\psi_\nu^n}{d\theta^n}(\x^{(n)}) \bigg|^2 = |\Theta_1(x_{\bar{v}})|^2 + |\Theta_2(x_{\bar{v}}, x_{\partial\bar{v}})|^2 + \Theta_1(x_{\bar{v}}) \cdot \Theta_2(x_{\bar{v}}, x_{\partial\bar{v}}),
    \label{eq:LSI:ThetaDecomp}
\end{equation}
where
\begin{equation*}
    \begin{aligned}
        \Theta_1(x_{\bar{v}}) &:= b(x_{\bv}) + \nabla_{x_v} \log \nu(x_{\bv}), \\
        \Theta_2(x_{\bar{v}}, x_{\partial\bar{v}}) &:= \nabla_{x_v} \log \nu(x_{v-2}|x_{v-1}, x_v) + \nabla_{x_v} \log \nu(x_{v+2}| x_{v+1}, x_v).
    \end{aligned}
\end{equation*}
We note that by \eqref{eq:LSI:finiteFI}, we can apply Lemma \ref{lem:entropy:symmetry} with $f = |\Theta_1(x_{\bar{v}})|^2$ to obtain
\begin{equation}
\label{eq:LSI:Theta1}
\int_{(\R^d)^{V_n}} |\Theta_1(x_{\bar{v}})|^2 \psi_\nu^n(d\x^{(n)}) =
\int_{(\R^d)^3} |b(\x)+ \nabla_{x_0} \log \nu(\x)|^2 \nu(d\x) < \infty.
\end{equation}
Therefore applying Lemma \ref{lem:LSI:orthogonal} with $f = \Theta_1$ yields
\begin{equation}
\label{eq:LSI:thetaDot}
    \begin{aligned}
&\int_{(\R^d)^{V_n}} \Big( \Theta_1(x_{\bar{v}}) \cdot \Theta_2(x_{\bar{v}}, x_{\partial\bar{v}})\Big) \psi^n_{\nu}(d\x^{(n)}) \\
            = & \int_{(\R^d)^{V_n}} \Big( \Theta_1(x_{\bar{v}}) \cdot \nabla_{x_v} \nu(x_{v-2}|x_{v-1}, x_v)\Big) \psi^n_{\nu}(d\x^{(n)}) \\
            &\qquad\qquad\qquad\qquad\qquad+ \int_{(\R^d)^{V_n}} \Big( \Theta_1(x_{\bar{v}}) \cdot \nabla_{x_v} \nu(x_{v+2}|x_{v+1}, x_v)\Big) \psi^n_{\nu}(d\x^{(n)}) \\ = &0.
    \end{aligned}
\end{equation}
By \eqref{eq:LSI:finiteFI} and Lemma \ref{lem:conv:marginalInformation}, we have can apply Lemma \ref{lem:entropy:symmetry} with $f = \nabla_{\x} \log \nu$ and $f = \nabla_{x_1} \log \bar{\nu}$ to obtain
\begin{equation*}
\begin{aligned}
    &\sup_{v \in V_{n-2}}\int_{(\R^d)^{V_n}} \Big(|\nabla_{x_v} \log \nu(x_{v-2}|x_{v-1}, x_v)|^2 + |\nabla_{x_v} \log \nu(x_{v+2}|x_{v+1}, x_v)|^2\Big) \psi^n_{\nu}(d\x^{(n)}) < \infty
\end{aligned}    
\end{equation*}
Thus applying Lemma \ref{lem:LSI:orthogonal} and then Lemma \ref{lem:entropy:symmetry} with $f = \nabla_{x_v} \log \nu(x_{v\pm2}|x_{v\pm1}, x_v)$, yields
\begin{equation}
\begin{aligned}
&\int_{(\R^d)^{V_n}} |\Theta_2(x_{\bar{v}}, x_{\partial \bar{v}})|^2 \psi_\nu^n(d\x^{(n)})
\\ = &\int_{(\R^d)^{V_n}} \Big(|\nabla_{x_v} \log \nu(x_{v+2}|x_{v+1}, x_v)|^2 + |\nabla_{x_v} \log \nu(x_{v-2}|x_{v-1}, x_v)|^2 \Big) \psi_\nu^n(d\x^{(n)})
\\ = &2\int_{(\R^d)^3} |\nabla_{x_1} \log \nu(x_{-1}|x_0, x_1)|^2 \nu(d\x).\end{aligned}
\label{eq:LSI:Theta2}
\end{equation}
Combining \eqref{eq:LSI:ThetaDecomp}, \eqref{eq:LSI:Theta1}, \eqref{eq:LSI:thetaDot}, and \eqref{eq:LSI:Theta2} with \eqref{eq:results:info} yields
\begin{equation*}
\begin{aligned}
\int_{(\R^d)^{V_n}}  \bigg| \nabla_{x_v}\log \frac{d\psi_\nu^n}{d\theta^n}(\x^{(n)}) \bigg|^2 \psi^n_\nu(d\x^{(n)}) = \bbI_2(\nu), \quad v \in V_{n-2}.
\end{aligned}
\end{equation*}
Thus, \eqref{eq:LSI:info1}, \eqref{eq:LSI:aPrioriInfoBound}, and the previous display  together imply
\begin{equation*}
    \frac{2n-3}{2n+1}\info_2(\nu) \leq \frac{1}{2n+1}\mathcal{I}(\psi^{n,2}_\nu|\theta^n) \leq \frac{4\tilde{C}}{2n+1} + \frac{2n-3}{2n+1}\info_2(\nu).
\end{equation*}
The conclusion follows on taking $n\rightarrow \infty$ in the above display. \end{proof}

\appendix
\section{Linear Fokker Planck Equations}
\label{_sec:ap:PDE}
In this appendix we prove several results about Fokker-Planck equations with drift satisfying a linear growth condition. First, we recall the definition of the following Sobolev space \cite{fokkerPlanck}. For $T > 0$, we define $\R^m_T := \R^m \times (0, T)$. For $u : \R^m_T \rightarrow \R$, define the $H^{s, p}$ norm by
\begin{equation*}
    \|u\|_{H^{s, p}(\R^m_T)} := \bigg[\int_0^T \|u_t\|_{W^{s, p}(\R^m)}^p dt\bigg]^{\frac{1}{p}},
\end{equation*}
and let $H^{s, p}(\R^m_T)$ denote the of space of measurable functions on $\R^m_T$ with finite $H^{s, p}$ norm.

Here, we include the necessary PDE techniques to justify the calculation of the derivative of the candidate Lyapunov function. Let $f: [0, T] \times \mathbb{R}^m \rightarrow \mathbb{R}^m$ be a measurable function. We consider the Fokker-Planck equation:
\begin{equation}
\label{eq:ap1:formalPDE}
    \partial_t \mu_t = \Delta \mu_t + \nabla \cdot (f_t \mu_t)
\end{equation}
We define weak solutions to the Fokker-Planck equation \eqref{eq:ap1:formalPDE} in the sense of Proposition 6.1.2(iii) in \cite{fokkerPlanck}.
\begin{definition}[Weak solution] A family of measures $\{\mu_t\}_{t \in [0, T]}$ is a weak solution to \eqref{eq:ap1:formalPDE} with drift $f$ and initial condition $\mu_0 \in \mP(\R^m)$ if for all $\phi \in C^{2, 1}(\R^m_T) \cap C(\R^m \times [0, T))$ such that there exists $R \in (0, \infty)$ such that for all $t \in [0, T]$ and $x \not \in B_R$, we have $\phi_t(x) = 0$,  we have
\begin{equation*}
    \int_{\R^m} \phi_t(x) \mu_t(dx) = \int_{\R^m} \phi_0(x) \mu_0(dx) + \int_0^t \int_{\R^d} (\partial_s \phi_s(x) + \Delta _s\phi(x) - f_s(x) \cdot \nabla \phi_s(x) )\mu_s(dx) ds
\end{equation*}
for almost every $t \in [0, T]$. We let $\mu(dt, dx) := \mu_t(dx) dt$ denote such a weak solution.
\label{def:ap1:wp}
\end{definition}
\begin{theorem}[Well-posedness of Fokker-Planck equations] Fix $T \in (0, \infty)$. Suppose $f:[0, T] \times \mathbb{R}^m \rightarrow \mathbb{R}^m$ is a measurable function that satisfies the linear growth condition \begin{equation}
    \sup_{t \leq T}|f_t(x)| \leq C_f(1 + |x|), \quad \text{a.e. }x \in \R^d,
    \label{eq:ap1:linGrowthCondition}
\end{equation}
for some $C_f \in (0, \infty)$. Suppose $\mu_0 \in \mP(\R^m)$ satisfies has finite entropy and second moment. Then the following properties hold.
\begin{enumerate}
    \item The Fokker-Planck equation  \eqref{eq:ap1:formalPDE} has a unique weak solution $\{\mu_t\}_{t \in [0, T]}$ such that $\mu_t \in \mP(\R^m)$ is a probability measure for all $t \in [0, T]$. Moreover, $\mu_t$ has a finite second moment, that is,
\begin{equation}
\label{eq:ap1:moments}
    \sup_{t \in [0, T]} \int_{\R^m} |x|^2 \mu_t(x)dx < \infty.
\end{equation}    
    \item There exists a positive locally H\"older continuous function $\mu:[0, T] \times \R^{m} \rightarrow (0, \infty)$ such that $\mu_t(dx) = \mu_t(x) dx$.  The H\"older coefficient and exponent of $\mu$ depends only on $(m,T)$, the linear growth condition of $f$, and the initial condition $\mu_0$. Moreover, $\mu_t \in W^{1, p}_{loc}(\R^m)$ for almost every $t \in [0, T]$ and $\mu \in H^{1, p}_{loc}(\R^m_T)$ for all $p \geq 1$.
    \item  For every $t \in [0, T]$, we have
\begin{equation}
\label{eq:ap1:logGrad}
    \int_0^t \int_{\R^{m}} \frac{|\nabla \mu_s(x)|^2}{\mu_s(x)} dx ds < \infty.
\end{equation}
Moreover, we have $\mu \in H^{1, 1}(\R^m_T)$ and $\mu_t \in W^{1, 1}(\R^m)$ for almost every $t \in (0, T)$.  
\item We have for all $t \in [0, T]$ that
\begin{equation}
    \int_{\R^m} \mu_t(x) |\log \mu_t(x)| dx < \infty
    \label{eq:ap1:stupidEntropy}
\end{equation}
and
\begin{equation}
        \int_0^T \int_{\R^m} \mu_t(x) |\log \mu_t(x)| dx < \infty.
        \label{eq:ap1:stupidEntropy2}
\end{equation}
\end{enumerate}

\label{thm:ap1:wellPosed2}
\end{theorem}
\begin{proof}
We show property (1) first. Since $f$ satisfies the linear growth condition \eqref{eq:ap1:linGrowthCondition}, it is locally bounded. Then Theorem 9.4.8 of \cite{fokkerPlanck} gives existence and uniqueness of probability solutions to the Cauchy problem. Since $\{\mu_t\}_{t \in [0, T]}$ is a probability solution, the local boundedness of $f$ implies that $f \in L^p_{\text{loc}}(\mu)$ for any $p \geq 1$. Lemma 9.1.1 of \cite{fokkerPlanck} associates solutions of the Cauchy problem with solutions to the Fokker-Planck equation \eqref{eq:ap1:formalPDE} in the sense of Definition \ref{def:ap1:wp}. By Example 7.1.3 of \cite{fokkerPlanck}, the linear growth condition \eqref{eq:ap1:linGrowthCondition} implies \eqref{eq:ap1:moments}.

Next, we turn to property (2). By Corollary 6.4.3 of \cite{fokkerPlanck}, since $f \in L^p_{loc}(\mu)$, then $\mu$ has a locally H\"older continuous density. Moreover, a careful analysis of the proof of Corollary 6.4.3 of \cite{fokkerPlanck} shows that the H\"older coefficient depends only on degree of the integrability of $f$, which is governed by the linear growth condition, as stated in Theorem 3.7 of \cite{conforti2023projected}. Moreover Corollary 6.4.3 of \cite{fokkerPlanck} implies that for all $p \in [1, \infty)$, $\mu_t \in W^{1, p}_{loc}(\R^m)$ for almost every $t \in [0, T]$ and $\mu \in H^{1, p}_{loc}(\R^m_T)$. By Corollary 8.3.7 of \cite{fokkerPlanck}, local boundedness of $f$ implies that the density $\mu$ is positive on all of $\R_T^m$. 

Next, we show property (3). Due to the linear growth condition \eqref{eq:ap1:linGrowthCondition} of $f$ and finite entropy of $\mu_0$, we can use Theorem 7.4.1 of \cite{fokkerPlanck} to deduce \eqref{eq:ap1:logGrad} and that $\mu_t \in W^{1, 1}(\R^{m})$ for almost every $t \in [0, T]$. The Cauchy-Schwarz inequality and \eqref{eq:ap1:logGrad} implies that $\mu \in H^{1, 1}(\R^m_T)$. 

Finally, we establish the property (4). By the linear growth condition \eqref{eq:ap1:linGrowthCondition} on $f$, we can use \eqref{eq:ap1:moments} and the proof of Proposition 8.2.5 in \cite{fokkerPlanck} to obtain \eqref{eq:ap1:stupidEntropy}.
\end{proof}

The next lemma is used to compute $\energy_\kap(\mu_t)$ in Theorem \ref{thm:results:lyap}. If solutions to \CMVE{} were smooth, then the following lemma would be a trivial consequence of the weak formulation of the Fokker-Planck equation in Definition \ref{def:ap1:wp}. However, the conditional expectation $\eta$ defined in \eqref{eq:not:eta} may be highly irregular. Since we only consider linear growth solutions to the \CMVE{}, our result must be established for Fokker-Planck equations with measurable drift satisfying a linear-growth condition. The proof is a standard truncation and mollification argument.

\begin{lemma}
\label{lem:ap1:lyapunov1}
Suppose $f:[0, T]\times \R^m \rightarrow \R^m$  and $\mu_0 \in \mP(\R^m)$ satisfy the conditions of Theorem \ref{thm:ap1:wellPosed2}. Let $g: \R^m \rightarrow \R$ be a $C^1$ function such that $\nabla g$ satisfies a linear growth condition. Let $\{\mu_t\}_{t \in [0, T]}$ be the weak solution to \eqref{eq:ap1:formalPDE}, which exists by (1) of Theorem \ref{thm:ap1:wellPosed2}. Then for almost every $0 < r < t < T$, we have
\begin{equation}
    \begin{split}
        \int_{\R^m} \mu_t \big( \log \mu_t + g \big) dx - &\int_{\R^m} \mu_r \big( \log \mu_r + g \big) dx = \\
        & - \int_r^t \int_{\R^m} (\nabla \mu_s + f_s \mu_s) \cdot \bigg(\frac{\nabla \mu_s}{\mu_s}  + \nabla g \bigg) dx ds.
    \end{split}
    \label{eq:ap1:lyapunovPart1}
\end{equation}
    
\end{lemma}
\begin{proof}
By (3) of Theorem \ref{thm:ap1:wellPosed2}, there exists a set $S_T \subset [0, T]$ of full Lebesgue measure such that $\mu_t \in W^{1, 1}(\R^m)$. We prove \eqref{eq:ap1:lyapunovPart1} for $r, t \in S_T$ and then extend to all $t \in [0, T]$ by continuity of the right hand side. Fix $r, t \in S_T$ such that $r <t$. Define the mollifier $\omega \in C_0^\infty(\R^m)$ by
    \begin{equation}
        \omega(x) =: c\exp \bigg( -\frac{1}{1 - |x|^2}\bigg) 1_{\{|x| < 1\}},
        \label{eq:ap1:mollifier}
    \end{equation}
    where $c$ is the constant that makes $\omega$ a probability density. Then $\omega \in C_0^\infty(\R^m)$. Fix $\eps > 0$. Define 
    \begin{equation*}
        \omega_\eps(x) := \eps^{- (1 + \kappa)} \omega \bigg( \frac{x}{\eps}\bigg).
    \end{equation*}
    Given any measurable function $\varphi$, We denote mollification by $\omega_\eps$ by:
    \begin{align*}
        (\varphi)_\eps(x) := \int_{\R^m} \varphi(y) \omega_\eps(x - y) dy.
    \end{align*}
    Consider the mollified solution $\mu_t^\eps := (\mu_t)_\eps$. We have (e.g. by the proof of Lemma 2.4 in \cite{bogachev2016distances}) that $\{\mu_t^\eps\}_{t \in [0, T]}$ is a classical (smooth) solution the equation
    \begin{equation*}
        \partial_t \mu_t^\eps = \Delta \mu_t^\eps + \hat{f}_t^\eps \cdot \nabla \mu_t^\eps,
    \end{equation*}
    with initial condition $\mu_0^\eps$, where
    \begin{equation*}
        \hat{f}_t^\eps := \frac{(f_t \mu_t)_\eps}{\mu_t^\eps}.
    \end{equation*}
    Note that $\hat{f}_t^\eps$ is well defined since by (2) of Theorem \ref{thm:ap1:wellPosed2}, $\mu_t$ is positive, and thus $\mu_t^\eps$ is positive. Then $\mue:[0, T] \times \R^m \rightarrow \R$ can be identified as the integral in time of a smooth function of space, and hence is absolutely continuous in time and smooth in space. By \eqref{eq:ap1:linGrowthCondition}, we have
   \begin{align*}
       \hat{f}_t^\eps(x) &= \frac{1}{\mu_t^\eps(x)} \int_{|y - x| < \eps} f_t(y) \mu_t(y) \omega_\eps(x - y) dy \\
       & \leq \frac{C_f}{\mu_t^\eps(x)} \int_{|y - x| < \eps} (1 + |y|) \mu_t(y) \omega_\eps(x - y) dy \\
       & \leq \frac{C_f(1 + \eps + |x|)}{\mu_t^\eps(x)} \int_{|y - x| < \eps} \mu_t(y) \omega_\eps(x - y) dy \\
       & \leq 2C_f( 1 + |x|),
   \end{align*}
   thus establishing that \eqref{eq:ap1:linGrowthCondition} holds with $f$ and $C_f$ replaced with $f^\eps$ and $C_{{f^\eps}} := 2C_f$ respectively. Let $\psi\in C^\infty_0(\R)$ be a compactly supported function taking values in $[0, 1]$ that is identically $1$ on $[-1/2, 1/2]$ and $0$ outside the ball of radius $[-1, 1]$. For $R \in (0, \infty)$, we define the functions $\psi_R \in C_0^\infty(\R^m)$ by
    \begin{align}
        \psi_R(x) := \psi \bigg( \frac{|x|}{R}\bigg),
        \label{eq:ap1:psiR}
    \end{align}and the function $\Psi^{\eps, R}: [0, T] \times \R^m \rightarrow \R$ by
    \begin{equation}
        \Psi_t^{\eps, R}(x) := \big(\log \mu_t^\eps(x) + g(x) \big) \psi_R(x).
        \label{eq:ap1:bigPsi}
    \end{equation}
     By the positivity of $\mu_t^\eps$, if follows  that $\Psi^{\eps, R}$ lies in $ C^{2, 1}(\R^m \times [0, T)) \times C(\R^m \times [0, T))$ and is compactly supported. By Theorem \ref{thm:ap1:wellPosed2}, Definition \ref{def:ap1:wp}, and integration by parts, we have
\begin{align}
     \int_{\R^m} \Psi^{\eps, R}_t \mu^\eps_t dx - \int_{\R^m} \Psi^{\eps, R}_r \mu^\eps_r dx 
    = I_1(\eps, R) - I_2(\eps, R),
    \label{eq:ap1:I0}
\end{align}
where
\begin{align}
    I_1(\eps, R) &:= \int_r^t \int_{\R^m} \partial_s \Psi^{\eps, R}_s \mu_s^\eps dx ds,
    \label{eq:ap1:I1} \\
    I_2(\eps, R) &:= \int_r^t \int_{\R^m} (\nabla \mu_s^\eps + \mu_s^\eps \hat{f}_t^\eps )\cdot  \nabla \Psi^{\eps, R}_s dx ds.
        \label{eq:ap1:I2}
\end{align}
We now compute the limits as $\eps \rightarrow 0$ of both sides. We start with the entropy terms. For any $R \in (0, \infty)$, by continuity of $\mu_t$, which follows from Theorem \ref{thm:ap1:wellPosed2}(2), there exists $C_R \in (0, \infty)$ such that
\begin{equation}
\|\mue\|_{L^\infty([r, t] \times B_{R+ \eps})} \leq \|\mu\|_{L^\infty([r, t] \times B_{R+1})} < C_R.
    \label{eq:ap1:upperBd}
\end{equation}
Since $\mu_t$ is also positive and continuous by Theorem \ref{thm:ap1:wellPosed2}(2), there exists $c_R \in (0, \infty)$ such that we have
\begin{equation}
\label{eq:ap1:lowerBd}
\begin{aligned}
    \inf_{s \in [r, t]} \inf_{x \in B_R} \mu_s^{\eps}(x) &= \inf_{s \in [r, t]} \inf_{x \in B_R} \int_{|y| \leq \eps} \mu_s(x -y) \omega_\eps(y) dy \\ &\geq \inf_{s \in [r, t]} \inf_{x \in B_{R + \eps}} \mu_s(x) \\ &\geq \inf_{s \in [r, t]} \inf_{x \in B_{R + 1}} \mu_s(x) \\ &=: c_{R} > 0.    
\end{aligned}
\end{equation}
In particular, $c_{R}$ does not depend on $\eps$. Together \eqref{eq:ap1:upperBd} and \eqref{eq:ap1:lowerBd} imply that $\mu^\eps_t(x) \log \mu_t^\eps(x)$ is uniformly bounded in $\eps$ from above and below for $x \in B_{R+1}$. Since \eqref{eq:ap1:psiR} ensures that $\psi_R$ is supported on $B_R$, by the bounded convergence theorem, we have
\begin{equation*}
    \lim_{\eps \rightarrow 0}\int_{\R^m} \psi_R(x) \mu_t^{\eps}(x) \log \mu_t^{\eps}(x) dx  = \int_{\R^m} \psi_R(x) \mu_t(x) \log \mu_t(x) dx.
\end{equation*}
Then by \eqref{eq:ap1:stupidEntropy} and the dominated convergence theorem,
\begin{equation*}
    \lim_{R \rightarrow \infty}  \lim_{\eps \rightarrow 0}\int_{\R^m} \psi_R(x) \mu_t^{\eps} \log \mu_t^{\eps} dx = \int_{\R^m} \mu_t \log \mu_t dx.
\end{equation*}
We can repeat the above argument for $\mu_r$ to obtain. 
\begin{equation}
\label{eq:ap1:1part1}
    \lim_{R \rightarrow \infty}  \lim_{\eps \rightarrow 0}\int_{\R^m} \psi_R(x) \mu_\tau^{\eps} \log \mu_\tau^{\eps} dx = \int_{\R^m} \mu_\tau \log \mu_\tau dx, \quad \tau \in \{r, t\}.
\end{equation}

Now, consider $I_1$. By absolute continuity in time of $\mue_s$, we have that $\partial_s \mue_s \in L^1_{\text{loc}}(\R^m_T)$ and
\begin{align*}
    \int_r^t \int_{\R^m} \partial_s \Psi_s^{\eps, R} \mue_s dx ds &= \int_r^t \int_{\R^m} \psi_R(x) \mue_s \partial_s \log \mue_s dx ds \\
    &= \int_r^t \int_{\R^m} \psi_R(x) \partial_s \mue_s dx ds \\
    &= \int_{\R^m} \psi_R(x) \bigg[\int_r^t \partial_s \mue_s ds \bigg]dx \\
    &= \int_{\R^m} \psi_R(x) (\mue_t - \mue_r) dx.
\end{align*}
Combining the above display, \eqref{eq:ap1:upperBd}, and the bounded convergence theorem, we obtain
\begin{align*}
    \lim_{\eps \rightarrow 0}\int_r^t \int_{\R^m} \partial_s \Psi_s^{\eps, R} \mue_s dx ds = \lim_{\eps \rightarrow 0}\int_{\R^m} \psi_R(x) (\mue_t - \mue_r) dx = \int_{\R^m} \psi_R(x) (\mu_t - \mu_r) dx.
\end{align*}
Therefore by \eqref{eq:ap1:I1}, the dominated convergence theorem, and the fact that $\mu_t$ and $\mu_0$ have probability densities, it follows that
\begin{equation}
\label{eq:ap1:1part2}
    \lim_{R \rightarrow \infty} \lim_{\eps \rightarrow 0}I_1(\eps, R) = \lim_{R \rightarrow \infty} \lim_{\eps \rightarrow 0}\int_r^t \int_{\R^m} \partial_s \Psi_s^{\eps, R} \mue_s dx ds = \lim_{R \rightarrow \infty}\int_{\R^m} \psi_R(x) (\mu_t - \mu_0) dx = 0.
\end{equation}

Next, recalling \eqref{eq:ap1:bigPsi}, we decompose $I_2$ from  \eqref{eq:ap1:I2} into two parts: 
\begin{align*}
    I_2(\eps, R) &= I_{2, 1}(\eps, R) + I_{2, 2}(\eps, R),
\end{align*}
where
\begin{align*}
    I_{2, 1}(\eps, R) &:= \int_r^t \int_{\R^m} (\nabla \mu_s^\eps + \mu_s^\eps \hat{f}_t^\eps )\cdot  \nabla (\log \mue_s + g)\psi_R \,dx ds \\
        I_{2, 2}(\eps, R) &:= \int_r^t \int_{\R^m} (\nabla \mu_s^\eps + \mu_s^\eps \hat{f}_t^\eps )\cdot  \nabla\psi_R (\log \mue_s + g) \,dx ds.
\end{align*}
By (2) of Theorem \ref{thm:ap1:wellPosed2}, $\mu \in H^{1, 2}_{loc}(\R^m_T)$ and by \eqref{eq:ap1:mollifier}, $\omega$ is compactly supported; therefore, we have $(\nabla \mu_t)_\eps = \nabla \mu_t^\eps$. Then by standard properties of mollifiers (see for example Appendix C of \cite{evans1998PDE}) it follows that $\nabla \mu^\eps \rightarrow \nabla \mu$ in $L^2_{loc}(\R^m_T)$. By Theorem 4.9 in \cite{brezis2011Functional}, there exists a subsequence $\{\eps_k\}$ converging to 0 and a function $h \in L^1(B_R \times [0, T])$ such that
\begin{equation}
    \sup_{k}|\nabla \mu^{\eps_k}_s(x)| +  \sup_{k}|\nabla \mu^{\eps_k}_s(x)|^2 \leq h_s(x), \quad (x, s) \in B_R \times [0, T].
    \label{eq:ap1:gradientL1bound}
\end{equation}
Now we compute the limits in $\eps$ and $R$ for $I_{2, 1}(\eps, R)$ by further decomposing it as
\begin{equation}
    \begin{aligned}
    I_{2, 1}(\eps, R) 
    &= I_{2, 1, 1}(\eps, R) + I_{2, 1, 2}(\eps, R),
\end{aligned}
    \label{eq:ap1:I21}
\end{equation}
where
\begin{equation}
    \begin{aligned}
    I_{2, 1, 1}(\eps, R) &:= \int_r^t \int_{\R^m} \psi_R \frac{|\nabla \mue_s|^2}{\mue_s} dx ds, \\
     I_{2, 1, 2}(\eps, R) &:=\int_r^t \int_{\R^m} \psi_R \bigg(\nabla \mue_s \nabla g + f_s^\eps \nabla \mue_s + \mue_sf_s^\eps \nabla g \bigg)dx ds.
     \label{eq:ap1:I2110}
\end{aligned}
\end{equation}
Recall the definition of $c_{R}$ from \eqref{eq:ap1:lowerBd}. By \eqref{eq:ap1:lowerBd} and \eqref{eq:ap1:gradientL1bound} it follows that
\begin{align*}
    \sup_{k}\frac{|\nabla \mu_s^{\eps_k}|^2}{\mu_s^{\eps_k}} \leq \frac{1}{c_{R}}h_s(x), \quad (x, s) \in B_R \times [r, t].
\end{align*}
Together with \eqref{eq:ap1:I2110} and the dominated convergence theorem, this implies that
\begin{equation}
    \label{eq:ap1:I211}
\lim_{k \rightarrow \infty} I_{2, 1, 1}(\eps_k, R) = \lim_{k \rightarrow \infty} \int_{r}^t \int_{\R^m} \psi_R \frac{|\nabla \mu^{\eps_k}_s|^2}{\mu^{\eps_k}_s} dx ds = \int_{r}^t \int_{\R^m} \psi_R \frac{|\nabla \mu_s|^2}{\mu_s} dx ds.
\end{equation}
On the other hand, by \eqref{eq:ap1:logGrad} and the dominated convergence theorem, we have
\begin{align*}
    \lim_{R \rightarrow \infty} \lim_{k \rightarrow \infty}I_{2, 1, 1}(\eps_k, R)  = \lim_{R \rightarrow \infty} \int_{r}^t \int_{\R^m} \psi_R \frac{|\nabla \mu_s|^2}{\mu_s} dx ds = \int_r^t\int_{\R^m}\frac{|\nabla \mu_s|^2}{\mu_s} dx ds.
\end{align*}
Next, observe that \eqref{eq:ap1:gradientL1bound}, the fact that $g \in C^2(\R^m)$, and the dominated convergence theorem yield the limit
\begin{align*}
    \lim_{k \rightarrow \infty} I_{2, 1, 2}(\eps_k, R) = \int_r^t \int_{\R^m} \psi_R \bigg(\nabla \mu_s \nabla g + f_s \nabla \mu_s + \mu_sf_s \nabla g \bigg)dx ds.
\end{align*}
Also, for any measurable $\varphi$, 
 by \eqref{eq:ap1:logGrad},  \eqref{eq:ap1:moments}, and the Cauchy-Schwarz inequality, it follows that that \begin{equation}
    \label{eq:ap1:polyGrowthIntegrability}
    \begin{aligned}
            \int_{r}^t\int_{\R^m} \varphi(x) \cdot \nabla \mu_s(x) dx &= \int_r^t\int_{\R^m} \sqrt{\mu_s(x)}\varphi(x) \cdot \frac{\nabla \mu_s(x)}{\sqrt{\mu_s(x)}} dx \\ &\leq \bigg[\int_r^t \int_{\R^m} |\varphi(x)|^2 \mu_s(x) dx\bigg]^{\tfrac{1}{2}}\bigg[ \int_r^t \int_{\R^m} \frac{|\nabla \mu_s(x)|^2}{\mu_s(x)} dx \bigg]^{\tfrac{1}{2}},
    \end{aligned}
\end{equation} 
Hence if $\varphi \in L^2(\mu)$ then $\varphi \cdot \nabla \mu \in L^1(\R^m \times [0, t])$.
Using the above display with $\varphi(x) = f_s(x) + \nabla g(x)$, and applying \eqref{eq:ap1:gradientL1bound}, the linear growth of $\nabla g$, \eqref{eq:ap1:moments}, and the dominated convergence theorem, we obtain
\begin{equation}
    \lim_{R \rightarrow \infty} \lim_{k \rightarrow \infty} I_{2, 1, 2}(\eps_k, R) = \int_r^t \int_{\R^m} \bigg((f_s + \nabla g)\cdot \nabla \mu_s + \mu_sf_s \nabla g \bigg)dx ds.
    \label{eq:ap1:I212}
\end{equation}
Therefore, combining \eqref{eq:ap1:I21},  \eqref{eq:ap1:I211}, and \eqref{eq:ap1:I212}, we conclude that
\begin{equation}
    \label{eq:ap1:1part3}
    \lim_{R \rightarrow \infty} \lim_{k \rightarrow \infty}I_{2, 1}(\eps_k, R) = \int_r^t \int_{\R^m} \frac{|\nabla \mu_s|^2}{\mu_s} dx ds + \int_r^t \int_{\R^m} \bigg(\nabla \mu_s \nabla g + f_s \nabla \mu_s + \mu_sf_s \nabla g \bigg)dx ds.
\end{equation}

In view of \eqref{eq:ap1:I0}-\eqref{eq:ap1:I2}, \eqref{eq:ap1:1part1}, \eqref{eq:ap1:1part2}, \eqref{eq:ap1:1part3}, to prove \eqref{eq:ap1:lyapunovPart1} if suffices to show that \begin{equation} \lim_{R \rightarrow \infty}\lim_{k \rightarrow \infty} I_{2, 2}(\eps_k, R) = 0. \label{eq:ap1:lastthing} \end{equation}
Once again, we decompose:
\begin{align*}
    I_{2, 2}(\eps, R) 
    &= I_{2, 2, 1}(\eps, R) + I_{2, 2, 2}(\eps, R),
\end{align*}
where
\begin{align*}
    I_{2, 2, 1}(\eps, R) &:= \int_r^t \int_{\R^m} \nabla \mu_s^\eps \cdot \nabla \psi_R \log \mue_s dx ds, \\
    I_{2, 2, 2}(\eps, R) &:=  \int_r^t \int_{\R^m} f_s^\eps \cdot \nabla \psi_R ( g \mue_s + \mue_s \log \mue_s) dx ds.
\end{align*}
Since $\psi_R$ has compact support, we integrate by parts to obtain
\begin{align*}
    I_{2, 2, 1}(\eps_k, R) &= - \int_r^t \int_{\R^m} \mu_s^{\eps_k} \nabla \cdot( \nabla \psi_R \log \mu^{\eps_k}_s) dx ds  \\
    &= - \int_r^t \int_{\R^m} \Delta \psi_R \mu_s^{\eps_k} \log \mu^{\eps_k}_s dx ds - \int_r^t \int_{\R^m} \nabla \mu_s^{\eps_k} \cdot \nabla \psi_R  dx ds.
\end{align*}
Applying Theorem \ref{thm:ap1:wellPosed2}, \eqref{eq:ap1:lowerBd}, and \eqref{eq:ap1:gradientL1bound} yields
\begin{align*}
\lim_{k \rightarrow \infty}    I_{2, 2, 1}(\eps_k, R) = - \int_r^t \int_{\R^m} \Delta \psi_R \mu_s \log \mu_s dx ds - \int_r^t \int_{\R^m} \nabla \mu_s \cdot \nabla \psi_R  dx ds.
\end{align*}
Since $\mu \in H^{1, 1}(\R^m_T)$, and $\sup_{R > 0} (\|\nabla \psi_R\|_{L^\infty(\R^m)} + \|\nabla^2 \psi_R\|_{L^\infty(\R^m)}) < \infty$ we can apply the above display, \eqref{eq:ap1:stupidEntropy2}, and the dominated convergence theorem to obtain
\begin{equation*}
    \lim_{R \rightarrow \infty} \lim_{k \rightarrow \infty}    I_{2, 2, 1}(\eps_k, R) = 0.
\end{equation*}
 By the uniform bound \eqref{eq:ap1:upperBd} and \eqref{eq:ap1:lowerBd}, we can use the bounded convergence theorem to get
\begin{align*}
    \lim_{k \rightarrow \infty} I_{2, 2, 2}(\eps_k, R) &= \lim_{k \rightarrow \infty} \int_r^t \int_{\R^m} f_s^{\eps_k} \cdot \nabla \psi_R ( g \mu^{\eps_k}_s + \mu^{\eps_k}_s \log \mu^{\eps_k}_s) dx ds, \\
    &= \int_r^t \int_{\R^m} f_s \cdot \nabla \psi_R ( g \mu_s + \mu_s \log \mu_s) dx ds.
\end{align*}
Since $g$ is a smooth function such that $\nabla g$ has linear growth, we have $g(x) \leq C(1 + |x|^2)$. By the bound \eqref{eq:ap1:moments}, we have $g \mu \in L^1(\R_T^m)$. Next, we control $f_s \cdot \nabla \psi_R$ for $R > 1$. Since $\psi \in C_0^\infty(\R^m)$, for all $j \in \N$ we have that $\nabla^j \psi_R$ remains bounded and supported on $\{|x| \leq R\}$. Moreover, we have by \eqref{eq:ap1:psiR} and the definition of $\psi$ that
\begin{equation}
    \sup_{x \in \R^m}|\nabla^j \psi_R(x)| \leq C_{ \psi, j}R^{-j}.
    \label{eq:ap1:gradPsi}
\end{equation}
for some $C_{\psi, j} > 0$ independent of $R$. For all $|x| \geq R$, we have $f_s \cdot \nabla \psi_R = 0$ since $\nabla \psi_R$ is supported on $|x| < R$. If $|x| < R$, we have by \eqref{eq:ap1:linGrowthCondition} and \eqref{eq:ap1:gradPsi} that
\begin{equation}
    |f_s(x) \cdot \nabla \psi_R(x)| \leq C_f(1 + |x|) C_{\psi, 1} R^{-1} \leq C_{b}C_{\psi, 1}(1 + R)R^{-1}  \leq 2C_fC_{\psi, 1}.
        \label{eq:ap1:bPsiBound}
\end{equation}
Then by the above discussion, \eqref{eq:ap1:bPsiBound} and \eqref{eq:ap1:stupidEntropy2}, we can apply the dominated convergence theorem to obtain
\begin{equation*}
    \lim_{R \rightarrow\infty}\lim_{k \rightarrow \infty} I_{2, 2, 2}(\eps_k, R) = \lim_{R \rightarrow\infty}\int_r^t \int_{\R^m} f_s \cdot \nabla \psi_R ( g \mu_s + \mu_s \log \mu_s) dx ds = 0,
\end{equation*}
which proves \eqref{eq:ap1:lastthing} and hence, concludes the proof.
\end{proof}

We also include the following \emph{superposition principle} from \cite{trevisan2016superposition} which allows us to go between analytic and probabilistic representations of Fokker-Planck equations.

\begin{proposition}[Superposition principle] \label{prop:ap1:superposition}
Suppose $f:[0, T] \times \R^m \rightarrow \R^m$ satisfies a linear growth condition. Suppose $\mu_0 \in \mP(\R^m)$ satisfies
\begin{equation}
\label{eq:ap1:finiteFirstMoment}
    \int_{\R^m} |x|^2 \mu_0(dx) < \infty.
\end{equation}
Then, the trajectory $\{\mu_t\}_{t \in [0, T]}$ is a weak solution to \eqref{eq:ap1:formalPDE} in the sense of Definition \ref{def:ap1:wp} if and only if there is a weak solution $\tilde{\mu} \in \mP(C[0, T]; \R^m))$ to the following SDE
\begin{equation}
\begin{aligned}
    dX(t) &= - f\big(t, X(t)\big) dt + \sqrt{2}dW_t,
    X(0) &\sim \mu_0
\end{aligned}
\label{eq:ap1:SDE}
\end{equation}
such that $\mu_t = \tilde{\mu}_t$ for all $t \in [0, T]$.
\end{proposition}
\begin{proof}
    First, suppose $\{\mu_t\}_{t \in [0, T]}$ is a weak solution to \eqref{eq:ap1:formalPDE} in the sense of of Definition \ref{def:ap1:wp}. By Example 7.1.3 of \cite{fokkerPlanck}, the condition \eqref{eq:ap1:finiteFirstMoment} implies that \begin{equation}
        \sup_{t \in [0, T]} \int_{\R^m}|x|^2\mu_t(dx) < \infty.
    \end{equation}
    Then, (2.3) of \cite{trevisan2016superposition} is satisfied and $\{\mu_t\}_{t \in [0, T]}$ can be modified to a narrowly continuous weak solution to the Fokker-Planck equation in the sense of Definition 2.2 of \cite{trevisan2016superposition}. Therefore, by Theorem \cite{trevisan2016superposition}, there exists a solution to the associated martingale problem. By Proposition 5.4.11 of \cite{karatzas1991stochastic}, there exists a weak solution to \eqref{eq:ap1:SDE}. 

    Now, suppose $\tilde{\mu} \in \mP(C[0, T]; \R^m))$ is a weak solution to \eqref{eq:ap1:SDE}. By Proposition 5.4.11 of \cite{karatzas1991stochastic}, there exists a solution to the martingale problem in the sense of Definition 2.4 of \cite{trevisan2016superposition}, where (2.6) is verified by \eqref{eq:wp:finiteMoments} and (3.18) of Problem 5.3.15 in \cite{karatzas1991stochastic}. By the discussion above Theorem 2.5 of \cite{trevisan2016superposition}, $\{\tilde{\mu}_t\}_{t \in [0, T]}$ is a solution to \eqref{eq:ap1:formalPDE} in the sense of Definition \ref{def:ap1:wp}.
\end{proof}

\bibliographystyle{amsplain}
\bibliography{refs}

\end{document}